\documentclass[11pt,a4paper]{article}

\usepackage[hmargin=25mm, vmargin=30mm]{geometry}

\usepackage[full]{textcomp}
\usepackage{newtxtext}

\usepackage{colortbl} 
\usepackage{amsmath}
\usepackage{mhequ}
\usepackage{booktabs}
\usepackage{centernot}
\usepackage{amsthm}
\usepackage{amssymb}
\usepackage{mathrsfs}
\usepackage{enumerate}
\usepackage{marginnote}
\usepackage{scalerel}
\usepackage{tikz}
\usepackage{tikz-cd}
\usepackage{graphicx}
\usepackage{caption}
\usepackage{subcaption}
\usepackage{pdfpages}

\numberwithin{equation}{section}

\newtheorem{thm}{Theorem}[section]

\newtheorem*{thm*}{Theorem}
\newtheorem*{qtn*}{Question}
\newtheorem*{prop*}{Proposition}
\newtheorem*{lem*}{Lemma}
\newtheorem*{rem*}{Remark}

\newtheorem{lem}[thm]{Lemma}
\newtheorem{prop}[thm]{Proposition}
\newtheorem{cor}[thm]{Corollary}
\newtheorem{defn}[thm]{Definition}
\newtheorem{rem}[thm]{Remark}

\newcommand\N{{\mathbb N}}
\newcommand\R{{\mathbb R}}

\newcommand\Z{{\mathbb Z}}

\newcommand{\CC}{\mathbb{C}}
\newcommand{\EE}{\mathbb{E}}

\newcommand{\LL}{\mathbb{L}}
\newcommand{\NN}{\mathbb{N}}

\newcommand{\RR}{\mathbb{R}}

\newcommand{\ZZ}{\mathbb{Z}}

\newcommand{\cC}{\mathcal C}

\newcommand{\cE}{\mathcal E}
\newcommand{\cF}{\mathcal F}

\newcommand{\cL}{\mathcal L}

\newcommand{\cT}{\mathcal T}

\newcommand{\sL}{\mathscr{L}}

\newcommand{\sT}{\mathscr{T}}
\newcommand{\sR}{\mathscr{R}}

\makeatletter
\newsavebox{\@brx}
\newcommand{\llangle}[1][]{\savebox{\@brx}{\(\m@th{#1\langle}\)}%
  \mathopen{\copy\@brx\kern-0.5\wd\@brx\usebox{\@brx}}}
\newcommand{\rrangle}[1][]{\savebox{\@brx}{\(\m@th{#1\rangle}\)}%
  \mathclose{\copy\@brx\kern-0.5\wd\@brx\usebox{\@brx}}}
\makeatother

\colorlet{symbols}{black}
\colorlet{testcolor}{green!60!black}

\def\1{\mathbf{{1}}}

\usetikzlibrary{shapes.misc}
\usetikzlibrary{shapes.symbols}
\usetikzlibrary{snakes}
\usetikzlibrary{decorations}
\usetikzlibrary{decorations.markings}
\definecolor{dblue}{rgb}{0.1, 0.1, 0.9}

\tikzset{
	root/.style={circle,fill=testcolor,inner sep=0pt, minimum size=2mm},		
	dot/.style={circle,fill=black,draw=black, solid,inner sep=0pt,minimum size=0.75mm},
	bdot/.style={circle,fill=blue,draw=dblue, solid,inner sep=0pt,minimum size=0.75mm},
		}
\colorlet{symbols}{blue!90!black}

\makeatletter
\def\DeclareSymbol#1#2#3{\expandafter\gdef\csname MH@symb@#1\endcsname{\tikz[baseline=#2,scale=0.15]{#3}}%
\expandafter\gdef\csname MH@symb@#1s\endcsname{\scalebox{0.6}{\tikz[baseline=#2,scale=0.15]{#3}}}}
\def\<#1>{\csname MH@symb@#1\endcsname}
\makeatother

\DeclareSymbol{sunset}{-3}{\draw (0,0) circle [x radius = 1.5, y radius = 1]; \draw (-1.5,0) node[dot] {} -- (1.5,0) node[dot] {};}
\DeclareSymbol{tadpole}{-3}{\draw (0,0) circle [x radius = 0.75, y radius = 1]; \draw (0,-1) node[dot] {};}

\DeclareSymbol{1}{0}{\draw[white] (-.4,0) -- (.4,0); \draw (0,0)  -- (0,1.2) node[dot] {};}
\DeclareSymbol{2}{0}{\draw (-0.5,1.2) node[dot] {} -- (0,0) -- (0.5,1.2) node[dot] {};}
\DeclareSymbol{3}{0}{\draw (0,0) -- (0,1.2) node[dot] {}; \draw (-.7,1) node[dot] {} -- (0,0) -- (.7,1) node[dot] {};}
\DeclareSymbol{4}{0}{\draw (-0.36,1.2) node[dot] {} -- (0,0) -- (0.36,1.2) node[dot] {}; \draw (-1,1) node[dot] {} -- (0,0) -- (1,1) node[dot] {};}
\DeclareSymbol{30}{-3}{\draw (0,0) -- (0,-1); \draw (0,0) -- (0,1.2) node[dot] {}; \draw (-.7,1) node[dot] {} -- (0,0) -- (.7,1) node[dot] {};}
\DeclareSymbol{31}{-3}{\draw (0,0) -- (0,-1) -- (1,0) node[dot] {}; \draw (0,0) -- (0,1.2) node[dot] {}; \draw (-.7,1) node[dot] {} -- (0,0) -- (.7,1) node[dot] {};}
\DeclareSymbol{32}{-3}{\draw (0,0) -- (0,-1) -- (1,0) node[dot] {}; \draw (0,0) -- (0,-1) -- (-1,0) node[dot] {}; \draw (0,0) -- (0,1.2) node[dot] {}; \draw (-.7,1) node[dot] {} -- (0,0) -- (.7,1) node[dot] {};}
\DeclareSymbol{20}{-3}{\draw (0,0) -- (0,-1);\draw (-.7,1) node[dot] {} -- (0,0) -- (.7,1) node[dot] {};}
\DeclareSymbol{22}{-3}{\draw (0,0.3) -- (0,-1) -- (1,0) node[dot] {}; \draw (0,0.3) -- (0,-1) -- (-1,0) node[dot] {};\draw (-.7,1) node[dot] {} -- (0,0.3) -- (.7,1) node[dot] {};}
\DeclareSymbol{202}{-3}{\draw (-.6625,0) -- (0,-1) -- (.6625,0)  {}; \draw (-1.1,1) node[dot] {} -- (-.6625,0) -- (-0.365,1) node[dot] {}; \draw (0.365,1) node[dot] {} -- (.6625,0) -- (1.1,1) node[dot] {};}

\DeclareSymbol{1b}{0}{\draw[white] (-.4,0) -- (.4,0); \draw[color=dblue] (0,0)  -- (0,1.2) node[bdot] {};}
\DeclareSymbol{2b}{0}{\draw[color=dblue] (-0.5,1.2) node[bdot] {} -- (0,0) -- (0.5,1.2) node[bdot] {};}
\DeclareSymbol{3b}{0}{\draw[color=dblue] (0,0) -- (0,1.2) node[bdot] {}; \draw[color=dblue] (-.7,1) node[bdot] {} -- (0,0) -- (.7,1) node[bdot] {};}
\DeclareSymbol{30b}{-3}{\draw[color=dblue] (0,0) -- (0,-1); \draw[color=dblue] (0,0) -- (0,1.2) node[bdot] {}; \draw[color=dblue] (-.7,1) node[bdot] {} -- (0,0) -- (.7,1) node[bdot] {};}
\DeclareSymbol{31b}{-3}{\draw[color=dblue] (0,0) -- (0,-1) -- (1,0) node[bdot] {}; \draw[color=dblue] (0,0) -- (0,1.2) node[bdot] {}; \draw[color=dblue] (-.7,1) node[bdot] {} -- (0,0) -- (.7,1) node[bdot] {};}
\DeclareSymbol{32b}{-3}{\draw[color=dblue] (0,0) -- (0,-1) -- (1,0) node[bdot] {}; \draw[color=dblue] (0,0) -- (0,-1) -- (-1,0) node[bdot] {}; \draw[color=dblue] (0,0) -- (0,1.2) node[bdot] {}; \draw[color=dblue] (-.7,1) node[bdot] {} -- (0,0) -- (.7,1) node[bdot] {};}
\DeclareSymbol{20b}{-3}{\draw[color=dblue] (0,0) -- (0,-1);\draw[color=dblue] (-.7,1) node[bdot] {} -- (0,0) -- (.7,1) node[bdot] {};}
\DeclareSymbol{22b}{-3}{\draw[color=dblue] (0,0.3) -- (0,-1) -- (1,0) node[bdot] {}; \draw[color=dblue] (0,0.3) -- (0,-1) -- (-1,0) node[bdot] {};\draw[color=dblue] (-.7,1) node[bdot] {} -- (0,0.3) -- (.7,1) node[bdot] {};}

\DeclareSymbol{box}{0.65}{\draw[semithick] (-0.7,0) -- (0.7,0) --  (0.7,1.4) -- (-0.7,1.4) -- (-0.7,0);}
\DeclareSymbol{sbox}{0.5}{\draw[semithick] (-0.5,0) -- (0.5,0) -- (0.5,1) -- (-0.5,1) -- (-0.5,0);}
\DeclareSymbol{blackbox}{0.65}{\draw[semithick, fill=black] (-0.7,0) -- (0.7,0) --  (0.7,1.4) -- (-0.7,1.4) -- (-0.7,0);}
\DeclareSymbol{boxinbox}{0.65}{\draw[semithick, fill=black] (-0.7,0) -- (0.7,0) --  (0.7,1.4) -- (-0.7,1.4) -- (-0.7,0); \draw[semithick, fill=white] (-0.7,0) -- (0,0) -- (0,0.7) -- (-0.7,0.7) -- (-0.7,0) ; }
\DeclareSymbol{sboxinbox}{0.5}{\draw[semithick, fill=black] (-0.5,0) -- (0.5,0) -- (0.5,1) -- (-0.5,1) -- (-0.5,0);\draw[semithick, fill=white] (-0.5,0) -- (0,0) -- (0,0.5) -- (-0.5,0.5) -- (-0.5,0) ; }

\let\Re=\undefined\DeclareMathOperator*{\Re}{Re}
\let\Im=\undefined\DeclareMathOperator*{\Im}{Im}

\let\P= \undefined
\newcommand{\P}{\mathbf{P}}
\newcommand{\E}{\mathbb{E}}

\newcommand{\eps}{\varepsilon}

\newcommand{\V}{\mathcal{V}}

\newcommand{\U}{\mathcal{U}}

\usepackage{scrextend}
\usepackage{url}

\newcommand\sm{\partial\mathbf{m}}

\newcommand{\Bot}{{\rm B}}
\newcommand{\Top}{{\rm T}}
\newcommand{\Left}{{\rm L}}
\newcommand{\Right}{{\rm R}}
\usepackage{wasysym}

\newcommand{\conn}{\longleftrightarrow}
\newcommand{\connstar}{\overset{*}{\longleftrightarrow}}

\newcommand{\ising}{{\rm Ising}}
\newcommand{\bbE}{\mathbb{E}}
\newcommand{\Y}{{\bf Y}}

\renewcommand{\Re}{\operatorname{Re}}
\renewcommand{\Im}{\operatorname{Im}}

\definecolor{darkred}{rgb}{0.9,0.1,0.1}
\definecolor{darkergreen}{rgb}{0.0, 0.5, 0.0}

\DeclareMathAlphabet{\pazocal}{OMS}{zplm}{m}{n}

\newcommand{\crit}{{\text{tric}}}
\newcommand{\Sym}{\mathrm{Sym}}

\begin{document}

\title{ Existence of a tricritical point for the Blume-Capel model on $\ZZ^d$}

\author{Trishen S. Gunaratnam\footnotemark[1]\footnote{Universit\'e de Gen\`eve,~trishen.gunaratnam@unige.ch}, \, Dmitrii Krachun\footnotemark[2]\footnote{Princeton University,~dk9781@princeton.edu}, \, and\, Christoforos Panagiotis\footnotemark[3]\footnote{University of Bath,~cp2324@bath.ac.uk}}

\maketitle

\begin{abstract}

We prove the existence of a tricritical point for the Blume-Capel model on $\mathbb{Z}^d$ for every $d\geq 2$. The proof in $d\geq 3$ relies on a novel combinatorial mapping to an Ising model on a larger graph, the techniques of Aizenman, Duminil-Copin, and Sidoravicious (Comm. Math. Phys, 2015), and the celebrated infrared bound. In $d=2$, the proof relies on a quantitative analysis of crossing probabilities of the dilute random cluster representation of the Blume-Capel. In particular, we develop a quadrichotomy result in the spirit of Duminil-Copin and Tassion (Moscow Math. J., 2020), which allows us to obtain a fine picture of the phase diagram in $d=2$, including asymptotic behaviour of correlations in all regions. Finally, we show that the techniques used to establish subcritical sharpness for the dilute random cluster model extend to any $d\geq 2$.
\end{abstract}

\section{Introduction}

Let $\LL^d = (\ZZ^d, \E^d)$ be the standard $d$-dimensional hypercubic lattice with vertex set $\ZZ^d$ and nearest-neighbour edges $\E^d$. Given a finite subgraph $G=(V,E)$ of $\LL^d$, the Blume-Capel model on $G$ with inverse temperature $\beta>0$, crystal field strength $\Delta \in \RR$, and boundary condition $\eta \in \{-1,0,1\}^{\ZZ^d}$ is the probability measure $\mu_{G,\beta,\Delta}^\eta$ defined on spin configurations $\sigma \in \{-1,0,1\}^V$ by  
\begin{equs}
\mu_{G,\beta,\Delta}^\eta(\sigma)
=
\frac{1}{Z_{G,\beta,\Delta}^\eta} e^{-H_{G,\beta,\Delta}^\eta(\sigma)},
\end{equs}
where
\begin{equs}
H^\eta_{G,\beta,\Delta}(\sigma)
=
-\beta\sum_{xy \in E} \sigma_x\sigma_y -\Delta\sum_{x \in V} \sigma_x^2 - \beta \sum_{\substack{xy \in \E^d \\ x \in V,\, y \in \ZZ^d\setminus V }} \sigma_x \eta_y
\end{equs} 
is the Hamiltonian, and $Z_{G,\beta,\Delta}^\eta$ is the partition function. By convention, we write $\mu^{\eta}_{V,\beta,\Delta}$ to denote the Blume-Capel measure on the subgraph of $\LL^d$ spanned by $V$, and we denote expectations by $\langle \cdot \rangle^{\eta}_{V,\beta,\Delta}$. We denote by $0$ (resp. $+$, $-$) the boundary conditions $\eta_x=0$ (resp. $\eta_x=+1$, $\eta_x=-1$) for all $x\in \ZZ^d$. 

The Blume-Capel model is closely related to one of the most famous models in statistical physics, the Ising model. The latter is defined analogously, with spins taking values in $\{\pm 1\}$ and the Hamiltonian only consisting of the terms involving $\beta$. In particular, one can view the Blume-Capel model as an Ising model on an annealed random environment, i.e.\ on the (random) set of vertices for which $\sigma_x \neq 0$. In the limit $\Delta\rightarrow \infty$, the underlying random environment becomes deterministically equal to $\ZZ^d$, and the Blume-Capel model converges to the classical Ising model on $\ZZ^d$. 

The model was introduced independently by Blume \cite{BL} and Capel \cite{Capel} in 1966 to study the magnetisation of uranium oxide and an Ising system consisting of triplet ions, respectively. Both papers were trying to explain first-order phase transitions that are driven by mechanisms other than external magnetic fields. Since then, it has been studied extensively by physicists due to its particularly rich phase diagram, i.e.\ as an archetypical example of a model that exhibits a multicritical point, see \cite{ButeraPernici} and references therein. Indeed, for each value of the parameter $\Delta\in \RR$, the Blume-Capel model undergoes a phase transition at a critical\footnote{In the physics literature, $\beta_c$ is sometimes called a transition point or triple point when the phase transition is discontinuous, to distinguish from a critical point corresponding to continuous phase transition. In this article we adopt percolation terminology and call $\beta_c$ a critical point.} parameter $\beta_c(\Delta)$, which is defined as 
\begin{equs}
\beta_c(\Delta)=\inf\{\beta>0 \mid \langle \sigma_0 \rangle^+_{\beta,\Delta}>0\},
\end{equs}
where $\langle \cdot \rangle^+_{\beta,\Delta}$ denotes the plus measure at infinite volume, which is defined as the limit of the finite volume measures $\langle \cdot \rangle^+_{G,\beta,\Delta}$ as $G$ tends to $\LL^d$. It is expected that the critical behaviour of the model depends strongly on $\Delta$: as $\Delta$ increases, the phase transition changes from discontinuous, when $\langle \sigma_0 \rangle^+_{\beta_c(\Delta),\Delta}>0$, to continuous, when $\langle \sigma_0 \rangle^+_{\beta_c(\Delta),\Delta}=0$, and this transition happens as $\Delta$ crosses a single point $\Delta_{\crit}$. See Figure \ref{fig: critical line}.

The model at the so-called tricritical point $(\Delta_\crit, \beta_c(\Delta_{\crit}))$ is of particular interest. Indeed, in some dimensions it is expected to exhibit vastly different behaviour from the points on the critical line when $\Delta>\Delta_{\crit}$, even though in both cases the phase transition is expected to be continuous. In particular, in $d=2$, the scaling limit of the model at criticality for $\Delta > \Delta_{\crit}$ is expected to be in the Ising and $\phi^4$ universality class. On the other hand, at $\Delta_\crit$, the scaling limit of the model is expected to be in a distinct universality class corresponding to the $\phi^6$ minimal conformal field theory with central charge $c=7/10$ (whilst the Ising universality class is of central charge $c=1/2$) -- see Mussardo \cite{Mussardo}. A rigorous glimpse of this distinct universality class appears in the work of Shen and Weber \cite{ShenWeber}, where near-critical scaling limits of related models with so-called Kac interactions are considered. On the other hand, in $d=3$, the scaling limit of the Blume-Capel model at $\Delta > \Delta_{\crit}$ is expected to be nontrivial, as conjecturally for Ising, which is supported by e.g.\ conformal bootstrap methods -- see Rychkov, Simmons-Duffin, and Zan \cite{3D-Ising-Slava}. Whereas for $\Delta_\crit$, $d=3$ is predicted to be the upper-critical dimension for the model and one expects triviality of the scaling limit. This is supported by renormalisation group heuristics, which have recently been made rigorous for the $\phi^6$ model in the weak coupling regime by Bauerschmidt, Lohmann, and Slade \cite{BLS20}. See also \cite{BS20} for a mean-field random
walk model exhibiting a tricritical behaviour. In dimensions $d \geq 5$, one expects that the model is trivial throughout the continuous phase transition regime $\Delta \geq \Delta_\crit$ -- c.f. the case of Ising, where triviality was shown by Aizenman \cite{A82} and Fr\"ohlich \cite{F82}. We also refer to the review book \cite{BBS19} for an account of renormalisation group approaches to this problem. In $d=4$, for $\Delta=\Delta_\crit$ one also expects a triviality result, whereas for $\Delta > \Delta_{\crit}$ one expects a marginal triviality result as in the case for Ising, which was recently shown by Aizenman and Duminil-Copin \cite{4D-Ising-triviality}.

Despite the interest of this model in physics and the interesting predictions about the tricritical point, there is a lack of rigorous understanding of the phase diagram of the Blume-Capel model. Indeed, many rigorous results about the model have been focused well within the discontinuous transition regime, where it is a good test case for Pirogov-Sinai theory  -- see \cite[Chapter 7]{FV} and \cite{BS}. This is in contrast to the case of the Ising model, where much of the phase diagram is now well-understood \cite{HDC-IP}. Stochastic geometric methods have been at the heart of many recent developments, in particular probabilistic representations of spin correlations via the random cluster and random current representations. For the Blume-Capel model, an analogous random cluster representation, called the dilute random cluster representation, has been developed by Graham and Grimmett \cite{GG}. In this article, we show that the underlying philosophy of the recent techniques used to analyse the Ising model and its random cluster representation can be adapted to rigorously analyse the phase diagram of the Blume-Capel model in dimensions $d\geq 2$. 

Our first result establishes fundamental properties of the phase diagram of the model that has been up until now folklore.

\begin{thm} \label{thm: BC phase diagram}
Let $d \geq 2$. For every $\Delta \in \mathbb{R}$, we have that $\beta_c(\Delta) \in (0,\infty)$. In other words, there is a nontrivial phase transition. Moreover, the function $\Delta \mapsto \beta_c(\Delta)$ is continuous and decreasing with limits $\lim_{\Delta\rightarrow-\infty}\beta_c(\Delta)=+\infty$ and $\lim_{\Delta\rightarrow +\infty} = \beta_c({\rm Ising})$, where $\beta_c({\rm Ising})$ is the critical point of the Ising model on $\mathbb{Z}^d$. 
\end{thm}

The next theorem is the main result of this paper. It establishes the existence of at least one separation point between the points of continuous and discontinuous phase transitions, i.e.\ the existence of a tricritical point. Namely, we prove the following.

\begin{thm}\label{thm: existence} 
Let $d \geq 2$. Then, there exist $\Delta^-(d) \leq \Delta^+(d)$ such that
\begin{itemize}
\item for any $\Delta < \Delta^-(d)$, $\langle \sigma_0 \rangle^+_{\beta_c(\Delta),\Delta} > 0$ 
\item for any $\Delta \geq  \Delta^+(d)$, $\langle \sigma_0 \rangle_{\beta_c(\Delta),\Delta}^+ = 0$.
\end{itemize}
Moreover, one can take $\Delta^+(d):=-\log{2}$ for $d\geq 3$, and $\Delta^+(d):=-\log{4}$ for $d=2$.
\end{thm}

The proof of the existence of a discontinuous critical phase is a standard application of Pirogov-Sinai theory and is given in Section \ref{sec:basic-facts-phase-transition}. The crux of the article is to establish the existence of a continuous critical phase, for which we use different techniques depending on the dimension. In the case $d\geq 3$, the proof relies on a new representation of the model as an Ising model on a larger (deterministic) graph, which is ferromagnetic only when $\Delta\geq -\log{2}$ -- see Section~\ref{sec: mapping bc-ising}. This representation has the advantage that it naturally relates the correlation functions between the Blume-Capel and the Ising model, and allows us to use techniques and tools developed for the latter to study the former. We use the celebrated infrared/Gaussian domination bound of Fr\"ohlich, Simon, and Spencer \cite{FSS} to show that under the free boundary conditions, the two-point correlations decay, in an averaged sense, to $0$ for every $\beta\leq \beta_c(\Delta)$. This allows us to use the breakthrough result of Aizenmann, Duminil-Copin and Sidoravicius \cite{ADCS} on the continuity of the Ising model on $\ZZ^d$ (see also \cite{AR} for a relevant result for general amenable graphs) to conclude that the phase transition of the Blume-Capel model is continuous in $d\geq 3$ when the corresponding Ising model is ferromagnetic, namely, when $\Delta\geq -\log{2}$. 

In dimension $2$, the infrared bound is not sufficient to deduce that the two-point correlations under free boundary conditions decay. Instead, we prove this by using a representation of the correlation function in terms of connection events in the dilute random cluster model. It is well-known that in two-dimensional percolation theory the relevant observables that keep track of continuous versus discontinuous phase transitions are crossing probabilities of macroscopic rectangles. We analyse crossing probabilities in the dilute random cluster model using the renormalisation strategy of \cite{DCT}. As output, we obtain that the two-point correlation functions under free boundary conditions (associated to the Blume-Capel model) decay to $0$ for every $\beta \leq \beta_c(\Delta)$. In turn, by the same arguments discussed in the preceding paragraph, we obtain that in dimension $d= 2$ the phase transition is continuous when $\Delta\geq -\log 2$. The estimate on $\Delta^+$ can be improved in $d=2$ to obtain continuity for $-\log 4\leq \Delta< -\log 2$, where the corresponding Ising model is not ferromagnetic. Instead we can directly analyse the difference between the spin models with free and plus boundary conditions by interpolating between them using arbitrarily small positive boundary conditions (that we call weak plus). The comparison between the free and the weak plus boundary conditions is enabled by our quantitative analysis of crossing probabilities, whereas the comparison between weak plus and plus boundary conditions is enabled by a Lee-Yang type theorem on the complex zeros of the Blume-Capel partition function.

\begin{rem}
There does not seem to be an explicit formula for the tricritical point in any dimension. In $d=2$ numerics \cite{MFH24} suggest that in our parametrisation it is located approximately at $\Delta \approx -3.23$ and $\beta \approx 1.64$. Our bound of $\Delta^+ = -\log 4 \approx -1.39$ is therefore not sharp. In dimensions $d\geq 3$ we do not expect our current bound $\Delta^+ = -\log 2$ to be sharp. However, we note that $\Delta = -\log 4 $ and $\beta = 1/N$ correspond to the location of the tricritical point for the Blume-Capel model on the complete graph on $N$ vertices in the limit $N\rightarrow \infty$, see \cite{EllisOttoTouchette, ShenWeber}. It would be interesting to see whether the improved threshold $\Delta^+(2) = -\log 4$ can be extended to any dimension. Indeed, we expect that the tricritical point on $\ZZ^d$ converges to that of the complete graph as $d \rightarrow \infty$, i.e.\ in the mean-field limit. In support of this, we point out that the Lee-Yang type theorem that we establish holds for $\Delta \geq -\log 4$ on any graph. 
\end{rem}

Proving uniqueness of the tricritical point, which is the main omission of Theorem \ref{thm: existence}, amounts to showing that one can take $\Delta^-(d)=\Delta^+(d)$ in Theorem \ref{thm: existence}. Unfortunately, it is unclear whether there is monotonicity along the critical line and, in the absence of integrability, to the best of our knowledge all known techniques for showing discontinuity are intrinsically perturbative. Nevertheless, in dimension $2$, the quantitative estimates on crossing probabilities, and other considerations that we describe shortly, allow us to obtain a more precise picture of the phase diagram, although we fall short of proving uniqueness of the tricritical point.

To state our next result, we first recall the definition of the truncated $2$-point correlation: 
\begin{equs}
\langle \sigma_0 ; \sigma_x \rangle^+ = \langle \sigma_0  \sigma_x \rangle^+ - \langle \sigma_0 \rangle^+ \langle \sigma_x \rangle^+.
\end{equs}
In dimension $d=2$ we obtain a fine picture of the phase diagram by showing that the truncated $2$-point correlation decays exponentially everywhere except for the continuous critical regime, where it decays polynomially. We also give an alternative characterisation of the points of continuity in terms of the percolation of $0$ and $-$ spins (equivalently, $\ast$-percolation properties of $+$ spins, where $\ast$-percolate refers to percolation on $\ZZ^2$ union the diagonals).
\begin{thm}\label{thm: truncated}
Let $d=2$. Then the following hold.
\begin{labeling}{{\bf (DiscontCrit)}}
\item[{\bf (OffCrit)}] For all $\beta > 0$ and $\Delta \in \RR$ such that $\beta \neq \beta_c(\Delta)$, there exists $c=c(\beta,\Delta)>0$ such that
\begin{equs}
\hspace{-1.4em}\langle  \sigma_0 ; \sigma_x \rangle^+_{\beta,\Delta}\leq e^{-c\lVert x \rVert_{\infty}}, \qquad \forall x \in \ZZ^2.
\end{equs}
\item[{\bf (Discont)}] For all $\Delta \in \RR$ such that $\langle \sigma_0 \rangle^+_{\beta_c(\Delta),\Delta} > 0$, there exists $c=c(\Delta)>0$ such that
\begin{equs}
\langle  \sigma_0 ; \sigma_x \rangle^+_{\beta_c(\Delta),\Delta}\leq e^{-c\lVert x \rVert_{\infty}}, \qquad \forall x \in \ZZ^2.
\end{equs}
\item[{\bf (Cont)}] For all $\Delta \in \RR$ such that $\langle \sigma_0 \rangle^+_{\beta_c(\Delta),\Delta} = 0$, there exist $c_1=c_1(\Delta),c_2=c_2(\Delta)$, such that for all $x\in \ZZ^2$ with $\lVert x \rVert_{\infty}$ large enough
\begin{equs}
\hspace{-2.6em}\frac{1}{\lVert x \rVert_{\infty}^{c_1}}
\leq 
\langle  \sigma_0 \sigma_x \rangle^+_{\beta_c(\Delta),\Delta}
\leq 
\frac{1}{\lVert x \rVert_{\infty}^{c_2}}.
\end{equs}
\item[{\bf (TriCrit)}] The set of $\Delta\in \RR$ such that $\langle \sigma_0 \rangle^+_{\beta_c(\Delta),\Delta}>0$ is open.
\item[{\bf (Perc)}] For all $\Delta\in \RR$, $\langle \sigma_0 \rangle^+_{\beta_c(\Delta),\Delta}=0$ if and only if there is no infinite site cluster of $\{0,-1\}$ spins under $\langle \cdot \rangle^0_{\beta_c(\Delta),\Delta}$.
\end{labeling}
\end{thm}

\begin{figure}
    \centering
    \includegraphics[width=0.7\textwidth]{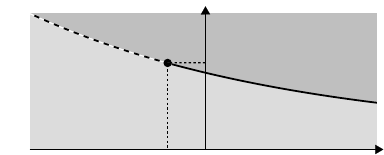}
    \put(2,0){$\Delta$}
    \put(-145,125){$\beta$}
    \put(-145,73){$\beta_c(\Delta_{\rm tric})$}
    \put(-183,-8){$\Delta_{\rm tric}$}
    \caption{Conjectural  phase diagram of the Blume-Capel model. The dark grey region is the supercritical phase $\beta > \beta_c(\Delta)$, where the spontaneous magnetisation does not vanish, i.e.\ $\langle \sigma_0 \rangle^+_{\beta(\Delta),\Delta}>0$. The light grey region is the subcritical phase $\beta < \beta_c(\Delta)$, where the spontaneous magnetisation vanishes, i.e.\ $\langle \sigma_0 \rangle^+_{\beta(\Delta),\Delta}=0$. The union of the dashed and solid lines, together with the black dot, indicate the critical curve $\Delta \mapsto \beta_c(\Delta)$. Note that it is a continuous decreasing curve (the convexity of the curve is purely for illustration purposes and the actual curve may not be convex). The curve of points of continuous phase transition is depicted solid, and the curve of discontinuous phase transition appears as dashed. The tricritical point is the boundary point between these two behaviours along the critical curve and is marked as a solid black dot. Theorem \ref{thm: existence} rigorously establishes all parts of this phase diagram \emph{except} for the behaviour on the critical curve in the neighbourhood of the (unique) tricritical point. We establish thresholds $\Delta^- < \Delta^+$ such that the curve is dashed for $\Delta \leq \Delta^-$ and solid for $\Delta \geq \Delta^+$, but we cannot say what happens along the critical curve for $\Delta \in (\Delta^-,\Delta^+) $.}
    \label{fig: critical line}
\end{figure}
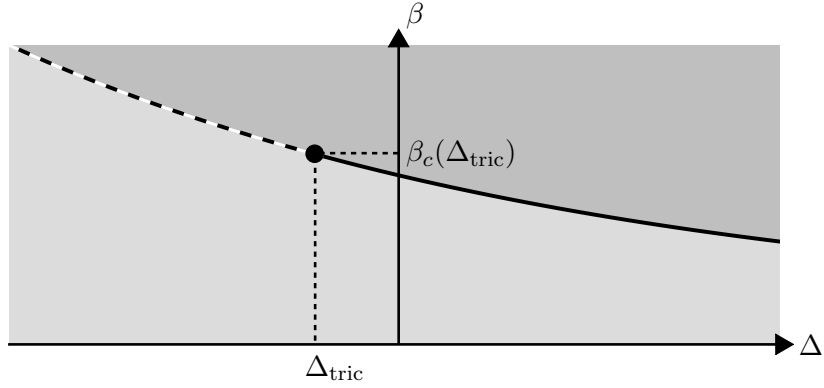

The proof of the behaviour at the critical points is based on the quadrichotomy for crossing probabilities for the dilute random cluster representation of the Blume-Capel model mentioned earlier. The fact that {\bf (TriCrit)} holds has the implication that at any separation point (i.e.\ tricritical point) on the line of critical points, the phase transition is continuous and hence satisfies {\bf (Cont)}. The percolation characterisation {\bf (Perc)} is an adaptation of an elegant geometric argument for the continuity of nearest-neighbour Ising on $\ZZ^2$, see \cite{W09}. The proof of the subcritical behaviour relies on a generalisation of the OSSS inequality for monotonic measures by Duminil-Copin, Raoufi and Tassion \cite{DCRT} to the dilute random cluster model, and more generally to \textit{weakly monotonic measures}, which we define in Section~\ref{sec: OSSS}. The original OSSS inequality was obtained by O'Donell et.\ al.\ \cite{OSSS} for product measures. In fact, the technique for showing subcritical sharpness is robust enough to extend to all dimensions:
\begin{thm}\label{thm: exp dec}
Let $d\geq 2$ and $\Delta\in \RR$. Then for every $\beta<\beta_c(\Delta)$ there exists $c=c(\beta,\Delta,d)>0$ such that 
\begin{equs}
\langle  \sigma_0  \sigma_x \rangle^+_{\beta,\Delta}\leq e^{-c\lVert x \rVert_{\infty}}
\end{equs}
for every $x\in \ZZ^d$.
\end{thm}

We end on a natural follow up question on Theorem \ref{thm: truncated}, which relates the percolative properties of the underlying random environment to the nature of the critical point.
\begin{qtn*}
In $d=2$, is it true that, for all $\Delta \in \RR$, $\langle \sigma_0 \rangle^+_{\beta_c(\Delta),\Delta}=0$ if and only if $\{0\}$ spins do not $\ast$-percolate under $\langle \cdot \rangle^0_{\beta_c(\Delta),\Delta}$? Is there a higher dimensional analogue of this?
\end{qtn*}

\subsection{Paper organisation}

In Section \ref{sec: drc} we define the dilute random cluster model and develop the necessary tools that we require for the rest of this article. In Section \ref{sec:basic-facts-phase-transition} we establish fundamental folklore facts about the phase transition for the Blume-Capel model, in particular that $\beta_c(\Delta)$ is continuous and decreasing (i.e.\ we prove Theorem \ref{thm: BC phase diagram}), and show the existence of a discontinuous critical phase (the first half of Theorem~\ref{thm: existence}). In Section \ref{sec: mapping bc-ising} we introduce a combinatorial mapping from the Blume-Capel model on $\ZZ^d$ to an Ising model on an enlarged graph, and we use it to prove the existence of a continuous critical phase  for $d\geq 3$ in Section \ref{sec: pf in 3d} (this corresponds to proving the second half of Theorem \ref{thm: existence} in $d\geq 3$). In Section \ref{sec: quantitative} we establish a quadrichotomy result for crossing probabilities under the dilute random cluster model in $d=2$ and derive some basic consequences of it. In Section \ref{sec: tric 2} we then apply the quadrichotomy result, together with Lee-Yang type arguments, to establish the second half of Theorem \ref{thm: existence} concerning the existence of a continuous critical phase for $d=2$. The rest of the paper concerns Theorem \ref{thm: truncated} and \ref{thm: exp dec}. In Section \ref{sec: subcrit} we prove Theorem \ref{thm: exp dec} via a generalisation of the OSSS argument to weakly monotonic measures. Finally, in Section \ref{sec: further} we use the results obtained in Sections~\ref{sec: quantitative} and \ref{sec: subcrit} to prove Theorem \ref{thm: truncated}.

\subsection*{Acknowledgements}

Above all we thank Hugo Duminil-Copin for encouraging this collaboration. We thank Hendrik Weber for introducing us to this model. We thank Amol Aggarwal, Roman Koteck\'y, Vieri Mastropietro, Romain Panis, S\'ebastien Ott, Tom Spencer and Yvan Velenik for exciting discussions at various stages of this project. This project was completed when all three authors were at the University of Geneva. TSG was supported by the Simons Foundation, Grant 898948, HDC. DK and CP were supported by the Swiss National Science Foundation and the NCCR SwissMAP.

\section{The dilute random cluster model} \label{sec: drc}

\subsection{Definitions}

We begin by defining the sample space. Let $\Psi=\{0,1\}^{\ZZ^d}$ and $\Omega=\{0,1\}^{\E^d}$. Note that elements of $\Psi$ correspond to site percolation configurations and elements of $\Omega$ correspond to bond percolation configurations. We say that $\psi=(\psi_x)\in \Psi$ and $\omega=(\omega_e)\in \Omega$ are compatible if $\omega_{xy}=0$ whenever $\psi_x=0$ or $\psi_y=0$. Denote by $\Theta \subset \Psi \times \Omega$ the set of all possible compatible configurations, equipped with the $\sigma$-algebra generated by cylinder events, $\cF$. We denote a generic element of $\Theta$ by $\theta$. Any such element has an associated site percolation and bond percolation corresponding to projections onto the first and second coordinates, respectively denoted by $\theta_\Psi$ and $\theta_\Omega$.

Let $\theta = (\psi,\omega) \in \Theta$ be a compatible configuration. A vertex $x\in V$ is called open in $\theta$ if it is open in $\psi$, i.e.\ if $\psi_x=1$, and it is called called closed in $\theta$ (equivalently $\psi$) otherwise. An edge $e$ is called open in $\theta$ if it is open in $\omega$, i.e.\ if $\omega_e=1$, and closed in $\theta$ (equivalently $\omega$) otherwise. Define $V_{\psi}=\{x\in \ZZ^d : \psi_x=1\}$ to be the set of open vertices in $\psi$ and $E_{\psi}=\{xy\in \E^d:
\psi_x=\psi_y=1\}$ to be the set of induced edges of $\psi$. In addition, define $o(\omega)$ to be the set of
open edges in $\omega$. Observe that, given $\psi \in \Psi$ and $\omega \in \Omega$, then $(\psi,\omega)\in\Theta$ if
and only if $o(\omega)\subseteq E_\psi$. When clear from context, we simply refer to vertices and edges as open or closed. 

For $\Lambda \subset \ZZ^d$ finite, let $\partial \Lambda = \{ x \in \ZZ^d\setminus \Lambda \text{ such that there exists } y \in \Lambda \text{ incident to } x\}$ denote the exterior boundary of $\Lambda$, and set $\overline{\Lambda} := \Lambda \cup \partial \Lambda$. In addition, write\footnote{This is not to be confused with the set of induced edges of the infinite volume configuration $\psi \in \Psi$.} $E_{\Lambda}$ for the set of edges with at least one endpoint in $\Lambda$, and $b(\Lambda)=\{xy\in \E^d : x,y\in \Lambda\}$ for the set of edges induced by $\Lambda$. We call the pair $(\Lambda, E_\Lambda)$ a \textit{region} and stress that it is not a graph, but $(\overline{\Lambda}, E_\Lambda)$ is a graph. Finally, given $\psi \in \Psi$, write $E_{\psi,\Lambda} = E_\psi \cap E_\Lambda$. 

Let $\xi=(\kappa,\rho)\in\Theta$ and $\Lambda \subset \ZZ^d$ finite. Define $\Theta_\Lambda^\xi$ to be the set of $\theta=(\psi,\omega) \in \Omega$ such that
$\psi_x=\kappa_x$ for $x\in \ZZ^d\setminus \Lambda$ and $\omega_e=\rho_e$ for $e\in \E^d\setminus E_{\Lambda}$. The set of elements in $\Theta_\Lambda^\xi$ correspond to configurations on $\Lambda$ with boundary condition $\xi$. 
Note that $\Theta_\Lambda^\xi$ inherits a sub $\sigma$-algebra of $\cF$ corresponding to configurations that are measurable with respect to the vertices and edges in the region $(\Lambda, E_\Lambda)$. Finally, we write $k(\theta,\Lambda)$ for the number of connected components of the graph $(V_\psi,o(\omega))$ that intersect $\overline{\Lambda}$. Here, $k(\theta,\Lambda)$ implicitly depends on $\xi$ through the condition that $\theta \in \Theta_\Lambda^\xi$. 

We now define the finite volume dilute random cluster model with parameters $a \in (0,1)$ and $p \in (0,1)$, and fixed boundary conditions.
\begin{defn}[Dilute random cluster model]
Let $p,a \in (0,1)$ and write $r = \sqrt{1-p}$. For $\Lambda \subset \ZZ^d$ finite and $\xi \in \Theta$, let $\varphi^{\xi}_{\Lambda,p,a}$ denote the probability measure on $\Theta_\Lambda^\xi$ defined for $\theta=(\psi,\xi) \in \Theta_\Lambda^\xi$ by
\begin{equs} \label{defRC}
\varphi^{\xi}_{\Lambda,p,a} [\theta] 
=
\frac{\1[\theta \in \Theta_\Lambda^\xi]}{Z^{\xi}_{\Lambda,p,a}}
\prod_{x\in \Lambda} \left(\frac{a}{1-a}\right)^{\psi_x}
\prod_{e\in E_{\psi,\Lambda}} \left(\frac{p}{1-p}\right)^{\omega_e}
 r^{|E_{\psi,\Lambda}|} 2^{k(\theta,\Lambda)}
\end{equs}
where $Z^{\xi}_{\Lambda,p,a}$ is the normalisation constant. 
\end{defn}
\noindent When the parameters are either clear from context or irrelevant, we simply write $\varphi^{\xi}_{\Lambda}=\varphi^{\xi}_{\Lambda,p,a}$.

Let us now define the usual random cluster model on the graph induced by $E_{\psi,\Lambda}$ with parameters $p\in (0,1)$ and $q=2$, and fixed boundary conditions.

\begin{defn}[Random cluster model]
Let $p\in (0,1)$. For $\Lambda \subset \ZZ^d$ finite, $\xi=(\kappa,\rho)\in \Theta$, and $\psi\in \Psi$ such that $\psi_x=\kappa_x$ for $x\in \ZZ^d\setminus \Lambda$, let
\begin{equs}\label{eq: rc def}
\phi^{\textup{RC},\xi}_{\Lambda,\psi,p}[\omega]
=
\frac{\1[\theta \in \Theta_\Lambda^\xi]}{Z^{\textup{RC},\xi}_{\Lambda,\psi,p}}2^{k(\theta,\Lambda)}
\prod_{e\in E_{\psi,\Lambda}} \left(\frac{p}{1-p}\right)^{\omega_e},
\end{equs}
where $Z^{\textup{RC},\xi}_{\Lambda,\psi,p}$ is the normalisation constant, and $\theta=(\psi,\omega)$.
\end{defn}

\begin{rem}
We can extend the definition of the dilute random cluster measure to include the cases $p,a \in \{0,1\}$. First, fix $a =0$. By convention, for any $p \in [0,1]$, we set: $\varphi^\xi_{\Lambda, p, 0} = \delta_{(0,0)}$, the Dirac mass on the configuration of only closed vertices and edges. Now, fix $a =1$. For any $p \in [0,1]$, we define $\varphi^\xi_{\Lambda,p,1}$ to be the usual  random cluster measure $\phi^{\textup{RC},\xi}_{\Lambda,\psi,p}$, where $\psi$ is equal to $1$ on $\Lambda$ and coincides with $\kappa$ on $\ZZ^d\setminus \Lambda$.  

On the other hand, for any $a \in (0,1]$, we set: $\varphi^\xi_{\Lambda,0,a}$ to be the tensor product of Bernoulli site percolation of parameter $a$ and the Dirac measure centred at $0 \in \Omega$; and, $\varphi^\xi_{\Lambda,1,a} = \delta_{(1,1)}$, the Dirac mass on the configuration of only open vertices and bonds. 
\end{rem}

On any finite subset of $\ZZ^d$, there are two natural candidates for extremal\footnote{In the sense of Proposition \ref{prop: domination}.} dilute random cluster measures. These measures play a central role in this paper, so we formalise them in the following definition. 
\begin{defn}\label{def: 0,1}
Let $p \in [0,1)$ and $a \in [0,1)$. For $\Lambda \subset \ZZ^d$ finite, we define
\begin{itemize}
    \item the free measure $\varphi^0_{\Lambda,p,a} := \varphi^{(0,0)}_{\Lambda,p,a}$
    \item and, the wired measure $\varphi^1_{\Lambda,p,a} := \varphi^{(1,1)}_{\Lambda,p,a}$
\end{itemize}
where we recall that $(0,0)$ is the configuration consisting of only closed vertices and edges, and $(1,1)$ is the configuration consisting of only open vertices and edges. 
\end{defn}

Finally, we end with two more definitions. The first introduces a probability measure on $\Psi$ which captures the law of the (dependent) site percolation/ vertex marginal induced by the dilute random cluster measure -- i.e.\ the pushforward measure under the projection $\theta \mapsto \theta_\Psi$. 
\begin{defn}
Let $p,a \in [0,1]$. For $\Lambda \subset \ZZ^d$ and $\xi = (\kappa,\rho) \in \Theta$, let $\Psi^\xi_{\Lambda,p,a}$ denote the probability measure on $\Psi$ defined, for $\psi \in \Psi$, by
\begin{equs}
\Psi^{\xi}_{\Lambda,p,a}[\psi]
=
\sum_{\omega\in \Omega} \varphi^{\xi}_{\Lambda,p,a}[(\psi,\omega)] \1[(\psi,\omega)\in \Theta^{\xi}_{\Lambda}].
\end{equs}
We call $\Psi^{\xi}_{\Lambda,p,a}$ the vertex marginal and write $\Psi^{\xi}_{\Lambda}=\Psi^\xi_{\Lambda,p,a}$ when clear from context.
\end{defn}

Similarly we define the pushforward measure under the projection $\theta \mapsto \theta_\Omega$.

\begin{defn}
Let $p,a \in [0,1]$. For $\Lambda \subset \ZZ^d$ and $\xi = (\kappa,\rho) \in \Theta$, let $\Omega^\xi_{\Lambda,p,a}$ denote the probability measure on $\Omega$ defined, for $\omega \in \Omega$, by
\begin{equs}
\Omega^{\xi}_{\Lambda,p,a}[\omega]
=
\sum_{\psi \in \Psi} \varphi^{\xi}_{\Lambda,p,a}[(\psi,\omega)]\1[(\psi,\omega)\in \Theta^{\xi}_{\Lambda}].
\end{equs}    
We call $\Omega^{\xi}_{\Lambda,p,a}$ the edge marginal and write $\Omega^{\xi}_{\Lambda}=\Omega^\xi_{\Lambda,p,a}$ when clear from context.
\end{defn}

\subsection{Basic properties}\label{sec:basic-properties}

We now describe some fundamental properties of the dilute random cluster model. The proofs of these properties follow from straightforward adaptations of the proofs of the analogous properties for the random cluster model, found in \cite[Section 1.2]{HDC-IP}, but we include them for the reader's convenience.

The first property tells us that, conditional on the random environment $\psi \in \Psi$, the dilute random cluster measure coincides with the usual random cluster measure on the random graph induced by $\psi$, $(V_\psi,E_\psi)$.
\begin{prop}
Let $p,a \in [0,1]$, $\Lambda \subset \ZZ^d$ finite, and $\xi=(\kappa,\rho) \in \Theta$. For every $\psi\in \Psi$ such that $\psi_x = \kappa_x$ for all $x \in \ZZ^d \setminus \Lambda$, we have
\begin{equs}\label{eq: conditioning}
\varphi^{\xi}_{\Lambda,p,a}[(\psi,\omega) \mid \theta_\Psi = \psi]
=
\phi^{\textup{RC},\xi}_{\Lambda,\psi,p}[\omega], \qquad \forall \omega \in \Omega: (\psi,\omega) \in \Omega^\xi_\Lambda
\end{equs}
where we recall $\theta_\Psi$ denotes the projection of $\theta$ onto its component in $\Psi$.
\end{prop}

\begin{proof}
Let $\omega \in \Omega$ be such that $(\psi,\omega) \in \Omega^\xi_\Lambda$. Then \eqref{eq: conditioning} is immediate from the following computation:
\begin{equs}
\varphi^{\xi}_{\Lambda,p,a}[(\psi,\omega) \mid \theta_\Psi = \psi]
&=
\frac{\varphi^{\xi}_{\Lambda,p,a}[(\psi,\omega)]}{\sum_{\tilde \omega \in \Omega: (\tilde\omega,\psi) \in \Omega^\xi_\Lambda} \varphi^{\xi}_{\Lambda,p,a}[(\psi,\tilde \omega)] }
\\
&=
\frac
{\prod_{e\in E_{\psi,\Lambda}} \left(\frac{p}{1-p}\right)^{\omega_e} 2^{k\big((\psi,\omega),\Lambda\big)}}
{\sum_{\tilde \omega \in \Omega: (\tilde\omega,\psi) \in \Omega^\xi_\Lambda}\prod_{e\in E_{\psi,\Lambda}} \left(\frac{p}{1-p}\right)^{\tilde\omega_e} 2^{k\big((\psi,\tilde\omega),\Lambda\big)}}.
\end{equs}
\end{proof}

The second property concerns automorphism invariance of the measure. Recall that any automorphism of $\LL^d$ has a natural action on vertices and edges, and hence any event in the $\sigma$-algebra $\cF$. Below, we have that: $\tau\Lambda = \{ \tau x : x \in \Lambda \}$; $\tau\xi = (\tau\kappa,\tau\rho)$ where $\tau\kappa = (\kappa_{\tau x})_{x \in \ZZ^d}$ and $\tau\rho = (\rho_{\tau e})_{e \in \E^d}$; and, $\tau A$ denotes the image of $A$ under the action of $\tau$. 
\begin{prop}
Let $\tau$ be an automorphism of $\LL^d$. Let $p,a \in [0,1]$, $\Lambda \subset \ZZ^d$ finite, and $\xi=(\kappa,\rho) \in \Theta$.  Then, for every event $A$ depending on the vertices and edges in $(\Lambda,E_{\Lambda})$, we have
    \begin{equs}
      \varphi_{\tau\Lambda}^{\tau \xi}[\tau A]
      =
      \varphi_{\Lambda}^{\xi}[A].
    \end{equs}
\end{prop}

\begin{proof}
Note that $\theta \in A$ if and only if $\tau \theta \in \tau A$, and thus it suffices to show that, for any $\theta \in \Omega^\xi_\Lambda$, $ \varphi_{\tau\Lambda}^{\tau \xi}[\tau \theta ] =\varphi_{\Lambda}^{\xi}[\theta]$. For any such $\theta=(\psi,\omega)$, we have
\begin{equs}
\varphi_{\tau\Lambda}^{\tau \xi}[\tau \theta ]
&=
\frac{1}{Z_{\tau \Lambda}^{\tau \xi}} \prod_{x \in \tau\Lambda} \left( \frac{a}{1-a}\right)^{\tau \psi_x} \prod_{e \in E_{\tau\psi, \tau \Lambda}} \left( \frac{p}{1-p}\right)^{\tau \omega_e} r^{|E_{\tau\psi,\tau \Lambda}|} 2^{k(\tau\theta,\tau\Lambda)}
\\
&=
\frac{1}{Z_{\tau \Lambda}^{\tau\xi}} \prod_{x \in \Lambda} \left( \frac{a}{1-a}\right)^{\psi_x} \prod_{e \in E_{\psi, \Lambda}} \left( \frac{p}{1-p}\right)^{\omega_e} r^{|E_{\psi,\Lambda}|} 2^{k(\tau\theta,\tau\Lambda)}
\\
&=
\frac{1}{Z_{\tau \Lambda}^{\tau\xi}} \prod_{x \in \Lambda} \left( \frac{a}{1-a}\right)^{\psi_x} \prod_{e \in E_{\psi, \Lambda}} \left( \frac{p}{1-p}\right)^{\omega_e} r^{|E_{\psi,\Lambda}|} 2^{k(\theta,\Lambda)}
\end{equs}
where the first equality is by definition of the dilute random cluster measure, the second equality is by definition of $\tau$, and the third equality is because graph automorphisms preserve the number of connected components. The proposition then follows from the above calculation together with the observation that the same computations show $Z_{\tau\Lambda}^{\tau \xi} = Z_{\Lambda}^\xi$.
\end{proof}

Next, we turn to the spatial Markov property. We require the following generalisation of a compatible configuration. Let $\Lambda' \subset \Lambda$ be finite subsets of $\ZZ^d$ and let $\xi=(\kappa,\rho) \in \Theta$. We say that $\xi' = (\kappa',\rho') \in \{0,1\}^{\Lambda \setminus \Lambda'} \times \{0,1\}^{E_\Lambda \setminus E_{\Lambda'}}$ is outer-compatible with $\xi$ if and only if $o(\rho')$ is a subset of the edges induced by $\kappa'$ union the edge set $E_\kappa \setminus b(\Lambda)$. Denote by $\xi\cup\xi'\in \Theta$ the boundary condition which is equal to $\xi$ on $(\ZZ^d \setminus \Lambda, \E^d \setminus E_\Lambda)$, $\xi'$ on $(\Lambda\setminus\Lambda')\times(E_{\Lambda}\setminus E_{\Lambda'})$, and $(1,0)$ on $(\Lambda', E_{\Lambda'})$. Note that the choice of $(1,0)$ on the inner region ensures that the entire configuration is compatible and works for all choices of $\xi'$. This is purely for technical and notational convenience - any other choice of compatible extension does not affect the configuration space $\Theta^{\xi \cup \xi'}_{\Lambda'}$. We write
\begin{equs} \label{eq: outer compatible config space}
\Theta^\xi_{\Lambda}(\xi', \Lambda')
=
\{
\theta=(\psi,\omega) \in \Theta^\xi_\Lambda : \psi_x=\kappa'_x \; \forall \, x\in \Lambda\setminus\Lambda',\omega_e=\rho'_e \; \forall \, e\in E_{\Lambda}\setminus E_{\Lambda'}
\}.
\end{equs}
Note that there is a natural bijection $\Theta^\xi_{\Lambda}(\xi',\Lambda') \rightarrow \Theta^{\xi \cup \xi'}_{\Lambda'}$.

\begin{prop} \label{prop: SMP}
Let $p,a \in [0,1]$, $\Lambda'\subset \Lambda$ finite subsets of $\ZZ^d$, and $\xi=(\kappa,\rho)\in \Theta$. For every configuration $\xi'=(\kappa',\rho')\in \{0,1\}^{\Lambda\setminus\Lambda'}\times \{0,1\}^{E_{\Lambda}\setminus E_{\Lambda'}}$ that is outer-compatible with $\xi$, we have
\begin{equs}\label{eq:DMPRC}
\varphi_{\Lambda}^\xi[\cdot \mid \Theta^\xi_\Lambda(\xi',\Lambda')]
=
\varphi_{\Lambda'}^{\xi\cup\xi'}.
\end{equs}
\end{prop}

\begin{proof}
Let $\xi'$ be as above. Observe that, for all $\theta \in \Theta^\xi_\Lambda(\xi',\Lambda')$,
\begin{equs}
\varphi_{\Lambda}^\xi[\theta]
=
\prod_{x \in \Lambda \setminus \Lambda'} \left(\frac{a}{1-a} \right)^{\kappa'_x} 
	&\prod_{x \in  \Lambda'} \left(\frac{a}{1-a} \right)^{\psi_x}
	\prod_{e \in E_{(\xi \cup \xi')_\Psi, \Lambda} \setminus E_{\lambda'}} \left( \frac{p}{1-p}\right)^{\rho'_e} 
	\\
	& \times	\prod_{e \in E_{(\xi \cup \xi')_\Psi, \Lambda'} } \left( \frac{p}{1-p} \right)^{\omega_e}
	r^{|E_{(\xi \cup \xi')_\Psi, \Lambda} \setminus E_{\lambda'}| + | E_{\theta, \Lambda'}|} 
	2^{k(\theta,\Lambda')}.
\end{equs}
Hence,
\begin{equs}
\frac{\varphi_{\Lambda}^\xi[\theta]}{\varphi_{\Lambda}^\xi[ \Theta^\xi_\Lambda(\xi',\Lambda')]}
&=
\frac
{ 
	\prod_{x \in  \Lambda'} \left(\frac{a}{1-a} \right)^{\psi_x} 
	\prod_{e \in E_{(\xi \cup \xi')_\Psi, \Lambda'} } \left( \frac{p}{1-p} \right)^{\omega_e}
	r^{| E_{(\xi \cup \xi')_\Psi, \Lambda'}|} 
	2^{k(\theta,\Lambda')}
}
{
	\sum_{\theta'=(\psi',\omega') \in \Theta^\xi_\Lambda(\xi',\Lambda')} 
	\prod_{x \in  \Lambda'} \left(\frac{a}{1-a} \right)^{\psi'_x} 
	\prod_{e \in E_{(\xi \cup \xi')_\Psi, \Lambda'} } \left( \frac{p}{1-p} \right)^{\omega'_e}
	r^{| E_{\theta, \Lambda'}|} 
	2^{k(\theta,\Lambda')}
}
\end{equs}
from which \eqref{eq:DMPRC} then follows.
\end{proof}

Finally, we consider the finite energy property of the measure.
\begin{prop}\label{prop: finite energy}
Let $p \in (0,1)$, $a \in (0,1)$, $\Lambda \subset \ZZ^d$ finite, and $\xi=(\kappa,\rho)\in \Theta$. For every $x\in \Lambda$ and every configuration $\xi'=(\kappa',\rho')\in \{0,1\}^{\Lambda\setminus\{x\}}\times \{0,1\}^{E_{\Lambda}\setminus E_{\{x\}}}$ that is outer-compatible with $\xi$, we have that for all $\theta = (\psi,\omega) \in \Theta^\xi_\Lambda$,
\begin{equs} \label{eq: finite energy 1}
0
<
\varphi_{\Lambda}^\xi[\psi_x=1 \mid \Theta^\xi_\Lambda(\xi',\{x\})]
<
1.
\end{equs}
 Moreover, for every edge $xy \in b(\Lambda)$ or $xy \in E_{\Lambda} \setminus b(\Lambda)$ such that $x \not\in\Lambda$ and $\kappa_x = 1$, and for every configuration $\xi'=(\kappa',\rho')\in \{0,1\}^{\Lambda\setminus\{x,y\}}\times \{0,1\}^{E_{\Lambda}\setminus \{xy\}}$ that is outer-compatible with $\xi$, we have that for all $\theta = (\psi,\omega) \in \Theta^\xi_\Lambda$,
\begin{equs} \label{eq: finite energy 2}
0
<
\varphi_{\Lambda}^\xi[\omega_{xy}=1 \mid \Theta^\xi_\Lambda(\xi',\{x,y\})]
<
1.
\end{equs}
\end{prop}

\begin{proof}
Note that, by Proposition \ref{prop: SMP}, we have that 
\begin{equs}
\varphi_{\Lambda}^\xi[\psi_x=1 \mid \Theta^{\xi}_{\Lambda}(\xi',\{x\})]
&=
\varphi_{\{x\}}^{\xi \cup \xi'}[\psi_x = 1]
\end{equs}
and 
\begin{equs}
\varphi_{\Lambda}^\xi[\omega_{xy}=1 \mid \Theta^{\xi}_{\Lambda}(\xi',\{x,y\})]
=
\varphi_{\{x,y\}}^{\xi \cup \xi'}[\omega_{xy}=1].
\end{equs}
Equations \eqref{eq: finite energy 1} and \eqref{eq: finite energy 2} follow from the above displays together with the respective conditions on $p$ and $a$ outlined in the proposition.

\end{proof}

\subsection{Edwards-Sokal coupling with the Blume-Capel model}

In this section, we introduce the coupling between the Blume-Capel and the dilute random cluster model. For simplicity, we will restrict to the case where $\Lambda$ is the box $[-n,n]^d\cap \ZZ^d$, which we denote $\Lambda_n$, as some care is required when working with domains whose complement is not connected.

Recall the definitions of the measures $\varphi^0_{\Lambda_n,p,a}$ and $\varphi^1_{\Lambda_n,p,a}$ in Definition~\ref{def: 0,1}, and of the measures $\mu^0_{\Lambda_n,\beta,\Delta}$ and $\mu^+_{\Lambda_n,\beta,\Delta}$ in the Introduction. 
Let $\Delta\in \RR$ and $\beta>0$. Set $p = 1-e^{-2\beta}$ and $a = \frac{e^\Delta}{1+e^{\Delta}}$. For any $n\geq 1$, let $\Sigma_{\Lambda_n}=\{-1,0,1\}^{\Lambda_n}$. We define the measures $\P^0_{\Lambda_n}$ and $\P^1_{\Lambda_n}$ on $\Sigma_{\Lambda_n} \times \Psi\times\Omega$ by
\begin{equs}
\P^0_{\Lambda_n}[ (\sigma,\psi,\omega) ]
=
\frac{\1_{(\sigma,\psi,\omega)\in \mathcal{S}^0}}{Z^0_{\Lambda_n,p,a}}
r^{|E_{\psi,\Lambda_n}|} \prod_{x\in \Lambda_n} \left(\frac{a}{1-a}\right)^{\psi_x}
\prod_{e\in E_{\psi,\Lambda_n}} \left(\frac{p}{1-p}\right)^{\omega_e}
\end{equs}
and 
\begin{equs}
\P^1_{\Lambda_n}[ (\sigma,\psi,\omega) ]
=
\frac{\1_{(\sigma,\psi,\omega)\in \mathcal{S}^1}}{Z^1_{\Lambda_n,p,a}}
r^{|E_{\psi,\Lambda_n}|} \prod_{x\in \Lambda_n} \left(\frac{a}{1-a}\right)^{\psi_x}
\prod_{e\in E_{\psi,\Lambda_n}} \left(\frac{p}{1-p}\right)^{\omega_e}, 
\end{equs}
where for $\#\in \{0,1\}$,
$\mathcal{S}^{\#}$ is the set of triples $(\sigma,\psi,\omega)\in \Sigma_{\Lambda_n}\times \Psi\times \Omega$ such that the following hold:
\begin{enumerate}[(1)]
    \item $(\psi,\omega)\in \Theta^{\#}_{\Lambda_n}$,
    \item $\sigma_x^2=\psi_x$ for every $x\in \Lambda_n$,
    \item for every edge $xy \in E_{\psi,\Lambda_n}$, if $\sigma_x \neq \sigma_y$ then $\omega_{xy}=0$.
\end{enumerate}
Note that the latter condition is equivalent to $\sigma$ being constant on the clusters of $\omega$. 

\begin{prop} \label{prop: es coupling}
Let $\Delta \in \RR$, $\beta>0$ and $n\geq 1$. The marginal measure of $\P^0_{\Lambda_n}$ on $\Sigma_{\Lambda_n}$ is $\mu^0_{\Lambda_n,\beta,\Delta}$, whilst the marginal on $\Psi\times \Omega$ is $\phi^0_{\Lambda_n,p,a}$. Moreover, the marginal measure of $\P^1_{\Lambda_n}$ on $\Sigma_{\Lambda_n}$ is $\mu^+_{\Lambda_n,\beta,\Delta}$, whilst the marginal on $\Psi\times \Omega$ is $\phi^1_{\Lambda_n,p,a}$.
\end{prop}

\begin{proof}
The first assertion is proved in \cite[Theorem 3.7]{GG}. We only prove the second. We first consider the marginal on $\Psi \times \Omega$. Summing over spin configurations in $\Sigma_{\Lambda_n}$, we find that
\begin{equs} 
\begin{split} \label{eq: ES 1}
\sum_{\sigma \in \Sigma_{\Lambda_n}} \P^1_{\Lambda_n}[(\sigma,\psi,\omega)]
&=
\frac{1}{Z^1_{\Lambda_n,p,a}}r^{|E_{\psi,\Lambda_n}|} \prod_{x \in \Lambda_n} \left( \frac{a}{1-a}\right)^{\psi_x} \prod_{e \in E_{\psi,\Lambda_n}} \left( \frac{p}{1-p} \right)^{\omega_e} \sum_{\sigma \in \Sigma_{\Lambda_n}} \1_{(\sigma,\psi,\omega)\in \mathcal{S}^1}
\\
&=
\frac{\1_{(\psi,\omega) \in \Theta^1_{\Lambda_n}}}{Z^1_{\Lambda_n,p,a}}r^{|E_{\psi,\Lambda_n}|} \prod_{x \in \Lambda_n} \left( \frac{a}{1-a}\right)^{\psi_x} \prod_{e \in E_{\psi,\Lambda_n}} \left( \frac{p}{1-p} \right)^{\omega_e} 
\\
&\qquad \qquad \qquad \times
\sum_{\sigma \in \Sigma_{\Lambda_n} : \sigma^2 = \psi} \prod_{\mathcal{C} \in \tilde k(\omega,\Lambda_n)}\1_{\sigma_x = \sigma_y  \, \forall xy \in \mathcal{C}} 
\\
&=
\frac{\1_{(\psi,\omega) \in \Theta^1_{\Lambda_n}}}{Z^1_{\Lambda_n,p,a}}r^{|E_{\psi,\Lambda_n}|} \prod_{x \in \Lambda_n} \left( \frac{a}{1-a}\right)^{\psi_x} \prod_{e \in E_{\psi,\Lambda_n}} \left( \frac{p}{1-p} \right)^{\omega_e} 2^{k(\theta,\Lambda_n)}
=
\phi^1_{\Lambda_n}[(\psi,\omega)],
\end{split}
\end{equs}
where $\tilde k(\theta,\Lambda_n)$ is the set of connected clusters of $\theta=(\omega,\psi)$ which do not intersect $\ZZ^d\setminus \Lambda_n$.

We now turn to the marginal on $\Sigma_{\Lambda_n}$. First, observe that, for any $\sigma \in \Sigma_{\Lambda_n}$ and $\psi \in \Psi$ such that $\sigma^2 = \psi$, we have
\begin{equs} \label{eq: ES 2}
\begin{split}
\prod_{xy \in E_{\psi,\Lambda_n}} e^{\beta \sigma_x \sigma_y}
&=
\prod_{xy \in E_{\psi,\Lambda_n}} e^{\beta (2\1_{\sigma_x = \sigma_y} - 1)} 
=
e^{-\beta |E_{\psi,\Lambda_n}|} \prod_{xy \in E_{\psi,\Lambda_n}} \left( e^{2\beta} \1_{\sigma_x = \sigma_y } + \1_{\sigma_x \neq \sigma_y} \right)
\\
&=
e^{-\beta|E_{\psi,\Lambda_n}|} \prod_{xy \in E_{\psi,\Lambda_n}} \Big((e^{2\beta}-1) \1_{\sigma_x = \sigma_y} + 1 \Big)
\\
&=
r^{|E_{\psi,\Lambda_n}|} \prod_{xy \in E_{\psi,\Lambda_n}} \Big(\frac{p}{1-p} \1_{\sigma_x = \sigma_y} + 1 \Big)
\\
&=
r^{|E_{\psi,\Lambda_n}|} \sum_{\omega \in \Omega} \1_{(\sigma,\psi,\omega)\in\mathcal{S}^1} \prod_{e \in E_{\psi,\Lambda_n}}\left( \frac{p}{1-p} \right)^{\omega_e} 
\end{split}
\end{equs}
Above, in the first line, we use that $\sigma^2 \equiv 1$ on $E_{\psi,\Lambda_n}$ and a trivial rewriting of the exponential of the indicator function; in the second line, we reorder terms; in the third line, we use the definition of $r$ and $p$ in terms of $\beta$; and in the fourth line, we expand the product in terms of a sum over $\omega$ such that $o(\omega) \subset E_{\psi,\Lambda_n}$ and use that, for any $\omega$ in the sum, $E_{\psi,\Lambda_n} = o(\omega) + c(\omega)$, where $c(\omega)$ denotes the set of closed edges in $\omega$.

Second, observe that, for any $\sigma \in \Sigma_{\Lambda_n}$,
\begin{equs} \label{eq: ES 3}
\begin{split}
\prod_{x \in \Lambda_n} e^{\Delta \sigma_x^2}
&=
\prod_{x \in \Lambda_n}\left( \frac{a}{1-a} \right)^{\psi_x}
\end{split}.
\end{equs}

Combining \eqref{eq: ES 2} and \eqref{eq: ES 3}, together with considerations made in the summation over $\sigma$ in \eqref{eq: ES 1}, allows us to deduce the equality of the Blume-Capel and dilute random cluster partition functions
\begin{equs}
Z^1_{\Lambda_n,\beta,\Delta}
=
Z^{1}_{\Lambda_n,p,a}
\end{equs}
and, hence, for all $\sigma \in \Sigma_{\Lambda_n}$,
\begin{equs}
\sum_{(\psi,\omega) \in \Psi \times \Omega} \P^1_{\Lambda_n}[(\sigma,\psi,\omega)]
=
\mu^+_{\Lambda_n,\beta,\Delta}(\sigma). 
\end{equs}
\end{proof}

\begin{rem}
For $\Lambda\subset \ZZ^d$ finite such that $\LL^d\setminus (\overline{\Lambda},E_{\Lambda})$ is disconnected, the same argument gives a coupling between $\mu^+_{\Lambda,\beta,\Delta}$ and a modification of the wired measure, where all clusters intersecting $\ZZ^d\setminus \Lambda$ are counted as one.
\end{rem}

The computations in the proof of Proposition \ref{prop: es coupling} have useful probabilistic interpretations. Conditionally on $(\psi,\omega)$, the following statements hold for the law of $\sigma$:
\begin{enumerate}[(a)]
    \item the spins are constant on the clusters of $\omega$, and the spin of each cluster which does not intersect $\ZZ^d\setminus \Lambda_n$ is uniformly distributed on the set $\{\pm 1\}$,
    \item the spins on different clusters are independent random variables.
\end{enumerate}
Conditionally on $\sigma$, the following statements hold for the law of $(\psi,\omega)$:
\begin{enumerate}[(i)]
    \item $\psi_x=0$ if and only if $\sigma_x=0$,
    \item the random variables $\omega_e$ are independent,
    \item for an edge $xy$, $\omega_{xy}=0$ if $\sigma_x\neq \sigma_y$, and $\omega_{xy}=1$ with probability $p$ if $\sigma_x=\sigma_y$.
\end{enumerate}

The couplings of Proposition \ref{prop: es coupling} allow to represent spin correlations in the Blume-Capel model geometrically in terms of connection probabilities in the dilute random cluster model. For $A,B \subset \Lambda$, denote by $A\longleftrightarrow B$ the event that there exists $x \in A$ and $y \in B$ such that $x$ is connected to $y$ by an open path in $\omega$ (that is to say, there is a sequence of edges emanating from $x$ and ending at $y$ such that each edge $e$ in this path satisfies $\omega_e = 1$). 

\begin{cor} \label{cor: ES correlations}
Let $\Delta \in \RR$ and $\beta>0$. For any $n\geq 1$, we have 
\begin{align}
\label{eq: es cor: 1}
\langle \sigma_x \rangle^+_{\Lambda_n,\beta,\Delta}
&=
\phi^1_{\Lambda_n,p,a}[x\longleftrightarrow\partial \Lambda_n], \qquad \forall x \in \Lambda_n
\\ \label{eq: es cor: 2}
\langle \sigma_x \sigma_y \rangle^+_{\Lambda_n,\beta,\Delta}
&=
\phi^1_{\Lambda_n,p,a}[x\longleftrightarrow y], \qquad \forall x,y \in \Lambda_n
\\ \label{eq: es cor: 3}
\langle \sigma_x \sigma_y \rangle^0_{\Lambda_n,\beta,\Delta}
&=
\phi^0_{\Lambda_n,p,a}[x\longleftrightarrow y], \qquad \forall x,y\in\Lambda_n.
\end{align}   
\end{cor}

\begin{proof}
Write $\mathbf{E}^1_{\Lambda_n}$ to denote expectation with respect to $\P_{\Lambda_n}^1$. By Proposition \ref{prop: es coupling}, we have
\begin{equs}
\langle \sigma_x \rangle^1_{\Lambda_n,\beta,\Delta}
&=
\mathbf{E}^1_\Lambda[\sigma_x \1_{x \longleftrightarrow \partial\Lambda_n}] + \mathbf{E}^1_\Lambda[\sigma_x \1_{x \centernot\longleftrightarrow \partial\Lambda}].
\end{equs}
The second term on the righthand side is $0$ by the (conditional) independence of spins on different connected clusters of $(\psi,\omega)$. Applying Proposition \ref{prop: es coupling} to the first term on the righthand side yields \eqref{eq: es cor: 1}. 

Equations \eqref{eq: es cor: 2} and \eqref{eq: es cor: 3} follow by similar considerations.
\end{proof}

\subsection{Stochastic ordering and the FKG inequality}

We introduce the following partial order on the set $\Theta$. For $\xi,\xi' \in \Theta$, write $\xi\leq \xi'$  to denote that $\psi_x\leq \psi'_x$ for every $x\in \ZZ^d$, and $\omega_e\leq \omega'_e$ for every $e\in \E^d$. A function $f:\Theta \mapsto \RR$ is called increasing if $\xi\leq \xi'$ implies that $f(\xi)\leq f(\xi')$. An event $A$ is said to be increasing if $\1_A$ is increasing. 

We are going to show that all the dilute random cluster measures satisfy the FKG inequality. In the proof, we use that the vertex marginals satisfy an analogous inequality, as formalised below.
\begin{lem}
Let $p,a \in [0,1]$, $\Lambda \subset \ZZ^d$ finite, and $\xi \in \Theta$. Then, 
\begin{equs} \label{eq: fkg vertex}
\Psi^{\xi}_{\Lambda}[A\cap B]
\geq 
\Psi^{\xi}_{\Lambda}[A]\Psi^{\xi}_{\Lambda}[B]
\end{equs}
for all increasing events $A$ and $B$ on $\Psi = \{0,1\}^{\ZZ^d}$ that are $\Lambda$-measurable.
\end{lem}
\begin{proof}
This is established in \cite[Theorem 5.3]{GG}.
\end{proof}

\begin{prop}\label{prop: FKG}
Let $p,a \in [0,1]$, $\Lambda \subset \ZZ^d$ finite, and $\xi \in \Theta$. For all increasing events $A,B \in \cF$ that are $\Lambda$-measurable, we have
\begin{equs}\label{eq:FKGRC}
\varphi^{\xi}_{\Lambda,p,a}[A\cap B]
\geq
\varphi^{\xi}_{\Lambda,p,a}[A]\varphi^{\xi}_{\Lambda,p,a}[B].
\end{equs}
\end{prop}

\begin{proof}
Let $A \in \cF$ be an increasing event that is $\Lambda$-measurable. For each $\psi\in \Psi$, let $C_{\psi,A}:=\{\omega \in \Omega : (\psi,\omega)\in A\}$. Note that this event is $V_\psi$-measurable. The fact that $A$ is increasing has two straightforward consequences: first, $C_{\psi,A}$ is an increasing event on the $\sigma$-algebra induced by the restriction of $\Omega$ to $E_{\psi}$; second, $C_{\psi,A}\subset C_{\psi',A}$ whenever $\psi\leq \psi'$.

Thus, by the conditioning equality \eqref{eq: conditioning} and the usual FKG inequality for the random cluster measure, we have that
\begin{equs} \label{eq: fkgproof}
\varphi^{\xi}_{\Lambda}[A \cap B]
&=
\Psi^{\xi}_{\Lambda}\Big[\phi^{\textup{RC},\xi \cap \mathbf{0}_\Lambda}_{\Lambda, \psi}[C_{\psi,A \cap B}]\Big]
\\
&\geq 
\Psi^{\xi}_{\Lambda}\Big[\phi^{\textup{RC},\xi \cap \mathbf{0}_\Lambda}_{\Lambda, \psi}[C_{\psi,A }]\phi^{\textup{RC},\xi \cap \mathbf{0}_\Lambda}_{\Lambda, \psi}[C_{\psi, B}]\Big].
\end{equs}

For any $\tilde A \in \cF$ increasing, the mapping $\psi \mapsto \phi^{\textup{RC},\xi}_{\Lambda, \psi}[C_{\psi,\tilde A }]$ is increasing. Indeed, for any $\psi,\psi' \in \Psi$ that coincide with $\xi$ outside $\Lambda$ and such that $\psi' \geq \psi$, we have
\begin{equs}
\phi^{\textup{RC},\xi \cap \mathbf{0}_\Lambda}_{\Lambda, \psi}[C_{\psi,\tilde A }]
\leq 
\phi^{\textup{RC},\xi \cap \mathbf{0}_\Lambda}_{\Lambda, \psi'}[C_{\psi,\tilde A }]
\leq 
\phi^{\textup{RC},\xi\cap \mathbf{0}_\Lambda}_{\Lambda, \psi'}[C_{\psi',\tilde A }].
\end{equs}
The fist inequality is subtle: it follows due to a special monotonicity of the random cluster measure in the domain corresponding to the pushing of free boundary conditions away -- see \cite[Proposition 5]{DCT} and also compare with Proposition \ref{prop: monotonicity in domain} below. The second inequality is due to inclusion of events. The desired result \eqref{eq:FKGRC} then follows from applying the inequality \eqref{eq: fkg vertex} in \eqref{eq: fkgproof}. 
\end{proof}

We now compare different dilute random cluster measures. For two measures $\varphi_1,\varphi_2$ on $\Theta$, we write $\varphi_1 \preceq \varphi_2$ if $\varphi_1(f)\leq \varphi_2(f)$ for every increasing function $f$. In this case, we say that $\varphi_2$ stochastically dominates $\varphi_1$.

The following proposition states that the dilute random cluster measures are stochastically increasing in $p,a$, and the boundary condition.
\begin{prop}\label{prop: domination}
For every $p,p',a,a'\in [0,1]$ and $\xi,\xi'\in \Theta$ such that $a\leq a'$, $p\leq p'$ and $\xi\leq \xi'$, 
\begin{equs}
\varphi^{\xi}_{\Lambda,p,a}
\preceq
\varphi^{\xi'}_{\Lambda,p',a'}.
\end{equs}
\end{prop}

\begin{proof}
See \cite[Theorem 5.12]{GG}. 
\end{proof}

It is important for our purposes to obtain a more refined stochastic domination which conveys that, among boundary conditions that induce no connections between boundary vertices, $0$ boundary conditions are the most favourable. In the remainder of this subsection, we fix $p,a \in [0,1]$.

We begin with some setup. First, we say that a measure $\Psi_1$ on $\{0,1\}^{\Lambda}$ is strictly positive if $\Psi_1(\psi)>0$ for every $\psi\in \{0,1\}^{\Lambda}$. Second, the notion of stochastic domination extends in the obvious way to probability measures on $\{0,1\}^\Lambda$. Third, we fix some notation. Let $\Lambda \subset \ZZ^d$ be finite. Given a configuration $\psi\in \{0,1\}^\Lambda$ and $T\subset \Lambda$, we define the configurations $\psi^{T}$ and $\psi_{T}$ by
\begin{equs}
    \psi^{T}_y
    = 
    \begin{cases}
        \psi_y, & y\not\in T \\
        1, & y\in T,
    \end{cases}
    \qquad
    (\psi_{T})_y
    = 
    \begin{cases}
        \psi_y, & y\not\in T \\
        0, & y\in T.
    \end{cases}
\end{equs}
Thus, $\psi^T$ corresponds to the configuration where the vertices of $\psi$ in $T$ are opened, and $\psi_T$ corresponds to the configuration where the vertices of $\psi$ in $T$ are closed.

We now recall a standard fact from the theory of monotonic measures.
\begin{lem} \label{lem: fkglattice}
Let $\Lambda \subset \ZZ^d$ be finite. Let $\Psi_1,\Psi_2$ be strictly positive probability measures on $\{0,1\}^\Lambda$. If 
\begin{equs}\label{eq:stoch1}
\Psi_2\left[\psi^{\{x\}}\right]\Psi_1\left[\psi_{\{x\}}\right]
\geq
\Psi_2\left[\psi_{\{x\}}\right]\Psi_1\left[\psi^{\{x\}}\right]
\end{equs}
and in addition if there exists $i \in \{1,2\}$ such that
\begin{equs}\label{eq:stoch2}
\Psi_i \left[\psi^{\{x,y\}}\right]\Psi_i \left[\psi_{\{x,y\}}\right]
\geq
\Psi_i \left[\psi^{\{x\}}_{\{y\}}\right]\Psi_i \left[\psi^{\{y\}}_{\{x\}}\right],
\end{equs}
then $\Psi_1 \preceq \Psi_2$.
\end{lem}

\begin{proof}
See\footnote{Note that the proof of \cite[Theorem 2.6]{GGBook} is based on \cite[Theorem 2.3]{GGBook}, the statement of which is not correct. A correct proof of \cite[Theorem 2.6]{GGBook} was pointed out in a subsequent correction \cite{GGBookErratum}.} \cite[Theorem 2.6]{GGBook}. 
\end{proof}

For $S\subset \partial \Lambda$, let $E(S,\Lambda)$ be the set of edges with one endpoint in $S$ and the other in $\Lambda$. Given $\xi\in \Theta$, we let $\xi\cap \mathbf{0}_S\in \Theta$ be the configuration, suitably extended to a compatible configuration if necessary, which is equal to $0$ on $S\times E_S$, and otherwise coincides with $\xi$. In addition, we let $\xi\cap \mathbf{0}_{E(S,\Lambda)}\in \Theta$ be the configuration whose edge configuration is equal to $0$ on $E(S,\Lambda)$, and otherwise coincides with $\xi$. Note that, in the latter, $\psi$ is not necessarily $0$ on $S$. 

We define 
\begin{equs}
\varphi^{\xi\cap \, \mathbf{0}_{E(S,\Lambda)}}_{\Lambda}
=
\varphi^{\xi}_{\Lambda}[\cdot \mid \omega_e=0 \textup{ for every } e\in E(S,\Lambda)].
\end{equs}
We write $\Psi^{\xi\cap \, \mathbf{0}_{S}}_{\Lambda}$ and $\Psi^{\xi\cap \, \mathbf{0}_{E(S,\Lambda)}}_{\Lambda}$ for the vertex marginals of $\varphi^{\xi\cap \, \mathbf{0}_{S}}_{\Lambda}$ and $\varphi^{\xi\cap \, \mathbf{0}_{E(S,\Lambda)}}_{\Lambda}$, respectively.

We now prove the desired refined stochastic domination result on the level of the vertex marginal. 
\begin{lem}\label{free-empty-vertex}
Let $\Lambda \subset \ZZ^d$ be finite and let $S\subset \partial \Lambda$. For every $\xi\in \Theta$, 
\begin{equs}
\Psi^{\xi\cap \, \mathbf{0}_{E(S,\Lambda)}}_{\Lambda}
\preceq
\Psi^{\xi\cap \, \mathbf{0}_S}_{\Lambda}.
\end{equs}
\end{lem}

\begin{proof}
We prove that \eqref{eq:stoch1} is satisfied for every $x \in \Lambda$. For $\psi \in \Psi$, write $N_x(\psi)$ and $N_{x,S}(\psi)$ to denote the number of neighbours of $x$ in the graph $(V_{\psi^{\{x\}}}, E_{\psi^{\{x\}}})$ and $(V_{\psi^{\{x\}} \cap \mathbf{0}_S}, E_{\psi^{\{x\}}\cap\mathbf{0}_S})$, respectively. 

By direct calculation, we have
\begin{equs}
\frac
{\Psi^{\xi\cap \, \mathbf{0}_{E(S,\Lambda)} }_{\Lambda}[\psi^{\{x\}}]}
{\Psi^{\xi\cap \, \mathbf{0}_{E(S,\Lambda)}}_{\Lambda}[\psi_{\{x\}}]}
&=
r^{N_x(\psi)} \frac{a}{1-a}  \frac{Z^{\textup{RC},\xi\cap \, \mathbf{0}_{E(S,\Lambda)}}_{\Lambda, \psi^{\{x\}}}}
{Z^{\textup{RC}, \xi\cap \, \mathbf{0}_{E(S,\Lambda)}}_{\Lambda, \psi_{\{x\}}}}
\\
\frac{\Psi^{\xi\cap \, \mathbf{0}_S}_{\Lambda}[\psi^{\{x\}}]}{\Psi^{\xi\cap \, \mathbf{0}_S}_{\Lambda}[\psi_{\{x\}}]}
&=
r^{N_{x,S}(\psi)} \frac{a}{1-a} \frac{Z^{\textup{RC},\xi\cap \, \mathbf{0}_S}_{\Lambda, \psi^{\{x\}}}}{Z^{\textup{RC},\xi\cap \, \mathbf{0}_S}_{\Lambda, \psi_{\{x\}}}}
\end{equs}
where $Z^{\textup{RC},\xi\cap \, \mathbf{0}_{E(S,\Lambda)}}_{\Lambda, \psi^{\{x\}}}$ denotes the partition function of the random cluster model on $(V_{\psi^{\{x\}}}, E_{\psi^{\{x\}}})$ with boundary conditions $\xi \cap \mathbf{0}_\Lambda$, conditioned on $E(S,\Lambda)$ being closed, and $Z^{\textup{RC},\xi\cap \, \mathbf{0}_{E(S,\Lambda)}}_{\Lambda, \psi_{\{x\}}}$ is defined similarly.  Note that we stress that the edge-projection $(\mathbf{0}_\Lambda)_{\Omega}$ is measurable with respect to edges in $b(\Lambda)$ and does \emph{not} impose any condition on edges in $E(\Lambda)$ with one endpoint in $\partial \Lambda$.

Note that 
\begin{equs}
 \frac{Z^{\textup{RC},\xi\cap \, \mathbf{0}_{E(S,\Lambda)}}_{\Lambda, \psi^{\{x\}}}}
{Z^{\textup{RC}, \xi\cap \, \mathbf{0}_{E(S,\Lambda)}}_{\Lambda, \psi_{\{x\}}}}
= 
 \frac{Z^{\textup{RC},\xi\cap \, \mathbf{0}_S}_{\Lambda, \psi^{\{x\}}}}{Z^{\textup{RC},\xi\cap \, \mathbf{0}_S}_{\Lambda, \psi_{\{x\}}}}
\end{equs}
because in both cases, the edges in $E(S,\Lambda)$ are necessarily closed.
Therefore, since $r\leq 1$ and we trivially have $N_{x,S}(\psi)\leq N_x(\psi)$, it follows that \eqref{eq:stoch1} is satisfied with $\Psi_1 = \Psi^{\xi\cap \, \mathbf{0}_{E(S,\Lambda)}}_{\Lambda}$ and $\Psi_2 = \Psi^{\xi\cap \, \mathbf{0}_S}_{\Lambda}$.

It remains to establish \eqref{eq:stoch2}. In this case, we may apply \cite[Proposition 5.5]{GG} for $\Phi_1=\Phi_2=\Psi^{\xi\cap \, \mathbf{0}_S}_{\Lambda}$ to obtain that \eqref{eq:stoch2} is satisfied by $\Psi^{\xi\cap \, \mathbf{0}_S}_{\Lambda}$. The desired result now follows by Lemma \ref{lem: fkglattice}.
\end{proof}

We now turn to the desired refined stochastic domination result for the full dilute random cluster model. We emphasise: this result states that the maximal (in terms of stochastic domination) boundary condition that induces free boundary conditions is that of $0$ vertices. See Figure~\ref{fig:closed edges boundary}.

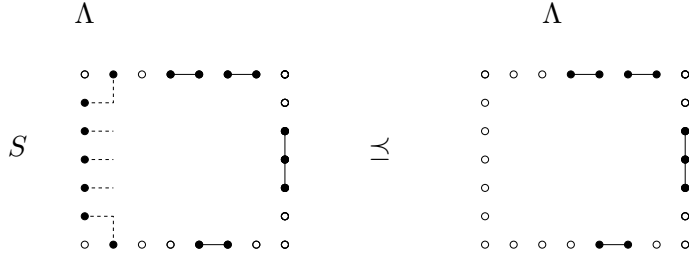
\begin{figure}
    \centering
    \includegraphics[width=.5\linewidth]{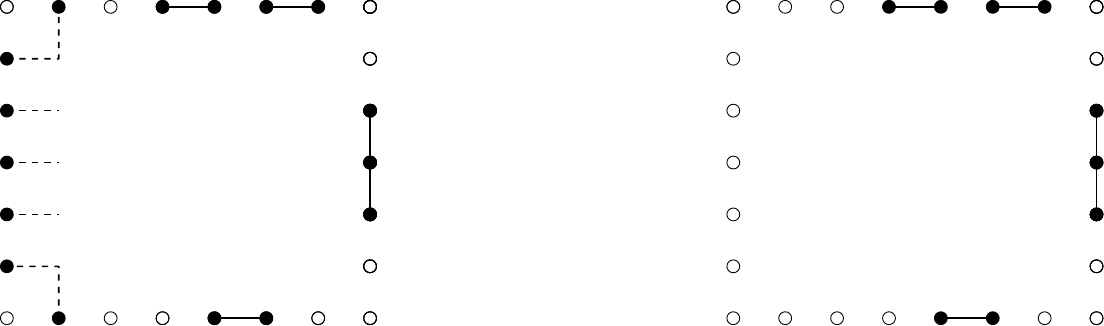}
    \put(-230,85){$\Lambda$}
    \put(-55,85){$\Lambda$}
    \put(-120,35){$\preceq$}
    \put(-255,35){$S$}
    \caption{An example illustrating the stochastic domination of Proposition~\ref{free-closed}. The dashed lines denote the edges of $S$ which are conditioned to be closed. The full dots denote open vertices and the dots without interior denote the closed vertices. Open edges are depicted in straight lines.}
    \label{fig:closed edges boundary}
\end{figure}

\begin{prop}\label{free-closed}
Let $\Lambda$ be a finite subset of $\ZZ^d$, and let $S\subset \partial \Lambda$. For every $\xi\in \Theta$,
\begin{equs}
\varphi^{\xi\cap \, \mathbf{0}_{E(S,\Lambda)}}_{\Lambda}
\preceq
\varphi^{\xi\cap \, \mathbf{0}_S}_{\Lambda}.
\end{equs}
\end{prop}

\begin{proof}[Proof of Proposition~\ref{free-closed}]
Let $A \in \cF$ be $\Lambda$-measurable and increasing. Recall the notation introduced in the proof of Proposition~\ref{prop: FKG}. By \eqref{eq: conditioning},
\begin{equs}
\varphi^{\xi\cap \, \mathbf{0}_S }_{\Lambda}[A]
=
\Psi^{\xi\cap \, \mathbf{0}_S}_{\Lambda}\left[\phi^{\textup{RC},\xi\cap \, \mathbf{0}_S}_{\Lambda,\psi}[C_{\psi,A}]\right].
\end{equs}
As argued above, the mapping $\psi\mapsto \phi^{\textup{RC},\xi\cap \, \mathbf{0}_S}_{\Lambda, \psi}[C_{\psi,A}]$ is increasing. Moreover, the measures $\phi^{\textup{RC},\xi\cap \, \mathbf{0}_S}_{\Lambda, \psi}$ and $\phi^{\textup{RC},\xi\cap \, \mathbf{0}_{E(S,\Lambda)}}_{\Lambda, \psi}$ coincide. The desired result then follows by applying Lemma~\ref{free-empty-vertex} and undoing the conditioning.
\end{proof}

\subsection{Monotonicity in the domain}

We now develop stochastic comparison results for dilute random cluster measures on different domains. Let $\textup{Cyl}(\ZZ^d \times \Theta) = \{ (\Lambda, \xi) : \Lambda\subset \Z^d \text{ finite}, \xi\in\Theta\}$. We begin with two definitions. The motivation for the first definition is to capture the notion that pulling inwards wired boundary conditions increases the probability of increasing events. The motivation for the second definition is to capture the notion that pushing away free boundary conditions increases the probability of increasing events. See Figure~\ref{fig:0-1-partial orders}.

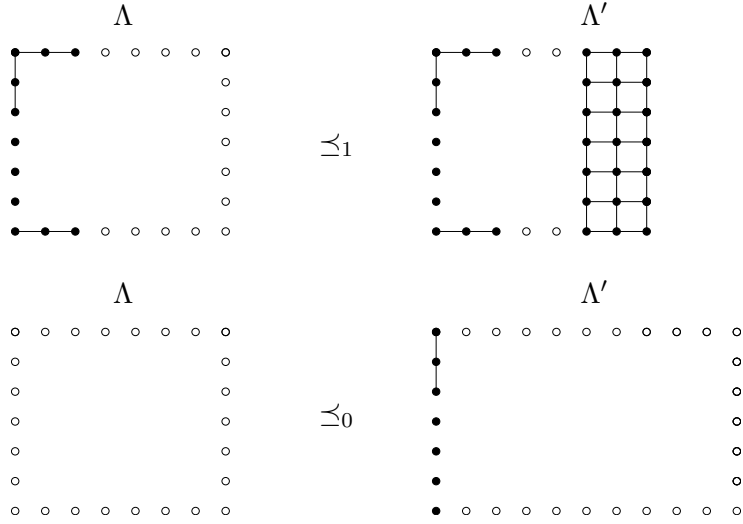
\begin{figure}
    \centering
    \includegraphics[width=.6\linewidth]{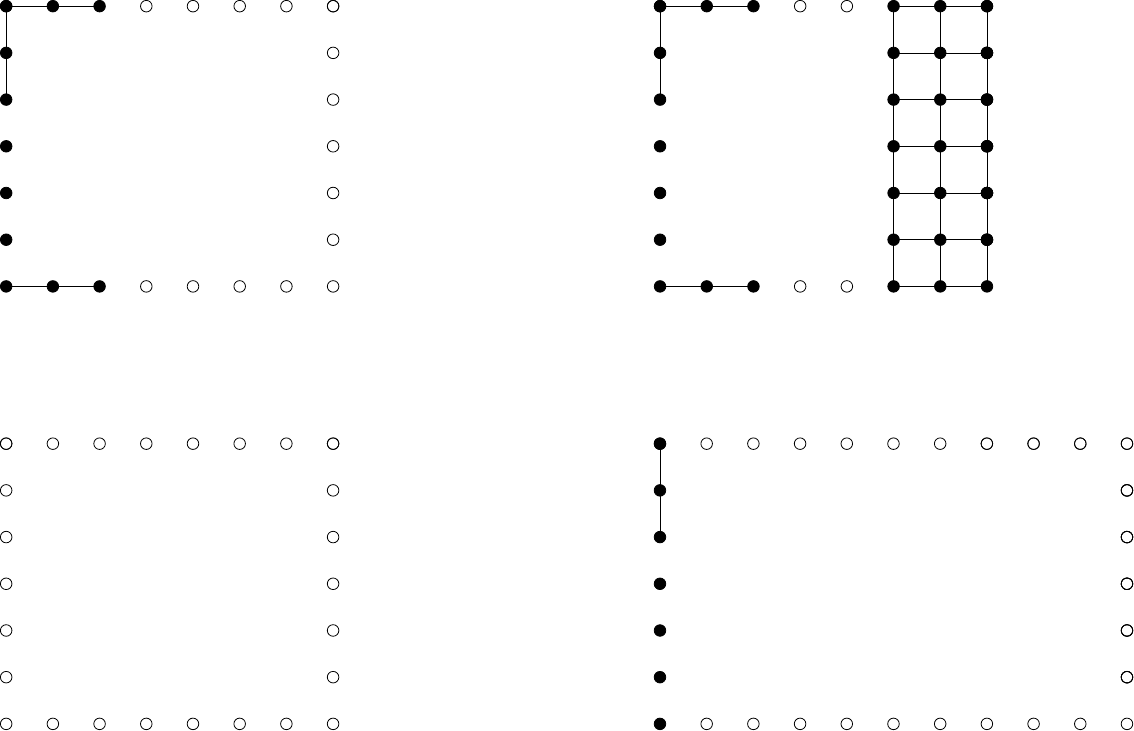}
    \put(-235,80){$\Lambda$}
    \put(-60,80){$\Lambda'$}
    \put(-158,35){$\preceq_0$}
    \put(-235,185){$\Lambda$}
    \put(-60,185){$\Lambda'$}
    \put(-158,137){$\preceq_1$}
    \caption{Two examples illustrating the $\preceq_1$ and $\preceq_0$ partial orders. The full dots denote open vertices and the dots without interior denote the closed vertices. Open edges are depicted in straight lines.}
    \label{fig:0-1-partial orders}
\end{figure}

\begin{defn}
Let $\preceq_1$ be the partial order on $\textup{Cyl}(\ZZ^d \times \Theta)$ defined, for $(\Lambda, \xi), (\Lambda',\xi') \in \textup{Cyl}(\ZZ^d \times \Theta)$, by the relation 
\begin{equs}
(\Lambda, \xi) \preceq_1 (\Lambda', \xi') 
\qquad \Longleftrightarrow \qquad
\Lambda'\subset \Lambda \quad \text{and}\quad \text{$\xi'\geq \xi\cup \1_{\Lambda,\Lambda'}$}
\end{equs}
where $\xi\cup\1_{\Lambda,\Lambda'}\in \Theta$ is the boundary condition which is equal to $1$ on $(\Lambda\setminus\Lambda')\times (E_{\Lambda}\setminus E_{\Lambda'})$, and otherwise coincides with $\xi$.
\end{defn}

\begin{defn}
Let $\preceq_0$ be the partial order on $\textup{Cyl}(\ZZ^d \times \Theta)$ defined, for $(\Lambda, \xi), (\Lambda',\xi') \in \textup{Cyl}(\ZZ^d \times \Theta)$, by the relation 
\begin{equs}
(\Lambda,\xi)\preceq_0(\Lambda',\xi')
\qquad \Longleftrightarrow \qquad
\Lambda\subset \Lambda' \quad \text{and}\quad \text{$\xi\leq \xi'\cap \mathbf{0}_{\Lambda,\Lambda'}$}
\end{equs}
where $\xi'\cap \mathbf{0}_{\Lambda,\Lambda'}\in \Theta$ is the boundary condition which is equal 
to $0$ on $(\Lambda\setminus\Lambda')\times (E_{\Lambda}\setminus E_{\Lambda'})$, and otherwise coincides with $\xi'$.
\end{defn}

With these definitions in hand, we now turn to the monotonicity in the domain of dilute random cluster probabilities. 
\begin{prop} \label{prop: monotonicity in domain}
Let $(\Lambda, \xi), (\Lambda',\xi') \in \textup{Cyl}(\ZZ^d \times \Theta)$ such that either $(\Lambda,\xi)\preceq_1(\Lambda',\xi')$ or $(\Lambda,\xi)\preceq_0(\Lambda',\xi')$. Then,
\begin{equs}
\varphi_\Lambda^\xi 
\preceq
\varphi_{\Lambda'}^{\xi'}
\end{equs}
in the sense that, for every increasing event $A \in \cF$ that is measurable with respect to $(\Lambda\cap \Lambda',E_{\Lambda}\cap E_{\Lambda'})$, we have
  \begin{equs}
    \label{eq:MON}
    \varphi_{\Lambda}^\xi[A] 
    \le 
    \varphi_{\Lambda'}^{\xi'}[A].
  \end{equs}
\end{prop}

\begin{proof}
We prove the assertion only in the case $(\Lambda,\xi)\preceq_1(\Lambda',\xi')$; the other case follows similarly.
By \eqref{eq:FKGRC} and \eqref{eq:DMPRC} we have
\begin{equs}
\varphi_{\Lambda}^{\xi}[A]
\leq 
\varphi_\Lambda^\xi[A \mid \theta|_{\Lambda\setminus\Lambda'} = \1_{\Lambda,\Lambda'}]
=
\varphi_{\Lambda'}^{\xi\cup\1_{\Lambda,\Lambda'}}[A].
\end{equs}
Note that the above could also be obtained by conditioning on the configuration in $\Lambda\setminus\Lambda'$ and applying \eqref{eq:DMPRC}.
To finish the proof, we use Proposition \ref{prop: domination} to obtain that 
\begin{equs}
\varphi_{\Lambda'}^{\xi\cup\1_{\Lambda,\Lambda'}}[A]
\leq
\varphi_{\Lambda'}^{\xi'}[A],
\end{equs}
as desired.
\end{proof}

\section{Basic facts about the phase transition}\label{sec:basic-facts-phase-transition}

In this section, we use the dilute random cluster representation of the Blume-Capel model to prove various fundamental properties about its phase diagram, in particular Theorem \ref{thm: BC phase diagram} and the discontinuity part of Theorem \ref{thm: existence}. First, we show that the critical point of the Blume-Capel model coincides with the critical percolation threshold of the dilute random cluster model and state some basic consequences of this. Then we prove Theorem \ref{thm: BC phase diagram} by proving that the curve of critical points $\{ (\Delta, \beta_c(\Delta)) : \Delta \in \R \}$ is continuous and decreasing. Finally, we show that  these properties together with classical results using Pirogov-Sinai theory establish that the phase transition is discontinuous at the critical point for sufficiently negative values of $\Delta$. 

We begin by collecting important facts about the natural infinite volume limits of the free and wired measures. 
\begin{prop}\label{prop: inf vol limits}
Let $p,a \in [0,1]$. Then, the weak limits of $\varphi^1_{\Lambda,p,a}$ and $\varphi^0_{\Lambda,p,a}$ as $\Lambda \Uparrow \ZZ^d$, denoted by $\varphi^1_{p,a}$ and $\varphi^0_{p,a}$, exist and satisfy the following Gibbs property: let $\varphi = \varphi^1_{p,a}$ or $\varphi^0_{p,a}$. Then, for all $\Lambda \subset \ZZ^d$ finite and $\varphi$-a.s. $\lambda \in \Theta$,
\begin{equs}
\varphi[A \mid \cF_{\Lambda^c}] (\lambda)
=
\varphi^\lambda_{\Lambda, p,a}[A],
 \qquad
A \in \cF_\Lambda
\end{equs}
where $\cF_\Lambda$ (resp. $\cF_{\Lambda^c}$) consists of events in $\cF$ that are $\Lambda$-measurable (resp. $\Lambda^c$-measurable).

Furthermore, both $\varphi^1_{p,a}$ and $\varphi^0_{p,a}$ are invariant under the group of automorphisms of $\LL^d$ and mixing (hence, ergodic) under the subgroup of translations, where by mixing we mean that, given any events $A,B \in \cF$ depending on finitely many edges and vertices, we have that
\begin{equs}
\lim_{|x| \rightarrow \infty}\varphi[A \cap \tau_x B]
=
\varphi[A]\varphi[B],
\end{equs}
with $\tau_x$ denoting translation by $x \in \ZZ^d$.
\end{prop}

\begin{proof}
When $p,a \in (0,1)$, the existence and Gibbs property follows from \cite[Theorems 7.2 and 7.4]{GG}. The invariance and mixing under translations is then a standard consequence of the convergence, \eqref{eq:FKGRC}, and Proposition \ref{prop: domination}. See for instance \cite[Lemma 1.11]{HDC-IP}. The cases when $a,p \in \{0,1\}$ are classical, see for instance \cite[Chapter 4]{GGBook}.
\end{proof}

\subsection{Definition of the critical point}

We begin by defining the critical point of the Blume-Capel model and the critical percolation threshold of the random cluster model.

\begin{defn} \label{def: crit point blume}
    Let $\Delta\in \mathbb{R}$. We define the critical temperature $\beta_c(\Delta)$ by 
    \begin{equs}
    \beta_c(\Delta)
    :=
    \inf\{\beta>0:\langle\sigma_0\rangle_{\beta,\Delta}^+>0\}.
    \end{equs}
\end{defn}
A consequence of the GKS inequalities \cite[Theorem 4.1.3]{GlimJaf} is that for each $\Delta$, the magnetisation is monotone increasing in $\beta$.

\begin{defn}
For $a\in (0,1)$,  define
\begin{equs}
p_c
:=
\inf\{p\in (0,1): \, \phi^1_{p,a}\left[0\longleftrightarrow\infty\right]>0\}
\end{equs}
where $\{0 \longleftrightarrow \infty\}$ denotes the event that the size of the connected component of $0$ is infinite.
Additionally, we set $p_c(0) = 1$ and $p_c(1) = p_c^\textup{Ising}(\ZZ^d)$, where the latter is the critical point for the usual random cluster model (i.e.\ FK-Ising model) on $\ZZ^d$.
\end{defn}

We now state the relation between the critical points of the dilute random cluster model and the Blume-Capel model.
\begin{prop} \label{prop: critical pts}
Let $\Delta\in \RR$ and $a_\Delta=\frac{e^{\Delta}}{1+e^{\Delta}}$. Then, 
\begin{equs}
p_c(a_\Delta)
=
1-e^{-2\beta_c(\Delta)}
\end{equs}
where we recall $\beta_c(\Delta)$ is defined in Definition \ref{def: crit point blume}.
\end{prop}
\begin{proof}
This is a direct consequence of the convergence in Proposition \ref{prop: inf vol limits} and Corollary \ref{cor: ES correlations}. 
\end{proof}

\begin{rem}
One important consequence of Proposition \ref{prop: critical pts} is that it allows to prove rigorously that the critical point of the Blume-Capel model coincides with the point at which the model undergoes a long-range order transition. This is implicitly used in Section \ref{sec: pf in 3d}. In addition, one may also show that the phase transition can be defined with respect to a uniqueness to non-uniqueness transition for the set of Gibbs measures. 
\end{rem}

\subsection{Non-triviality of critical points}

We now show that $p_c(a)$ is always nontrivial in dimensions $d\geq2$, i.e.\ $p_c(a)\in (0, 1)$, except when $a=0$. 
\begin{prop}\label{prop:existence}
For $a\in (0, 1)$, we have $p_c(a)\in (0, 1)$.
\end{prop}
\begin{proof}
Let $a \in (0,1)$ and set $p_c:=p_c(a)$. Let $p < p_c^\textup{Ising}(\ZZ^d)$. Fix $n\geq 1$ and let $\Lambda \subset \ZZ^d$ finite such that $\Lambda \supset \Lambda_n$. Then,
\begin{equs}
\varphi^1_{\Lambda,p,a}[0 \longleftrightarrow \partial \Lambda_n]
=
\Psi^1_{\Lambda,p,a} \Big[ \phi^{\textup{RC,1}}_{\Lambda,\psi,p} [0 \longleftrightarrow \partial \Lambda_n ] \Big]
\leq
\phi^{\textup{RC},1}_{\Lambda,p}[0 \longleftrightarrow \partial \Lambda_n]
\end{equs}
where the equality is by \eqref{eq: conditioning} and the inequality is by monotonicity in the domain for the usual random cluster model, see \cite[Proposition 5]{DCT}. Taking limits as $\Lambda \uparrow \ZZ^d$ and then $n \rightarrow \infty$, we obtain that $\varphi^1_{p,a}[0\longleftrightarrow  \infty]=0$. Thus, $p_c\geq p_c^{\operatorname{Ising}}(\ZZ^d)>0$.

In order to show $p_c < 1$, it is convenient to consider the edge marginal of $\varphi^1_{p,a}$.
We show that, for any $q\in (0,1)$, if $p$ is sufficiently close to $1$, 
\begin{equs}
\mathbf{P}^\textup{Ber}_q
\preceq
\Omega^1_{p,a}
\end{equs}
where $\mathbf{P}^\textup{Ber}_q$ is the law of Bernoulli bond percolation on $\E^d$ with parameter $q$. Letting $q$ be in the supercritical regime for Bernoulli bond percolation we thus get $\varphi^1_{p,a}[0 \longleftrightarrow \infty]> 0$, and so $p_c < 1$.

It is classical \cite{LSS} that the stochastic domination follows once we show that for an edge $e:=xy$, $\Omega^{\xi}_{\Lambda,p,a}[\omega_e=1]\geq \rho$ for some $\rho$ close enough to $1$, where $\xi$ is any boundary condition and $\Lambda=\{x,y\}$.
Note that, by Proposition~\ref{prop: domination} and a direct calculation, 
\begin{equs}
\Omega^{\xi}_{\Lambda,p,a}[\omega_e=1]\geq \Omega^0_{\Lambda,p,a}[\omega_e=1] =\frac{\left(\frac{a}{1-a}\right)^2\sqrt{1-p}\left(\frac{p}{1-p}\right)}{\left(\frac{a}{1-a}\right)^2\sqrt{1-p}\left(\frac{p}{1-p}\right)(1+f(p,a))}
\end{equs}
where $f(p,a)=o(1)$ as $p \rightarrow 1$. Hence, $\Omega^{\xi}_{\Lambda,p,a}[\omega_e=1]$ converges to $1$ as $p$ tends to $1$ uniformly in $\xi$, as desired. 
\end{proof}

\subsection{Continuity and monotonicity of the critical line}

In order to prove continuity of the critical line, we need to show the following strengthening of Proposition \ref{prop: domination} about stochastic domination.

\begin{prop}\label{prop:stochastic-domination-2}
Let $d\geq 1$ and $p_1>p_2$. Then, for every $a\in (0,1)$ there exists $\eps=\eps(d,p_1,p_2,a)\in (0,1-a)$ such that, for any $a_1>a-\varepsilon$ and any $a_2<a+\varepsilon$,
\begin{equs}
\varphi^1_{p_2,a_2}
\preceq
\varphi^1_{p_1,a_1}.
\end{equs}
Moreover, there exists $\delta=\delta(d,p_1,p_2)>0$ such that for any $a\in (0,1-\delta)$,
\begin{equs}
\phi^{\textup{RC},1}_{p_2}
\preceq
\Omega^{1}_{p_1,a}.
\end{equs}
\end{prop} 
\begin{proof}
We start by showing the first assertion. We prove the stochastic domination for $\Lambda\subset \ZZ^d$ finite. The assertion then follows by taking the limit as $\Lambda$ tends to $\ZZ^d$. 

We wish to find $\varepsilon>0$ such that, for any edge $xy$ of $\LL^d$ 
and any boundary conditions $\xi_1\geq \xi_2$,
\begin{equs}\label{eq:comparison}
\varphi^{\xi_2}_{\Lambda',p_2,a_2}
\preceq
\varphi^{\xi_1}_{\Lambda',p_1,a_1}
\end{equs}
where $\Lambda'=\{x,y\}$, and $a_1$ and $a_2$ are as above. Then, by Strassen's theorem \cite{Strassen}, there exists a coupling $(\theta_1,\theta_2)$ between $\varphi^{\xi_2}_{\Lambda',p_2,a_2}$ and $\varphi^{\xi_1}_{\Lambda',p_1,a_1}$ such that $\theta_1\geq \theta_2$. A standard Markov chain argument (see, for instance, \cite[Lemma 1.5]{HDC-IP}) then gives that  
\begin{equs}
\varphi^{1}_{\Lambda,p_2,a_2}
\preceq
\varphi^{1}_{\Lambda,p_1,a_1}.
\end{equs}
In addition, 
since by Proposition~\ref{prop: domination} the measure $\phi^\xi_{\Lambda',p,a}$ is stochastically increasing in $\xi$, it suffices to prove \eqref{eq:comparison} for $\xi_1=\xi_2=\xi$.

First observe that, for any increasing and non-empty event $A \in \cF$ that is measurable with respect to $(\Lambda',E_{\Lambda'})$, the function $p\mapsto \varphi^{\xi}_{\Lambda',p,a}[A]$ is analytic and non-constant. Thus, by Proposition \ref{prop: domination}, we have the strict inequality 
\begin{equs}
\varphi^{\xi}_{\Lambda',p_1,a}[A]
>
\varphi^{\xi}_{\Lambda',p_2,a}[A].
\end{equs}
Note that there are finitely many distinct maps\footnote{\label{Note} Recall our measures depend only on the state of $\partial \Lambda'$ and which vertices of $\partial \Lambda'$ are connected to each other outside of $\Lambda'$.} $a\mapsto \varphi^{\xi}_{\Lambda',p_2,a}$ ranging over all choices of $\xi$ and over all possible $A$. Hence, by continuity 
and monotonicity (see Proposition \ref{prop: domination}) 
of the map $a\mapsto \varphi^{\xi}_{\Lambda',p_2,a}[A]$, we can choose $\eps=\eps(d,p_1,p_2,a)>0$ such that
\begin{equs}
\varphi^{\xi}_{\Lambda',p_1,a_1}[A]
>
\varphi^{\xi}_{\Lambda',p_2,a_2}[A]
\end{equs}
for all $a_1>a-\varepsilon$ and $a_2<a+\varepsilon$, and uniformly over all increasing, non-empty events $A$ and boundary conditions $\xi$ on $(\Lambda', E_{\Lambda'})$. The first assertion follows.

For the second assertion, we need to show that for any $a$ close enough to $1$ and any boundary conditions $\rho\in \{0,1\}^{\E^d}$, 
\begin{equs}
\phi^{\textup{RC},\rho}_{\Lambda',p_2}
\preceq
\Omega^1_{\Lambda,p_1,a}[\cdot \mid \omega_e=\rho_e, e\in \E^d\setminus E_{\Lambda}]
\end{equs}
for the restriction of the measures on events depending on $E_{\Lambda'}$.
Note that for any increasing, non-empty event $A$ depending on $E_{\Lambda'}$, we have the strict inequality 
\begin{equs}
\phi^{\textup{RC},\rho}_{\Lambda',p_1}[A]> \phi^{\textup{RC},\rho}_{\Lambda',p_2}[A],
\end{equs}
where $\Lambda'=\{x,y\}$ for some neighbours $x,y$.
By continuity of the map $a\mapsto\Omega^{\xi}_{\Lambda',p_1,a}[A]$, where $\xi=(\kappa,\rho_1)$ with $\kappa\equiv 1$, we have that 
\begin{equs}
\Omega^{\xi}_{\Lambda',p_1,a}[A]> \phi^{\textup{RC},\rho}_{\Lambda',p_2}[A]
\end{equs}
for every $a$ close enough to $1$. 

To relate $\Omega^{\xi}_{\Lambda',p_1,a}$ to $\Omega^1_{\Lambda,p_1,a}[\cdot \mid \omega_e=\rho_e, e\in \E^d\setminus E_{\Lambda}]$, consider a finite set of vertices $R\subset \ZZ^d\setminus \Lambda'$ such that for all boundary conditions $\Omega^1_{\Lambda,p_1,a}[\cdot \mid \omega_e=\rho_e \, \forall \, e\in \E^d\setminus E_{\Lambda},\psi_u=1 \, \forall \, u\in R]=\Omega^{\xi}_{\Lambda',p_1,a}$ for the restriction of the measures on events depending on $E_{\Lambda'}$\footnote{See footnote \ref{Note}.}. Arguing as in the proof of Proposition~\ref{prop:existence}, we see that $\varphi^1_{\Lambda,p',a}[\psi_u=1 \, \forall u\in R \mid \rho_e, e\in \E^d\setminus E_{\Lambda}]$ convergences to $1$ uniformly in $\Lambda$ as $a$ tends $1$. Hence, for every $a$ close enough to $1$
\begin{equs}
\Omega^{1}_{\Lambda,p',a}[A \mid \omega_e=\rho_e, e\in \E^d\setminus E_{\Lambda}]\geq \varphi^1_{\Lambda,p',a}[\psi_u=1 \, \forall u\in R \mid \rho_e, e\in \E^d\setminus E_{\Lambda}]\Omega^{\xi}_{\Lambda',p_1,a}[A]> \phi^{\textup{RC},\rho}_{\Lambda',p}[A]
\end{equs}
for any increasing and non-empty event $A$ depending on $E_{\Lambda'}$.
The stochastic domination follows.
\end{proof}

\begin{prop}\label{prop: line cont}
The function $a\mapsto p_c(a)$ is decreasing and continuous on $[0, 1]$. 
\end{prop}
\begin{proof}
The fact that $a\mapsto p_c(a)$ is decreasing is a consequence of Proposition \ref{prop: domination}. We now turn to the continuity in the open interval $(0,1)$. Let $a\in (0,1)$ and $p_1>p_2>p_c(a)$. By Proposition \ref{prop:stochastic-domination-2}, we can find $\varepsilon>0$ such that 
\begin{equs}
\varphi^1_{p_2,a}
\preceq
\varphi^1_{p_1,a'}
\end{equs}
for any $a'> a-\varepsilon$.
Since $\varphi^1_{p_2,a}$ percolates, we have that $\varphi^1_{p_1,a'}$ percolates, hence $p_c(a')\leq p_1$. Since $p_c(a')\geq p_c(a)$ by Proposition~\ref{prop: domination}, sending $p_1$ to $p_c(a)$ we obtain that the critical probability is continuous from the left at $a$. To show continuity from the right, consider $p_2<p_1<p_c(a)$ and note that again we can find $\varepsilon>0$ such that 
\begin{equs}
\varphi^1_{p_2,a'}
\preceq
\varphi^1_{p_1,a}
\end{equs}
for any $a'< a+\varepsilon$. Then $p_c(a')\geq p_2$, and continuity from the right follows.

The case $a=1$, follows similarly using the second part of Proposition \ref{prop:stochastic-domination-2}. To prove continuity at $a=0$, we can argue as in the proof of Proposition~\ref{prop:existence} to see that for every $p\in (0,1)$ there is $\varepsilon=\varepsilon(p)>0$ such that for every $a\in (0,\varepsilon)$, $\Psi^1_{p,a}$ is dominated by subcritical Bernoulli site percolation on $\ZZ^d$. This implies that $\varphi^1_{p,a}$ is not supercritical, hence $p_c(a)\geq p$. It follows that $p_c(a)$ converges to $p_c(0)=1$ as $a$ tends to $0$.
\end{proof}

In light of the coupling between the dilute random cluster measure and the Blume-Capel model, we now obtain Theorem \ref{thm: BC phase diagram} as an immediate corollary.

\begin{proof}[Proof of Theorem \ref{thm: BC phase diagram}]

This follows from Propositions \ref{prop: critical pts} and \ref{prop: line cont}. 
\end{proof}

\begin{rem}
We believe that one can extend the arguments above to show that the function $a\mapsto p_c(a)$ is strictly decreasing for $d\geq 2$ (c.f. \cite[Theorem 5.10]{GGBook}).
\end{rem}


\subsection{Discontinuity via Pirogov-Sinai theory}

We now show that the first half of Theorem~\ref{thm: existence}, namely the existence of $\Delta^-(d)$, follows from the classical result \cite{BS} in Pirogov-Sinai theory (see also \cite[Chapter 7]{FV} for a textbook approach to Pirogov-Sinai theory applied to the Blume-Capel model). To state the latter result precisely, we introduce a slightly different parametrisation of the Blume-Capel model.

For this parametrisation, there are two parameters $\beta'>0$ and $\lambda\in \RR$.
The Hamiltonian is defined by
\begin{equs}
H^\eta_{G,\lambda}(\sigma)
=
\sum_{xy \in E} (\sigma_x-\sigma_y)^2 -\lambda\sum_{x \in V} \sigma_x^2 + \sum_{\substack{xy \in E^d \\x \in V, y \in \ZZ^d\setminus V }} (\sigma_x-\eta_y)^2,
\end{equs}
and the probability of a configuration is proportional to
\begin{equs}
e^{-\beta' H^{\eta}_{G,\lambda}(\sigma)}.
\end{equs}
Expressed in our parametrisation, this corresponds to the change of variables $\beta= 2\beta'$ and $\Delta=\beta'(\lambda-2d)$. In \cite{BS}, it is proved that there exists $\beta'_0>0$ such that for every $\beta'>\beta'_0$ the following holds: there exists $\lambda_0(\beta')>0$ such that 
\begin{equs}
\llangle \sigma_0 \rrangle^+_{\beta',\lambda}
\begin{cases}
=0, \quad \lambda < \lambda_0(\beta'),	\\
> 0, \quad \lambda \geq \lambda_0(\beta'),
\end{cases}	
\end{equs} 
where $\llangle \cdot \rrangle^+_{\beta',\lambda}$ denotes the Gibbs measure at parameters $\beta'$, $\lambda$ under the plus boundary conditions.

Let now $\beta=2\beta'$ and $\Delta_0=\beta'(\lambda_0(\beta')-2d)$. Note that either $\beta=\beta_c(\Delta_0)$ or $\beta>\beta_c(\Delta_0)$. If the latter holds, then $\beta>\beta(\tilde\Delta)$ for any $\tilde\Delta$ in a neighbourhood of $\Delta_0$ by Proposition~\ref{prop: line cont}, which implies that $\llangle \sigma_0 \rrangle^+_{\beta',\lambda}>0$ for some $\lambda<\lambda_0(\beta')$. The latter is a contradiction, hence we can deduce that $\beta=\beta_c(\Delta_0)$. This shows that the phase transition at $\Delta_0$ is discontinuous, as desired.

\section{Combinatorial mapping between the Blume-Capel model and the Ising model} \label{sec: mapping bc-ising}

Let $d\geq 2$, and let $G=(V,E)$ be a finite subgraph of $\LL^d$. We establish a correspondence\footnote{Other mappings of the Blume-Capel model to certain spin models on the same graph $\ZZ^d$ appeared in the literature, see, for instance, \cite{BLH}.} between the Blume-Capel model on $G$ and an Ising model (not necessarily ferromagnetic) on a larger graph associated to $G$, described below. An immediate byproduct of the proof is that the coupling constants of the associated Ising model are ferromagnetic provided $\Delta \geq - \log 2$. We use this extensively later in the article.

First, we lift $G$ to the graph\footnote{This graph coincides with the strong product of $G$ with the complete graph on two vertices.} $\ell(G)=(\ell(V),\ell(E))$ with vertex set $\ell(V)=V\times \{0,1\}$ and edge set $\ell(E)=E_1\cup E_2$, where 
\begin{equs}
E_1
=
\bigcup_{xy \in E} \bigcup_{i,j=0}^1 \{(x,i), (y,j)\},
\qquad
E_2
=
\bigcup_{x \in V} \{ (x,0), (x,1) \}. 	
\end{equs}
Note that $\ell(G)$ can be seen as a subgraph of the graph $\ell(\ZZ^d)$ with vertex set $\ZZ^d\times \{0,1\}$ and edge set $\EE_1^d\cup \EE_2^d$, where 
\begin{equs}
\EE_1^d
=
\bigcup_{xy \in \EE^d} \bigcup_{i,j=0}^1 \{(x,i), (y,j)\},
\qquad
\EE_2^d
=
\bigcup_{x \in \ZZ^d} \{ (x,0), (x,1) \}. 	
\end{equs}
In what follows, we abuse the notation and write $\tau^i_x$ instead of $\tau_{(x,i)}$.

Given some boundary conditions $\xi\in \{-1,0,1\}^{\ZZ^d}$, we define the boundary conditions $\ell(\xi)\in \{-1,0,1\}^{\ZZ^d\times \{0,1\}}$ by letting $\ell(\xi)_{(x,i)}=\xi_x$ for every $x\in \ZZ^d$ and $i\in \{0,1\}$.  Let also 
$T:\{\pm 1\}^V\times \{\pm 1\}^V \mapsto \{-1,0,1\}^V$ be the map $(\tau^0,\tau^1) \mapsto \sigma$, where $\sigma_x=\frac{1}{2}(\tau^0_{x}+\tau^1_{x})$. Identifying $\{\pm 1\}^V\times \{\pm 1\}^V$ with $\{\pm 1\}^{\ell(V)}$, we also view $T$ as a map defined on $\{\pm 1\}^{\ell(V)}$, i.e.\ a map defined over all possible Ising configurations on $\ell(G)$. We write $\tau$ for the spin variable of the Ising model to distinguish it from the spin variable $\sigma$ of the Blume-Capel model.

\begin{lem}\label{lem: comb mapping}
Let $\xi\in \{-1,0,1\}^{\ZZ^d}$. Let also $\mu^{\ising, \ell(\xi)}_{\ell(G),J}$ denote the Ising model on $\ell(G)$ with boundary conditions $\ell(\xi)$ and coupling constants
\begin{equs} 
J_{x,y}=J(\beta,\Delta)_{x,y}
=
\frac{\beta}4 \1_{xy \in E_1} + \frac{(\Delta + \log 2)}2 \1_{xy \in E_2}.
\end{equs}
	
Then, for every $\eta\in \{-1,0,1\}^V$, we have
\begin{equs}
\mu^{\xi}_{G,\beta,\Delta}(\sigma=\eta)=\mu^{\ising, \ell(\xi)}_{\ell(G),J}(\tau\in T^{-1}(\eta)).
\end{equs}
\end{lem}

\begin{proof}
We start by showing that 
\begin{equs} 
\begin{split}\label{bc}
\sum_{\tau \in T^{-1}(\eta)}
\prod_{uv \in E_1} 	e^{\frac{\beta}4 \tau_u \tau_v} \prod_{\substack{u\in \ell(V) \\ v \in \partial \ell(V)}} e^{\frac{\beta}4 \tau_u \ell(\xi)_v}
\prod_{uv \in E_2} e^{\frac{(\Delta + \log 2)}2 \tau_v \tau_v}
& = \\ 
e^{-\frac{(\Delta-\log 2)}2|V|}
 \prod_{xy \in E}e^{\beta \eta_x \eta_y} 
 \prod_{\substack{x\in V \\ y \in \partial V}} e^{\beta \eta_x \xi_y}
 \prod_{x \in V} e^{\Delta \eta_x}.
 \end{split}
\end{equs}
 
Note that

\begin{equs}
\begin{split} \label{bc1}
\sum_{\tau \in T^{-1}(\eta)}
\prod_{uv \in E_1} 	e^{\frac{\beta}4 \tau_u \tau_v} \prod_{\substack{u\in \ell(V) \\ v \in \partial \ell(V)}} e^{\frac{\beta}4 \tau_u \ell(\xi)_v}
\prod_{uv \in E_2} e^{\frac{(\Delta + \log 2)}2 \tau_v \tau_v}
\\ = 
\sum_{\substack{\tau^0, \tau^1 \in \{ \pm 1\}^{V} \\ T(\tau^0, \tau^1)=\eta} }
\prod_{xy \in E} \prod_{i,j=0}^1 e^{\frac{\beta}4 \tau^i_x\tau^j_y }
\prod_{\substack{x\in V \\ y \in \partial V}} \prod_{i,j=0}^1 e^{\frac{\beta}4 \tau^i_x \ell(\xi)_{(y,j)}}
\prod_{x \in V}e^{\frac{(\Delta+\log 2)}{2}\tau^0_x \tau^1_x}
\\ =
\sum_{\substack{\tau^0,\tau^1 \in \{\pm 1\}^V \\ T(\tau^0, \tau^1)=\eta}}
\prod_{xy \in E}e^{\beta \eta_x \eta_y} \prod_{\substack{x\in V \\ y \in \partial V}}e^{\beta \eta_x \xi_y} \prod_{x \in V} \left(e^{(\Delta + \log 2)\eta_x^2} e^{-\frac{(\Delta+\log 2)}2}\right)
\\ =
e^{-\frac{(\Delta+\log 2)}2|V|} \Bigg(\sum_{\substack{\tau^1,\tau^2 \in \{\pm 1\}^V \\ T(\tau^1, \tau^2)=\eta}} 1\Bigg)\prod_{xy \in E}e^{\beta \eta_x \eta_y} \prod_{\substack{x\in V \\ y \in \partial V}}e^{\beta \eta_x \xi_y} \prod_{x \in V} e^{(\Delta + \log 2)\eta_x^2},
\end{split}
\end{equs}
where in the second equality, we used that 
\begin{equs}
\prod_{i,j=0}^1 e^{\frac{\beta}4 \tau^i_x\tau^j_y }=e^{\beta \eta_x \eta_y}, \qquad \prod_{i,j=0}^1 e^{\frac{\beta}4 \tau^i_x \ell(\xi)_{(y,j)}}=e^{\beta \eta_x \xi_y}, \qquad \text{ and } \qquad \tau^0_x \tau^1_x=2\eta_x^2-1.
\end{equs}

It is not hard to see that $T$ is injective on the vertices $x\in V$ such that $\tau^0_x = \tau^1_x$ and is $2$-to-$1$ on the vertices $x\in V$ such that $\tau^0_x \neq \tau^1_x$. Thus, 
\begin{equs}
   \sum_{\substack{\tau^1,\tau^2 \in \{\pm 1\}^V \\ T(\tau^1, \tau^2)=\eta}} 1  =2^{\sum_{x\in V}(1- \eta_x^2)}.
\end{equs}
Hence,
\begin{equs}
\eqref{bc1}
&=
e^{-\frac{(\Delta-\log 2)}2|V|}
 \prod_{xy \in E}e^{\beta \eta_x \eta_y} 
 \prod_{\substack{x\in V \\ y \in \partial V}}e^{\beta \eta_x \xi_y}
 \prod_{x \in V} e^{\Delta \eta_x^2}
\end{equs}
which establishes \eqref{bc}.

Similarly, we obtain 
\begin{equs}
Z^\ising_{\ell(G),J}=e^{-\frac{(\Delta-\log 2)}2|V|}Z_{G,\beta,\Delta},
\end{equs}
and the desired assertion follows readily.
\end{proof}

\begin{rem}
$J$ is ferromagnetic provided $\Delta \geq -\log 2$. We note that the edges between the vertices $(x,0)$ and $(x,1)$ are antiferromagnetic for $\Delta < -\log 2$ and their interaction becomes stronger as $\Delta \rightarrow -\infty$, i.e.\ as $0$ becomes more likely.
\end{rem}

\begin{cor}\label{corollary: correlations}
Let $G=(V,E)$ be a finite subgraph of $\ZZ^d$, and let $J$ be as in Lemma \ref{lem: comb mapping}. For any non-empty set $A\subset V$, any indices $i_x\in \{0,1\}$, $x\in A$, and any boundary conditions $\xi\in \{-1,0,1\}^V$,
\begin{equs}
\langle \prod_{x\in A}\sigma_x  \rangle^{\xi}_{G,\beta,\Delta}
=
\langle \prod_{x\in A} \tau^{i_x}_x  \rangle^{\ising,\ell(\xi)}_{\ell(G),J}.
\end{equs}
\end{cor}

\begin{proof}
Observe that, by symmetry,
\begin{equs}
\langle \prod_{x\in A}\tau^{i_x}_x  \rangle^{\ising,\ell(\xi)}_{\ell(G),J}
=
\frac{1}{2^{|A|}}\sum_{j:A\mapsto \{0,1\}}\langle \prod_{x\in A}\tau^{j_x}_x \rangle^{\ising,\ell(\xi)}_{\ell(G),J}
=
\Big\langle \prod_{x\in A}\frac{\tau^0_x+\tau^1_x}2 \Big\rangle^{\ising,\ell(\xi)}_{\ell(G),J}.  \end{equs}
The latter is equal to $\langle \prod_{x\in A}\sigma_x  \rangle^{\xi}_{G,\beta,\Delta}$ by Lemma~\ref{lem: comb mapping}.
\end{proof}

\section{Tricritical point in $d\geq 3$ via the infrared bound}\label{sec: pf in 3d}

In this section, we prove Theorem \ref{thm: existence} in $d\geq 3$ for $\Delta^+ = -\log 2$. The key idea is to use the mapping described in Section \ref{sec: mapping bc-ising} to map the Blume-Capel model to a ferromagnetic Ising model on $l(\ZZ^d)$. Since this graph is transitive and amenable, we can then apply the main result of \cite{ADCS} concerning continuity of the phase transition for the Ising model, whose hypothesis is guaranteed in our setting by the infrared bound \cite{FSS}. Strictly speaking, because the mapping does not give us an Ising model on $\mathbb{Z}^d$, we use the generalisation obtained by \cite{AR} rather than \cite{ADCS}.

First, we claim that in $d\geq 3$ the two-point correlations $\langle \tau^1_x \tau^1_y \rangle^{\ising,0}_{J}$ decay, in an averaged sense, to $0$ for every $\beta\leq \beta_c(\Delta)$. 

\begin{lem} \label{lem: decay in 3d}
Let $d \geq 3$. Then, for any $\Delta \in \RR$ and $\beta \leq \beta_c(\Delta)$,
\begin{equs}
\inf_{B \subset \ZZ^d, |B|<\infty} \frac{1}{|B|^2} \sum_{ x,y \in B} 	\langle \tau^1_x \tau^1_y \rangle^{\ising,0}_{J(\beta,\Delta)}
=
0
\end{equs}
where $J(\beta,\Delta)$ is as in Lemma \ref{lem: comb mapping}.
\end{lem}

\begin{proof}
The proof is straightforward adaptation of the proof of \cite[Corollary 1.4]{ADCS} and is based on the celebrated infrared/Gaussian domination bound of \cite{FSS}, which applies to reflection positive models (such as the nearest-neighbour Blume-Capel model, see \cite[Section 3.1]{Biskup}) on $\ZZ^d$ for $d \geq 3$. We omit the details.
\end{proof}

Our aim now is to use the above result to prove Theorem~\ref{thm: existence} in $d\geq 3$. For that we will use a result of Raoufi \cite[Theorem 1]{AR}, which states that for the Ising model on any vertex-transitive amenable graph $G$, and for any $\beta>0$, we have 
\[
\langle \cdot \rangle^{\ising,0}_{\beta,G}=\lambda\langle \cdot \rangle^{\ising,+}_{\beta,G}+(1-\lambda)\langle \cdot \rangle^{\ising,-}_{\beta,G}
\]
for some $\lambda\in [0,1]$. Since for the $2$-point correlations we have $\langle \tau_x \tau_y \rangle^{\ising,+}_{\beta,G}=\langle \tau_x \tau_y \rangle^{\ising,-}_{\beta,G}$ by symmetry, it follows that
\begin{equs}\label{eq: raoufi}
\langle \tau_x \tau_y \rangle^{\ising,0}_{\beta,G}=\langle \tau_x \tau_y \rangle^{\ising,+}_{\beta,G}    
\end{equs}
Below we will argue that this is sufficient to prove continuity of the phase transition by using correlation inequalities and translation invariance.

We now recall the definition of amenability. Let $G$ be a countable locally finite graph. We say that $G$ is amenable if
\begin{equs}
\inf_{A \subset V, |A| < \infty} \frac{|\partial A|}{|A|}
=
0.
\end{equs}

\begin{lem} \label{lem: amenable}
The graph $l(\ZZ^d)$ is amenable and vertex-transitive. 
\end{lem}

\begin{proof}
It is easy to see that $\ell(\ZZ^d)$ is the Cayley graph of the group $\ZZ^d\times\{0,1\}$ with respect to the generating set $\{(e_1,0),\ldots,(e_d,0),(e_1,1),\ldots,(e_d,1),(0,1)\}$, where $\{e_1,\ldots,e_d\}$ is the standard generating set of $\Z^d$. This readily implies that $\ell(\ZZ^d)$ is vertex-transitive. Moreover, $\ell(\ZZ^d)$ is amenable because e.g.\ the boxes $F_n:=\Lambda_n\times \{0,1\}$ form a F{\o}lner sequence, i.e.\ 
\begin{equs}
\lim_{n\to \infty} \frac{|\partial F_n|}{|F_n|}=0.    
\end{equs}
\end{proof}

We are now ready to prove  Theorem \ref{thm: existence} in $d \geq 3$.

\begin{proof}[Proof of Theorem \ref{thm: existence} in $d \geq 3$]
With Lemma \ref{lem: decay in 3d} in hand, the proof follows essentially from applying \cite[Theorem 1.2]{ADCS}. The only caveat is that the result is stated for Ising models on $\ZZ^d$. Thus, we proceed as follows: by Lemma \ref{lem: amenable}, $l(\ZZ^d)$ is amenable and transitive. Thus, it follows from \cite[Proposition 1]{AR} that $\langle \tau^1_0 \tau^1_x\rangle^{\ising,+}_{J(\beta,\Delta)}=\langle \tau^1_0 \tau^1_x\rangle^{\ising,0}_{J(\beta,\Delta)}$ for every $\beta>0$ and $\Delta\geq-\log{2}$. (Recall that $\Delta \geq -\log 2$ is required for the Ising model on $l(\mathbb{Z}^d)$ to be ferromagnetic.)

Recall that Griffiths' inequality \cite{Griffiths} for the Ising model implies that
\begin{align*}
\langle \tau^1_x \tau^1_y \rangle^{{\rm Ising}, +}_{J(\beta_c(\Delta),\Delta)}
\geq
\langle \tau^1_x \rangle^{{\rm Ising}, +}_{J(\beta_c(\Delta),\Delta)}
\langle \tau^1_y \rangle^{{\rm Ising}, +}_{J(\beta_c(\Delta),\Delta)}
\end{align*}
Thus, by the result of the previous paragraph, the above display, and translation invariance, we have that
\begin{equs}
\langle \tau^1_x \tau^1_y \rangle^{\ising,0}_{J(\beta_c(\Delta),\Delta)}
=
\langle \tau^1_x \tau^1_y \rangle^{{\rm Ising}, +}_{J(\beta_c(\Delta),\Delta)}
\geq     
\left(\langle \tau^1_0\rangle^{\ising,+}_{J(\beta_c(\Delta),\Delta)} \right)^2.
\end{equs}
Hence by taking C\'esaro averages and using Lemma \ref{lem: decay in 3d} we obtain
\begin{equs}
\left(\langle \tau^1_0 \rangle^{\ising,+}_{J(\beta_c(\Delta),\Delta)} \right)^2\leq \inf_{B \subset \ZZ^d\times\{0,1\}, |B|<\infty} \frac{1}{|B|^2} \sum_{ x,y \in B} 	\langle \tau^1_x \tau^1_y \rangle^{\ising,0}_{J(\beta_c(\Delta),\Delta)}
=
0.
\end{equs}
Thus, $\langle \sigma_0 \rangle^+_{\beta_c(\Delta),\Delta}=0$.
\end{proof}

\section{Quantitative analysis of crossing probabilities in $d=2$}\label{sec: quantitative}

Given a rectangle $R=[a,b]\times [c,d] \subset \RR^2$, we define the bottom, top, left, and right sides of $R$ respectively by
\begin{equs}
\Bot[R]
&=	
[a,b] \times \{c\},
\quad 
\Top[R]
=
[a,b] \times \{ d \},
\\
\Left[R]
&=
\{ a \} \times [c,d],
\quad
\Right[R]
=
\{ b \} \times [c,d].
\end{equs} 
\noindent We identify $\LL^2$ with its natural embedding in $\RR^2$ and identify the sides of $R$ defined above with their corresponding subgraphs of $\LL^2$. We denote by $H_R$ the event that $\omega\cap R$ contains a path of open edges from $\Left[R]$ to $\Right[R]$, and we call such a path a {\em horizontal crossing}.
Similarly, we denote by $V_R$ the event that $\omega\cap R$ contains a path of open edges from $\Bot[R]$ to $\Top[R]$, and we call such a path a {\em vertical crossing}. When $R= [-n,n]^2$ for some $n \geq 1$, we simply write $H_n$ and $V_n$, respectively.

The following quadrichotomy on crossing probabilities is at the heart of our approach in $d=2$.
\begin{prop}\label{prop:quad}
Fix $p,a \in (0,1)$. Then exactly one of the following properties is satisfied for some $c>0$:
\begin{description}
\item[(SubCrit)] \qquad $\varphi_{\Lambda_{2n}}^1[H_n]\leq e^{-cn}$, $\qquad \forall n \in \NN$;
\item[(SupCrit)] \qquad  $\varphi_{\Lambda_{2n}}^0[H_n]\geq 1-e^{-cn}$, $\qquad \forall n \in \NN$;
\item[(ContCrit)] \qquad \hspace{-0.6em}
$c\leq \inf_{\xi} \varphi_{\Lambda_{2n}}^\xi[H_n]\leq \sup_{\xi} \varphi_{\Lambda_{2n}}^\xi[H_n]\leq 1-c$, $\qquad \forall n \in \NN$;
\item[(DiscontCrit)]  \quad \hspace{-0.7em}
$\varphi_{\Lambda_{2n}}^0[H_n]\leq e^{-cn} 
\; \text{ and } \;
\varphi_{\Lambda_{2n}}^1[H_n]\geq 1-e^{-cn}$, $\qquad \forall n \in \NN$.
\end{description}
\end{prop}
\noindent The proof of Proposition \ref{prop:quad} is based on the approach of \cite{DCT}. The behaviour in the off-critical regimes \textbf{(SubCrit)} and \textbf{(SupCrit)} follow from standard coarse-graining arguments. The critical behaviour is more subtle to analyse. Our main tool consists of two sets of inequalities that quantify how macroscopic connectivity (or lack of connectivity) changes from scale to scale. The first inequality encodes an approximate duality between crossing and non-crossing events at similar scales and is formalised in Lemmas \ref{lem: dual relation of strip densities}. The second inequality is a renormalisation inequality that compares macroscopic connectivity at scales that are triadically separated, and it is formalised in Lemma \ref{lem: renormalisation inequality for strip densities}. A key input to the renormlisation is a Russo-Seymour-Welsh theory developed for the dilute random cluster model (see the next subsection). Our goal in this subsection is to state these inequalities precisely and then show how they can be used to prove Proposition \ref{prop:quad}. Then in the remainder of this section we prove Proposition \ref{prop:quad} and detail some of its applications.

We begin by introducing the relevant observables. Roughly speaking, they consist of two sequences $(p_n)_{n \in \N}$ and $(q_n)_{n \in \N}$ that, as mentioned above, for each $n\in\N$ track, respectively, the macroscopic connectivity and lack of connectivity at that scale. For example, one could define $p_n$ as the existence of a circuit of open edges surrounding an annulus of order $n$ and $q_n$ as the non-existence of a path of open edges connecting the inner boundary of the annulus to the outer boundary. Motivated by \cite{DCT}, we instead choose to track densities of horizontal crossings and lack of vertical crossings of very long, thin strips. 

\begin{defn}
For $n \in \NN$ we define the strip crossing density and dual strip crossing density at scale $n$ respectively by,
\begin{equs}
p_n
&:=
\limsup_{\alpha \rightarrow \infty} \Big( \varphi^0_{[0,\alpha n]\times[-n,2n]} [ H_{[0,\alpha n]\times[0,n]} ] \Big)^{1/\alpha}
\\
q_n
&:=
\limsup_{\alpha \rightarrow \infty} \Big( \varphi^1_{[0,\alpha n]\times[-n,2n]} [ V_{[0,\alpha n]\times[0,n]}^c] \Big)^{1/\alpha}.
\end{equs}
\end{defn}

The quantity $p_n$ is called a strip crossing density for the following reason. Naively, one could imagine deterministically gluing crossings of rectangles of aspect ratio $\alpha_0$ where $\alpha_0 < \alpha$ approximately $\alpha/\alpha_0$ times to create a crossing of a rectangle of aspect ratio $\alpha$. Then since these events are increasing the FKG inequality would give a lower bound on the probability of crossing the rectangle of aspect ratio $\alpha$ in terms of the probability of crossing the rectange of aspect ratio $\alpha_0$ which scales as a power which is proportional to $\alpha$. The constant of proportionality would be uniform over $\alpha > \alpha_0$ and hence, after taking the $1/\alpha$-root and the limit as $\alpha \rightarrow \infty$, the strip crossing density would measure (if the inequalities were equalities) a certain power of the macroscopic crossing at scale $\alpha_0$. We briefly postpone the discussion of why $p_n$ is a convenient quantity to keep track of as opposed to a crossing of a rectangle of fixed aspect ratio (or of circuits in annuli) to give an interpretation of the quantity $q_n$.

The quantity $q_n$ keeps track of the density of lack of vertical crossings, but it is perhaps more useful to interpret this by using planar duality. Let $(\LL^2)^* = \LL^2 + (1/2, 1/2)$ be the dual lattice, which we identify with its natural embedding in $\RR^2$. Given $\omega$ on $\LL^2$, there is an associated dual configuration $\omega^*$ constructed by declaring that an edge in the dual graph is open if and only if it crosses a closed edge in the primal graph. Thus, if $V_R^c$ occurs for some rectangle $R$, then this implies a horizontal crossing in a suitable rectangle in the dual lattice. As such $q_n$ keeps track of the density of dual horizontal crossings.

We now explain why the strip densities $(p_n)_{n \in \N}$ and $(q_n)_{n \in \N}$ are the relevant observables in our setting. In order to motivate this, first consider the usual random cluster model on $\Z^2$. Renormalisation inequalities in this setting were first derived in \cite{DCST17}, where the analogous quantities to $p_n$ and $q_n$ were roughly given by the existence of circuits of open edges in annuli of scale $n$ in the primal and dual graph, respectively. In this setting the model is \emph{self-dual} at the critical point and these two quantities are essentially the same (again, we emphasise, at the critical/self-dual point). This is used extensively to derive analogous versions of the renormalisation inequalities. On more general infinite connected biperiodic planar graphs with sufficient dihedral symmetry (i.e.\ $\pi/2$ rotations and reflections) but where there is no self-duality, it is not clear how these quantities are related at a fixed scale $n$. One of the insights of \cite{DCT} is that the strip crossing and dual strip crossing densities can be related at different scales (which motivates why we take the limit as the aspect ratio goes to infinity, so that comparisons can be made between scales). Since in our setting there is no obvious self-duality for the dilute random cluster model (even on the self-dual graph $\Z^2$), we adopt this multiscale approach.

Our first set of inequalities relate the strip crossing densities and dual strip crossing densities on different scales. We omit its proof since the geometric arguments developed in \cite[Section 5.1]{DCT} adapt with minor modifications to our setting (we opt instead to explain these adaptations in detail when dealing with the renormalisation inequalities).

\begin{lem}[Duality of strip densities] \label{lem: dual relation of strip densities}
Let $(a,p)$ be such that neither {\bf (SubCrit)} nor { \bf (SupCrit)} occur. Then there exists $C>0$ such that, for all integers $\lambda \geq 2$ and $n \in 9\NN$,
\begin{equs}
p_{3n} \geq \frac{1}{\lambda^C} q_n^{3+3/\lambda},
\qquad
q_{3n} \geq \frac{1}{\lambda^C} p_n^{3+3/\lambda}.
\end{equs}
\end{lem}

We now turn to the second set of key inequalities that constitute the renormalisation inequalities. We defer its proof to Subsection \ref{subsec: proof of renorm}.

\begin{lem}[Renormalisation of strip densities] \label{lem: renormalisation inequality for strip densities}
Let $(p,a)$ be such that neither {\bf (SubCrit)} nor {\bf (SupCrit)} occur. Then there exists $C>0$ such that for all integers $\lambda \geq 2$ and $n \in 9\NN$,
\begin{equs}
p_{3n} \leq \lambda^C p_n^{3-9/\lambda},
\qquad
q_{3n} \leq \lambda^C q_n^{3-9/\lambda}.
\end{equs}
\end{lem}
\noindent

Before proving Lemmas~\ref{lem: dual relation of strip densities} and \ref{lem: renormalisation inequality for strip densities}, we will show that they imply the following: either (i) $\inf_{n \in 9\NN} p_n > 0$ and $\inf_{n \in 9\NN} q_n > 0$; or, (ii) there exists $c> 0$ such that $p_n \leq e^{-cn}$ and $q_n \leq e^{-cn}$ for all $n \in 9\NN$.

It suffices to show that (ii) holds whenever $\inf_{n \in 9\NN} p_n = 0$ or $\inf_{n \in 9\NN} q_n = 0$. First, note that by Lemma~\ref{lem: dual relation of strip densities} we have that in this case both $\inf_{n \in 9\NN} p_n$ and $\inf_{n \in 9\NN} q_n$ are equal to $0$. Then applying Lemma~\ref{lem: renormalisation inequality for strip densities} for $\lambda = 9$ we obtain stretch exponential decay along a geometric subsequence. Note that, by the finite energy property  -- see Proposition \ref{prop: finite energy} -- we may open primal or dual paths of edges with exponential cost in the length of the path.  Hence there exists $c'>0$ such that $p_n \geq e^{-c'n}$ and $q_n \geq e^{-c'n}$ for every $n\geq 1$, and so $\sup_n p_n^{-9/n} <\infty$ and $\sup_n q_n^{-9/n} <\infty$. This implies that for $\lambda=n$, we have that $p_{3n} \leq C_1 n^{C_2} p_n^3$ and $q_{3n} \leq C_1 n^{C_2} q_n^3$, for some constants $C_1>0$ and $C_2>0$. Thus, for some $C_3$ sufficiently large, the sequences $(n^{C_3} p_n)_{n\geq 1}$ and $(n^{C_3} q_n)_{n\geq 1}$ decay exponentially, and hence both $(p_n)_{n\geq 1}$ and $(q_n)_{n\geq 1}$ decay exponentially.

We will show that one can take $n \in \NN$ instead of $n \in 9\NN$. To see this, note that by inclusion of events and \eqref{eq:MON}, for every $m\geq n$ we have
\begin{equs}
    \varphi^0_{[0,\alpha m]\times[-m,2m]} [ H_{[0,\alpha m]\times[0,m]} ]\geq \varphi^0_{[0,\alpha m]\times[-m,2m]} [ H_{[0,\alpha m]\times[0,n]} ]\geq \varphi^0_{[0,\alpha' n]\times[-n,2n]} [ H_{[0,\alpha' n]\times[0,n]} ]
\end{equs}
and
\begin{equs}
    \varphi^1_{[0,\alpha m]\times[-m,2m]} [ V^c_{[0,\alpha m]\times[0,m]} ]\geq \varphi^1_{[0,\alpha m]\times[-m,2m]} [ V^c_{[0,\alpha m]\times[0,n]} ]\geq \varphi^1_{[0,\alpha' n]\times[-n,2n]} [ V^c_{[0,\alpha' n]\times[0,n]} ],
\end{equs}
where $\alpha'=\alpha m/n$.
Hence $p_m\geq p_n^{m/n}$ and $q_m\geq q_n^{m/n}$ for every $m\geq n$.

We now assume Lemmas \ref{lem: dual relation of strip densities} and \ref{lem: renormalisation inequality for strip densities} in order to prove Proposition \ref{prop:quad}. The proof assuming these lemmas is based on \cite[Section 5.4]{DCT}.

\begin{proof}[Proof of Proposition \ref{prop:quad} assuming Lemmas~\ref{lem: dual relation of strip densities} and \ref{lem: renormalisation inequality for strip densities}]

First note that {\bf(SubCrit)} and {\bf (SupCrit)} are disjoint. Thus, it suffices to show that if {\bf non(SubCrit)} and {\bf non(SupCrit)} occur, then either {\bf (ContCrit)} or {\bf (DiscontCrit)} occurs. Assume that {\bf non(SubCrit)} and {\bf non(SupCrit)} occur. 

Arguing as in \cite[Section 5.4]{DCT}, we find that: if (i) holds, then {\bf (ContCrit)} occurs; whereas, if (ii) holds, then  {\bf (DiscontCrit)}  occurs. The adaptation of these arguments to our setting is mild, but requires some care concerning the dual crossings. Let us first show that if (ii) holds, then {\bf (DiscontCrit)} occurs. It is sufficient to show that either the crossing of squares or the dual crossing of squares decays exponentially. We treat the case of dual crossings since the other case follows by similar arguments. Note that by a union bound there exist $x \in \{-n-1/2\}\times[-n+1/2,n-1/2]$ and $y \in \{n+1/2\}\times[-n+1/2,n-1/2]$ such that
\begin{equs}
\varphi^1_{\Lambda_{2n}}[x \overset{*}\longleftrightarrow y]
\geq
\frac{1}{4n^2} \varphi^1_{\Lambda_{2n}}[V_{n}^c]	
\end{equs}
where $\overset{*}\longleftrightarrow $ denotes connection in the dual configuration $\omega^*$.
For $k \in \N$ let $x_k = x + (4kn+1,0)$. Let $S'_n=\RR \times [-n,n]$ denote the infinite strip of width $2n$. By reflection across $y$ and its translates (here we use the translation symmetry and reflection symmetries of the measure), the FKG inequality, and monotonicity in boundary conditions we have that
\begin{equs}
\varphi_{S'_{3n}}^1[V^c_{[-n,(4k+1)n]\times[-n,n]}]
\geq
\left(\frac{1}{4n^2} \varphi^1_{\Lambda_{2n}}[V_n^c]\right)^{k}.
\end{equs}
The choice of pushing the wired boundary conditions to $S'_{3n}$ is motivated since we want to see the event $q_{2n}$ appearing. On the other hand, by finite energy there exists $c'>0$ such that
\begin{equs}
\varphi_{S'_{3n}}^1[V^c_{[-n,(4k+1)n]\times[-n,n]}]
\leq
e^{c'n}
\varphi^1_{[-n,(4k+1)n]\times[-3n,3n]}[V^c_{[-n,(4k+1)n]\times[-n,n]}]	
\end{equs}
Shifting by $+n$ in the horizontal and vertical directions in order to get the correct domains, we may insert this bound into the penultimate display and take $k \rightarrow \infty$ to yield
\begin{equs}
4n^2q_{2n}^4
\geq
\varphi^1_{\Lambda_{2n}}[V_n^c].
\end{equs}
Since $q_n$ decays exponentially, we obtain that the righthand side decays exponentially, as desired.

Let us now assume (i). We wish to show {\bf (ContCrit)} holds. It is sufficient to show that for any $\rho > 0$ there exists $c_\rho \in (0,1/2)$ such that, for any boundary condition $\xi$,
\begin{equs}\label{eq: box crossing}
\varphi^\xi_{[-n,(\rho+1)n]\times[-n,2n]}\left[ H_{[0,\rho n]\times[0,n]}\right]
\in
(c_\rho,1-c_\rho). 
\end{equs}
Let ${\rm mix}$ denote the boundary conditions on $[-n,(\rho+1)n]\times[-n,2n]$ that are wired on the left, wired on the right, and $0$ elsewhere. Let $K = [-2n/3,-n/3]\times[-n,2n]$ and $K' = [\rho n +n/3,\rho n + 2n/3]\times[-n,2n]$. By the comparison of boundary conditions, we have
\begin{equs}
\varphi^{\rm mix}_{[-n,(\rho+1)n]\times[-n,2n]}\left[H_{[0,\rho n]\times[0,n]} \cap H^c_{K} \cap H^c_{K'}\right]
&\leq	
\varphi^{\rm mix}_{[-n,(\rho+1)n]\times[-n,2n]}\left[H_{[0,\rho n] \times[0,n]} \mid H^c_{K} \cap H^c_{K'}\right]
\\
&\leq
\varphi^0_{[-n,(\rho+1)n],[-n,2n]}[H_{[0,\rho n]\times [0,n]}].
\end{equs}
We claim that 
\begin{equs}
\begin{split} \label{eq: box cross 1}
\varphi^{\rm mix}_{[-n,(\rho+1)n]\times[-n,2n]}[H_{[0,\rho n]\times[0,n]}]
&\geq
p_n^{\rho+2},
\end{split}
\\
\begin{split} \label{eq: box cross 2}
\varphi^{\rm mix}_{[-n,(\rho+1)n]\times[-n,2n]}[H^c_K \cap H^c_{K'} \mid H_{[0,\rho n]\times[0,n]}]
&\geq
q_{n/3}^{18}.
\end{split}
\end{equs}
Assuming this claim, the lower bound in \eqref{eq: box crossing} follows. In order to obtain the upper bound, we apply the same reasoning to the dual crossings (we omit this since the arguments and the required claims are similar).

In order to prove \eqref{eq: box cross 1}, note that for any $\alpha \in \N$ we have 
\begin{equs}
\varphi^0_{S_n}\left[ H_{[0,(\rho+2)\alpha n]\times[0,n]} \right]
\leq	
\varphi^{\rm mix}_{[-n,(\rho+1)n]\times[-n,2n]}[H_{[0,\rho n]\times[0,n]}]^{\alpha-1}.
\end{equs}
To see this, note that the crossing event $ H_{[0,(\rho+2)\alpha n]\times[0,n]}$ is included in the event that $\alpha - 1$ disjoint translates of crossings in $[0, \rho n]\times[0,n]$ occur. The inequality then follows by using the spatial Markov property and monotonicity in boundary conditions. Using finite energy as above to get the correct boundary conditions, and then dividing by $\alpha-1$ and taking limits as $\alpha \rightarrow \infty$ yields the desired result. Let us now consider \eqref{eq: box cross 2}. By comparison of boundary conditions we have
\begin{align*}
\varphi^{\rm mix}_{[-n, (\rho+1)n]\times[-n,2n]}[H_K^c \cap H_{K'}^c \mid H_{[0,\rho n]\times[0,n]}]
\geq
\varphi^{\rm mix}_{[-n, 0]\times[-n,2n]}[H_K^c] \, \varphi^{\rm mix}_{[\rho n,(\rho+1)n]\times [-n,2n]}[H_{K'}^c]. 
\end{align*}
Above, ${\rm  mix}$ again means wired on the left and right sides of the rectangles, and free elsewhere. Note that by translation invariance and $\pi/2$ rotation invariance, each of the probabilities on the righthand side coincides with
\begin{align*}
\varphi^{\rm mix^\perp }_{[0,3n]\times [-n/3,2n/3]}[V_{[0,3n]\times[0,n/3]}^c],
\end{align*}
where ${\rm mix^\perp}$ are the boundary conditions that are wired on the top and bottom, and free on the sides. Arguing as in the proof of \eqref{eq: box cross 1}, we get that
\begin{equs}
\varphi^{\rm mix^\perp}_{[0,\rho n/3]\times[-n/3,2n/3]}[V^c_{[0,\rho n]\times[0,n/3]}]
&\geq
q_{n/3}^{\rho},
\end{equs}
from which \eqref{eq: box cross 2} follows.

\end{proof}

\subsection{Russo-Seymour-Welsh theory} \label{subsec: proof of RSW}

A central tool in the analysis of crossing probabilities for 2D percolation models is the Russo-Seymour-Welsh (RSW) theory. Loosely speaking, this allows to bound the probability of long horizontal crossings from below by short vertical crossings. The RSW estimates apply to the wired and free measures in infinite volume, and also the following strip measures: for $m_1,m_2 \in \N$ let $S_{m_1,m_2} = \Z \times [-m_1,2m_2]$. In the case $m_1=m_2$ we simply write $S_m = \RR \times [-m,2m]$ to denote the infinite strip of width $3m$. We write $\varphi^{1}_{S_{m_1,m_2}}, \varphi^{0}_{S_{m_1,m_2}}$, and $\varphi^{0/1}_{S_{m_1,m_2}}$ to denote the wired, free, and mixed boundary conditions (wired on bottom and free on top) random cluster measures on $S_{m_1,m_2}$, respectively.
\begin{prop}[RSW estimate] \label{prop: RSW}
For any $\rho >0$  there exists $c_\rho > 0$ such that for any $(p,a) \in (0,1)$, and for all $n \geq \frac 1{c_\rho}$ and $m_1,m_2 \geq n$,
\begin{equs} \label{eq: RSW} \tag{\textbf{RSW}}
\varphi^{0/1}_{S_{m_1,m_2}}[H_{[0,\rho n] \times [0,n]}]
\geq 
c_\rho \Big( \varphi^{0/1}_{S_{m_1,m_2}}[V_{[0,\rho n] \times [0,n]}] \Big)^{1/c_\rho}, 
\end{equs}
and
\begin{equs} \label{eq: RSWstar} \tag{\textbf{RSW*}}
\varphi^{0/1}_{S_{m_1,m_2}}[V_{[0,\rho n] \times [0,n]}^c]
\geq 
c_\rho \Big( \varphi^{0/1}_{S_{m_1,m_2}}[H_{[0,\rho n] \times [0,n]}^c] \Big)^{1/c_\rho}. 
\end{equs}
Above we could also replace $\varphi^{0/1}_{S_{m_1,m_2}}$ by either $\varphi^1_{p,a}$ or $\varphi^0_{p,a}$.
\end{prop}

\noindent The proof of Proposition \ref{prop: RSW} is  based on a symmetric domain argument developed in \cite{DCT}. We note that the full plane measures $\varphi^1_{p,a}$ and $\varphi^0_{p,a}$, since crossing events only depend on the edge-marginal of these dilute random cluster measure, satisfy the assumptions of the general RSW theory developed in \cite{KT23}. The case of the infinite strip, however, is not covered by that theory. In the proof below we assume without loss of generality that $m_1=m_2=n$.

We begin with some geometric setup. Let $R=[a,b]\times [c,d]$ be a rectangle in $\RR^2$. For $E,F \in \{ \Bot[R], \Top[R], \Left[R], \Right[R] \}$ and a rectangle $S \subset R$, we define the following connection events:
\begin{itemize}
\item we write $E \overset{S}{\conn} F$ if there exists a path in $\omega$ which lies in $S$, and intersects both $E$ and $F$;
\item we write $E \overset{*,S}{\conn} F$ if there exists a path in $\omega^*$ which, except from its first and last edge, lies in $S$, and intersects both $E$ and $F$.
\end{itemize}
If $S=R$, then we drop $S$ from the notation. We write $H_R$ and $V_R$ to denote the events that $\Left[R] \conn \Right[R]$ and $\Bot[R] \conn \Top[R]$, respectively. We write $H^*_R$ and $V^*_R$ to denote the events that $\Left[R] \connstar \Right[R]$ and $\Bot[R] \connstar \Top[R]$, respectively. Let $k = \lceil n/50 \rceil$. Define $R_0=\{-17k,\dots,18k\}\times\{0,\dots, n\}$ and $S_0 = \{0,\dots,k\}\times\{0\}$. Let $R_j, S_j$ denote the translates $R_0+(jk,0)$ and $S_0 +(jk,0)$, respectively. 

We consider the following bridge events:
\begin{equs}
A_j
&=
\{ S_j \overset{R_j \cup R_{j+4}}{\conn} S_{j+2} \cup S_{j+4} \}
\\
A_j^*
&=	
\{ S_j \overset{*, R_j \cup R_{j+4}}{\conn} S_{j+2} \cup S_{j+4}\}.
\end{equs}
It is easy to see that these bridge events may be glued together, using the FKG inequality, to create long horizontal crossings and dual crossings. The RSW estimates are then a consequence of the following bounds on the probability of bridge events.

\begin{prop} \label{prop: bridge events}
There exists $c_1>0$ such that, for every $\lambda \geq 1$ and $n \in \NN$,
\begin{equs}
\varphi[A_0]
\geq
\frac{c_1}{\lambda^3} \varphi[ V_{[0,\lambda n]\times[0,n]} ]^3
\text{ and }
\varphi[A_0^*]
\geq
\frac{c_1}{\lambda^3} \varphi [ V^*_{[0,\lambda n]\times[0,n]} ]^3.
\end{equs}	
\end{prop}

\begin{proof}[Proof of Proposition \ref{prop: RSW} assuming Proposition \ref{prop: bridge events}]
We prove the assertion for primal crossings. The case of dual crossings follows similarly. Note that if the event $A_j$ occurs for every $-1\leq j \leq 50\rho$, then $H_{[0,\rho n]\times[0,n]}$ occurs. Using Proposition~\ref{prop: bridge events} and \eqref{eq:FKGRC} we obtain 
\begin{equs}
\varphi[H_{[0,\rho n]\times[0,n]}]\geq \left(\frac{c_1}{\rho^3} \varphi[ V_{[0,\rho n]\times[0,n]} ]^3\right)^{50\rho+2}.
\end{equs}
\end{proof}

To prove Proposition \ref{prop: bridge events}, we follow the strategy of \cite[Section 3]{DCT}. The main differences are the following: 
\begin{enumerate}
    \item[(1)] conditioning on a set of edges being closed does not induce the free boundary conditions, but the conditional measure is dominated by the free measure (see Lemma \ref{free-empty-vertex});
    \item[(2)] the dual model is not a dilute random cluster measure, hence the estimate on dual crossings needs to be proved directly.
\end{enumerate}

\begin{proof}[Proof of Proposition \ref{prop: bridge events}] 
Note that $\bigcup_{j=0}^C R_j \supset [0,\lambda n]\times[0,n]$, 
where $C= \lfloor \lambda n / k \rfloor \in (0,50\lambda]$. We bound the events $V_{[0,\lambda n]\times[0,n]}$ and $V^*_{[0,\lambda n]\times[0,n]}$ according to the crossings in the $R_j$ and $R_j^*$, respectively. Define the events
\begin{align*}
\sT_j
&=
\{ S_j \overset{R_j}{\conn} \Top[R_j] \}
&&
\sT_j^*
=
\{ S_j \overset{*,R_j}{\conn} \Top[R_j] \}
\\
\sL^g_j
&=
\{ S_j \overset{R_{j-13}}{\conn} \Left[R_{j+4}] \}
&&
\sL^{*,g}_j
=
\{ S_j \overset{*, R_{j-13}}{\conn} \Left[R_{j+4}] \}
\\
\sL^b_j
&=
\{ S_j \overset{R_{j}}{\conn} \Left[R_j] \} \setminus \sL^g_j
&&
\sL^{*,b}_j
=
\{ S_j \overset{*, R_{j}}{\conn} \Left[R_j] \} \setminus \sL^{*,g}_j
\\
\sR^{g}_j
&=
\{ S_j \overset{R_{j+13}}{\conn} \Right[R_{j-4}] \}
&&
\sR^{*,g}_j
=
\{ S_j \overset{*, R_{j+13}}{\conn} \Right[R_{j-4}] \}
\\
\sR^{b}_j
&=
\{ S_j \overset{R_j}{\conn} \Right[R_j] \} \setminus \sR^{g}_j
&&
\sR^{*,b}_j
=
\{ S_j \overset{*, R_j}{\conn} \Right[R_j] \} \setminus \sR^{*,g}_j.
\end{align*}

\begin{rem}
The events $\sT_j, \sT_j^*$ consist of up-down crossings and dual crossings in $R_j$. The event $\sL^g_j$ (resp. $\sL^{*,g}_j$) consists of a good left crossing in the sense that there exists a path (resp. dual path) from $S_j$ to $\Left[R_{j+4}]$ that does not explore to the right of the $y$-axis translated by the vector $(jk+5k,0)$. The event $\sL^b_j$ (resp. $\sL^{*,b}_j$) consists of a bad left crossing in the sense that any path (resp. dual path) from $S_j$ to $\Left[R_j]$ must cross the $y$ axis translated by the vector $(jk + 5k,0)$. There are similar observations for the events involving right crossings. Note that we may take these crossings to be vertex self-avoiding.
\end{rem}

Observe that, by translation invariance, we have $\varphi[C_j] = \varphi[C_0]$, where $C_j$ refers to any of the crossings or dual crossings defined above. Moreover, by reflection invariance along the $y$-axis translated by the vector $(jk +k,0)$, we have that $\varphi[C^L_j] = \varphi[C^R_j]$, where $C^L_j$ and $C^R_j$ refer to any of the same type of crossing/dual crossing event defined above. Therefore, by union bounds,
\begin{equs}
\varphi[V_{[0,\lambda n]\times[0,n]}]
&\leq
\sum_{j = 0}^C \Big( \varphi[\sT_j] + \varphi[\sL^g_j] + \varphi[\sL^b_j] + \varphi[\sR^g_j]+\varphi[\sR^b_j] \Big)	
\\
\varphi[V^*_{[0,\lambda n]\times[0,n]}]
&\leq
\sum_{j = 0}^C \Big( \varphi[\sT^*_j] + \varphi[\sL^{*,g}_j] + \varphi[\sL^{*,b}_j] + \varphi[\sR^{*,g}_j]+\varphi[\sR^{*,b}_j] \Big)
\end{equs}
from which we deduce
\begin{equs}
\begin{split} \label{eq : rsw max}
\max\{ \varphi[\sT_0], \varphi[\sL^{g}_0],\varphi[\sL^{b}_0] \}
&\geq
\frac{1}{5(C+1)} \varphi[V_{[0,\lambda n]\times[0,n]}]
\end{split}
\\
\begin{split} \label{eq : rsw max dual}
\max\{ \varphi[\sT^*_0], \varphi[\sL^{*,g}_0],\varphi[\sL^{*,b}_0] \}
&\geq
\frac{1}{5(C+1)} \varphi[V^*_{[0,\lambda n]\times[0,n]}].
\end{split} 
\end{equs}

Note that if the first maximiser is $\varphi[\sL^{g}_0]$, then on the event $\sR^{g}_0\cap\sL^{g}_2$, we have that $A_0$ happens and the desired inequality follows from translation invariance, reflection invariance, and the FKG inequality. A similar observation holds for $\varphi[\sL^{*,g}_0]$. Thus me may assume that the maxima are taken over the tortuous paths $\sT_0$, $\sT_0^*$, $\sL^b_0$, and $\sL^{*,b}_0$. Let $C_0$ be the maximiser amongst $\{ \sT_0, \sL^b_0\}$ and let $C_0^*$ be the maximiser amongst $\{ \sT_0^*, \sL^{*,b}_0 \}$. It is sufficient to establish the following conditional probability estimate. 
\begin{lem}\label{lem: cond prob}
Let $C_0$ and $C_0^*$ be as above. Then,
\begin{equs}
\varphi[A_0 \mid C_0 \cap C_4]
&\geq
\frac 12 \varphi[ C_2 \setminus (A_0 \cup A_2) \mid C_0 \cap C_4 ].
\end{equs}
and
\begin{equs}
\varphi[A_0^* \mid C_0^* \cap C_4^*]
&\geq
\frac 12 \varphi[ C_2^* \setminus (A_0^* \cup A_2^*) \mid C_0^* \cap C_4^*].
\end{equs}
\end{lem}
\noindent We assume Lemma \ref{lem: cond prob} for now and show how it is used to prove Proposition \ref{prop: bridge events}. After this proof, we then prove Lemma \ref{lem: cond prob}.

Let us do the case of $A_0$ since the same reasoning applies for $A_0^*$. By Lemma \ref{lem: cond prob} and translation invariance, we have that
\begin{equs}
2 \varphi [ A_0 ]
\geq 
\varphi \left[ (C_2 \setminus (A_0 \cup A_2)) \cap C_0 \cap C_4 \right] 	
\end{equs}
and 
\begin{equs}
2\varphi[A_0]
\geq
\varphi[A_0 \cap (C_2 \cap C_0 \cap C_4)] + \varphi[ A_2 \cap (C_0 \cap C_2 \cap C_4)].	
\end{equs}
Hence, by summing over disjoint events and using the FKG inequality, we obtain 
\begin{equs}
4\varphi[A_0]
&\geq
\varphi[(C_0 \cap C_2 \cap C_4) \setminus (A_0 \cup A_2)] + \varphi [(C_0 \cap C_2 \cap C_4) \cap (A_0 \cup A_2)]
\\
&=
\varphi[C_0 \cap C_2 \cap C_4]
\geq
\varphi[C_0]^3
\end{equs}
These estimates combined with \eqref{eq : rsw max} and \eqref{eq : rsw max dual} complete the proof.
\end{proof} 

It remains to prove Lemma \ref{lem: cond prob}.

\begin{proof}[Proof of Lemma \ref{lem: cond prob}]
Since the case concerning dual bridges requires some care, we only treat this one. Moreover, we only analyse the case of $C_0=\sL^{*,b}_0$, which requires a much more subtle symmetric domain argument. See Figure \ref{fig:symmetric domain}. The analogous arguments for the other case and also the events involving primal crossings are easier and therefore we omit them. We emphasise that the only differences in the proofs of the latter are purely geometric. The interested reader can consult \cite[Lemma 9]{DCT} for more details.

Assume that $C_0 = \cL^{*,b}_0$, i.e.\ the case of tortuous left dual crossings. We partition the event $\cL^{*,b}_0 \cap \cL^{*,b}_4$ in terms of the left-most and right-most tortuous left dual crossings in $R_0$ and $R_4$ and denote these random variables by $\Gamma_L^*$ and $\Gamma_R^*$, respectively. We fix $\gamma_L^*$ and $\gamma_R^*$ dual paths such that all edges except the first and last lie in their respective domains. We additionally impose that they satisfy the constraints of the crossings in $\cL^{*,b}_0$ and $\cL^{*,b}_4$ in $R_0$ and $R_4$, respectively. We condition on $\Gamma_L^* = \gamma_L^*$ and $\Gamma_R^* = \gamma_R^*$. We first assume that $\gamma_L^*$ and $\gamma_R^*$ do not intersect, and we will handle the case where they intersect separately.

Let $\mathbf{L}$ denote the vertical axis $\{(5k,t) : t \in \mathbb{R} \}$ endowed with its natural ordering. We call $\mathbf{H}$ the half-plane to the right of $\mathbf{L}$. Let $x \in \mathbf{L}$ denote the first intersection point of $\gamma_L^*$ and $\mathbf{L}$, where we start the exploration of the path from $S_0$. Consider now the path $\gamma_R^*$. Since we assume that $\gamma_R^* \cap \gamma_L^* = \emptyset$ and by the definition of $S_4$, necessarily there exist $x_-,x_+ \in \mathbf{L}$ consecutive intersection points of $\gamma_R^*$ (traced from the starting point in $S_4$) with $\mathbf{L}$ such that $x_- < x < x_+$. Let $\upsilon_L$ denote the arc of $\gamma_L^*$ from the start point to $x$ and let $\tilde\upsilon_L$ denote its reflection across $\mathbf{L}$. Necessarily we have that $\tilde\upsilon_L \cap \gamma_R^*|_{[x_-,x_+]} \neq \emptyset$, where $\gamma_R^*|_{[x_-,x_+]}$ denotes the arc of $\gamma_R^*$ traced between $x_-$ and $x_+$.  Let $\upsilon_R$ denote the arc of $\gamma_R^*$ traced between $x_-$ and its first intersection point with $\tilde\upsilon_L$, and let $\tilde \upsilon_R$ denote its reflection across $\mathbf{L}$. We identify these dual paths with the primal edges that they cross and let ${\rm Sym}(\Omega)$ denote the domain in the primal enclosed by these arcs in the sense that these arcs form its external edge boundary. Furthermore, let $\Omega$ denote the subdomain of ${\rm Sym}(\Omega)$ enclosed between the full paths of $\gamma_L^*$ and $\gamma_R^*$. See Figure~\ref{fig:symmetric domain} Note that ${\rm Sym}(\Omega)$ and $\Omega$ are homotopic and the only paths that possibly need to be deformed are $\tilde\upsilon_L$ and $\tilde\upsilon_R$. We call these pieces $\Top[{\rm Sym}(\Omega)]$ and $\Bot[{\rm Sym}(\Omega)]$, respectively, and they correspond after the homotopy to two pieces that we call $\Top[\Omega]$ and $\Bot[\Omega]$, respectively. In addition, we call $\Left(\Omega)=\Left({\rm Sym}(\Omega)) = \upsilon_L$ and $\Right[\Omega]=\Right[{\rm Sym}(\Omega)]=\upsilon_R$.

\begin{figure}
    \centering
    \includegraphics[width=.6\linewidth]{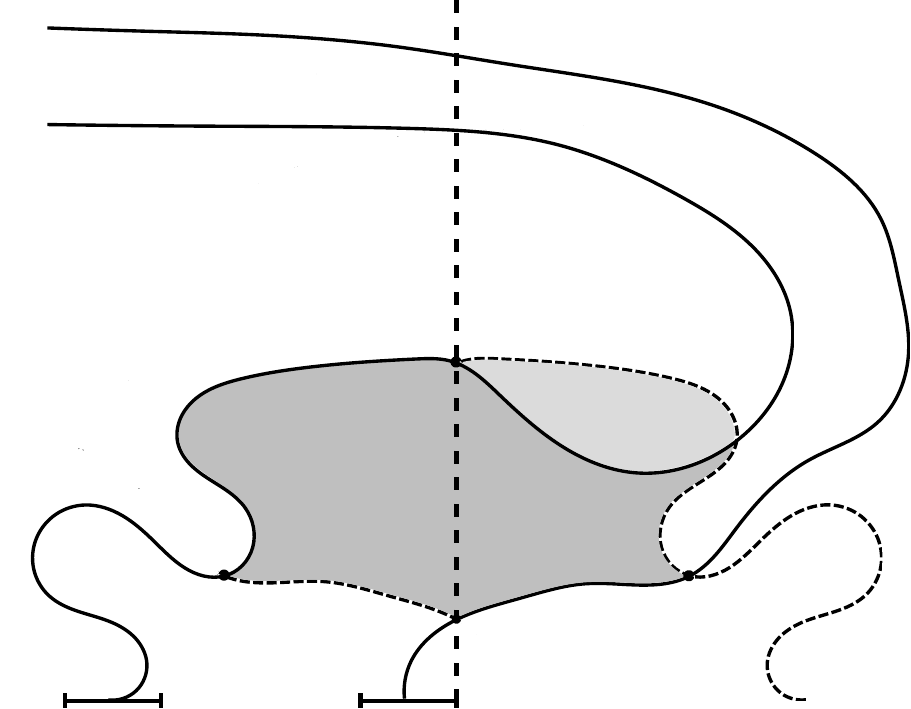}
    \put(-90,110){${\rm Sym}(\Omega)$}
    \put(-158,65){$\Omega$}
    \put(-245,-10){$S_0$}
    \put(-155,-10){$S_4$}
    \put(-150,140){$\mathbf{L}$}
    \put(-130,20){$x_-$}
    \put(-130,110){$x$}
    \caption{The domain $\Omega$ shaded in dark grey. The symmetric domain ${\rm Sym}(\Omega)$ is the union of the two grey shaded areas. The arc emanating from $S_0$ is $\upsilon_L$. Its reflection, which is dashed, is $\tilde\upsilon_L$. The solid arc emanating from $x_-$ to its first intersection with $\tilde\upsilon_L$ is $\upsilon_R$. Its reflection, which is also dashed, is $\tilde\upsilon_R$.}
    \label{fig:symmetric domain}
\end{figure}
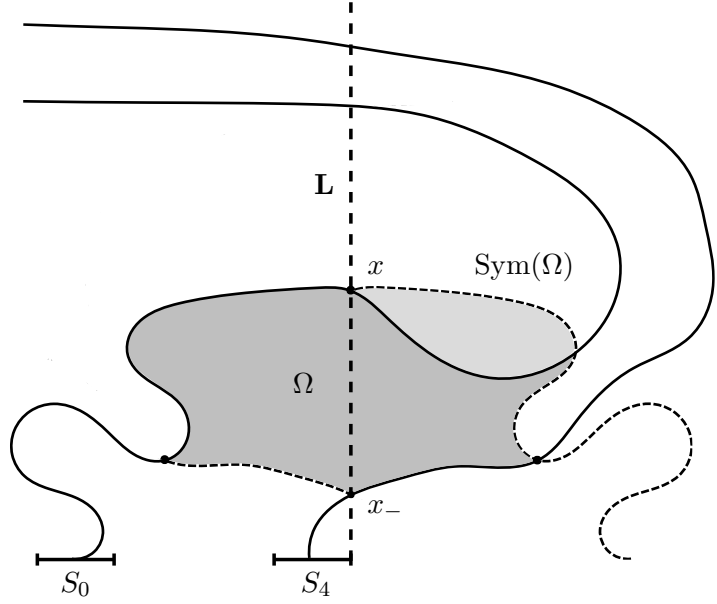

Let ${\rm mix}$ denote the following boundary conditions on $\Omega$:
\begin{itemize}
\item $0$ vertices on the outside of $\Left[\Omega]\cup\Right[\Omega]$, 
\item wired on $\Top[\Omega]$,
\item wired on $\Bot[\Omega]$.	
\end{itemize} 
We may define the analogous boundary conditions on ${\rm Sym}(\Omega)$ via the homotopy. Now let $\xi$ denote the random boundary condition on $\Omega$ induced by $\gamma_L^*$ and $\gamma_R^*$ such that the left-most and right-most interior bonds are closed. Then by Lemma \ref{free-empty-vertex} and monotonicity in the domain we have that
\begin{equs}
\varphi[A_0^* \mid \Gamma_L^* = \gamma_L^*, \Gamma_R^* = \gamma_R^*]
&\geq
\varphi\left[ \varphi_\Omega^\xi	[ \gamma_L^* \connstar \gamma_R^*  ] \mid \Gamma_L^* = \gamma_L^*, \Gamma_R^* = \gamma_R^* \right]
\\
&\geq
\varphi_\Omega^{\rm mix} \left[ \Left[\Omega] \connstar \Right[\Omega]\right]
\\
&\geq
\varphi_\Sym^{\rm mix} \left[ \Left[\Sym(\Omega)] \connstar \Right[\Sym(\Omega)] \right].
\end{equs}

Now, we turn to the righthand side of the lemma. We again partition $ \cL_0^{*,b} \cap \cL_4^{*,b}$ according to $\Gamma_L^*$ and $\Gamma_R^*$ and assume $\Gamma_L^* = \gamma_L^*$ and $\Gamma_R^* = \gamma_R^*$ as before. Conditionally on this, note that the event $\cL_2^{*,b} \setminus (A_0^* \cup A_2^*)$ implies the existence of random primal paths $\Pi_L$ (left-most) and $\Pi_R$ (right-most) paths separating $\gamma_L^*$ from $\gamma_R^*$. Conditionally on $\Pi_R = \pi_R$ and $\Pi_L = \pi_L$, let $\Omega^*$ denote the set of vertices enclosed by $\pi_R \cup \pi_L$, and $\Top[\Omega]\cup \Bot[\Omega]$, and denote $\Top[\Omega^*]$ and $\Bot[\Omega^*]$ the intersection of $\Top[\Sym(\Omega)]$ and $\Bot[\Sym(\Omega)]$ with the boundary of $\Omega^*$, respectively. Let $\xi$ denote a random boundary condition on $\Omega^*$ that is wired on $\pi_R \cup \pi_L$. Let also ${\rm mix}^*$ denote the following boundary conditions on $\Omega^*$:
\begin{itemize}
    \item $0$ vertices on $\Top[\Omega^*]$ and $\Bot[\Omega^*]$
    \item wired on $\pi_L \cup \pi_R$.
\end{itemize}
We transfer these boundary conditions to ${\rm Sym}(\Omega)$ via the homotopy. 
By the spatial Markov property, monotonicity in boundary conditions, and inclusion of events,
\begin{equs}
\varphi[\cT_2^* \setminus (A_0^* \cup A_4^*) \mid \gamma_L^*, \gamma_R^*, \pi_L, \pi_R ]
&\leq
\varphi[\varphi_{\Omega^*}^\xi[\Bot[\Omega^*] \connstar \Top[\Omega^*]] \mid \gamma_L^*, \gamma_R^*, \pi_L, \pi_R  ]	
\\
&\leq
\varphi^{{\rm mix}^*}_{\Omega^*}[\Bot[\Omega^*] \connstar \Top[\Omega^*]]
\\
&\leq
\varphi_\Sym^{{\rm mix}^*}[\Bot[\Sym(\Omega)] \connstar \Top[\Sym(\Omega)]]
\\
&\leq
2\varphi_\Sym^{{\rm mix}}[\Bot[\Sym(\Omega)) \connstar \Top[\Sym(\Omega)]]
\end{equs}
where the factor $2$ comes from difference in the number of connected clusters when comparing ${\rm mix}^*$ and $\rm mix$ boundary conditions.
By the reflection symmetries of the domain and the measure
\begin{equs}
\varphi_\Sym^{\rm mix} [ \Left[\Sym(\Omega] \connstar \Right[\Sym(\Omega)] ]=\varphi_\Sym^{{\rm mix}}[\Bot[\Sym(\Omega)) \connstar \Top[\Sym(\Omega)]],
\end{equs}
hence 
\begin{equs}
\varphi[A_0^* \mid \Gamma_L^* = \gamma_L^*, \Gamma_R^* = \gamma_R^*]
&\geq
\frac{1}{2}\varphi[\cT_2^* \setminus (A_0^* \cup A_4^*) \mid \gamma_L^*, \gamma_R^*, \pi_L, \pi_R ].
\end{equs}

Note that if $\gamma_L^* \cap \gamma_R^* \neq \emptyset$, the event $A_0^*$ occurs, and the last inequality is trivially true. The desired assertion follows.

\end{proof}

\subsection{Renormalisation: Proof of Lemma \ref{lem: renormalisation inequality for strip densities}} \label{subsec: proof of renorm}

We require an intermediary lemma that gives estimates on crossing probabilities under balanced boundary conditions at macroscopic distance away from the rectangle. This allows us to push boundary conditions in the proof of the renormalisation inequality.

\begin{lem} \label{lem: push bc}
Assume {\bf non(SubCrit)} and {\bf non(SupCrit)}. There exists $c>0$ such that, for all $n \in 9\NN$, at least one of the following must occur:
\begin{align}\label{eq: pushprimal} \tag{\textbf{PushPrimal}}
\forall \alpha > 0, \qquad &\varphi_{\overline{R}}^{{\rm LTR}}[H_R] 
\geq
c^\alpha
\\ \label{eq: pushdual} \tag{\textbf{PushDual}}
\forall \alpha > 0,
\qquad 
&\varphi_{\overline{R}}^{\rm B}[V_R^c]
\geq
c^\alpha
\end{align} 
where $R=[0,\alpha n]\times [0,n/9]$; $\overline{R}=[0,\alpha n] \times [0, 28 n/9]$; {\rm LTR} is the boundary condition that is wired on $\Left[\overline R]\cup\Top[\overline  R]\cup\Right[\overline  R]$ and $0$ on $\Bot[\overline  R]$; and, {\rm B} is the boundary condition that is $0$ on $\Left[\overline  R]\cup\Top[\overline  R]\cup\Right[\overline  R]$, and wired on $\Bot[\overline R]$.
\end{lem}

The proof of Lemma \ref{lem: push bc} is almost exactly the same as \cite[Sections 5.2]{DCT} and we sketch  it for completeness.
\begin{proof}
It is sufficient to show that there exists $c>0$ such that for all\footnote{The case for generic $n$ then follows by monotonicity arguments.} $n \in 27\N$ 
\begin{align}
\begin{split} \label{eq: push lemma 1}
\text{ either } \qquad &\forall \alpha \geq 1, \qquad \varphi^{0/1}_{S_n}[H_{[0,\alpha n] \times [0,n/9]}] \geq c^\alpha,
\end{split}
\\
\begin{split}\label{eq: push lemma 2}
\text{ or } \qquad &\forall \alpha \geq 1, \qquad \varphi^{0/1}_{S_n}[V_{[0,\alpha n]\times [0,n/9]}^c] \geq c^\alpha.
\end{split}
\end{align}
Indeed assume that \eqref{eq: push lemma 2} holds. Define the thin rectangles
\begin{equs}
    R_i
    &=
    [0,\alpha n]\times [in/27, (i+1)n/27], \qquad i \in \{1, \dots, 83\}.
\end{equs} 
In order to show that \eqref{eq: pushdual} occurs, it is sufficient by monotonicity in the boundary conditions and inclusion of events to show that there exists $c_2>0$ such that
\begin{equs}
\varphi^{0/1}_{\overline{R}}[V_{R_1}^c]
\geq
c_2^\alpha. 
\end{equs}
Observe that
\begin{equs} \label{eq: pushdual iteration}
\varphi^{0/1}_{\overline{R}}[V_{R_1}^c]
\geq
\varphi^{0/1}_{\overline{R}}[V_{R_1}^c \mid V_{R_2}^c] \varphi^{0/1}_{\overline{R}}[V_{R_2}^c].
\end{equs}
Furthermore, observe that, conditionally on $V_{R_2}^c$, we may use monotonicity in boundary conditions and \eqref{eq: push lemma 2} to deduce that $V_{R_1}^c$ occurs with probability $c^{3\alpha}$, since we apply it for $n/3$ and $3\alpha$. More generally, we have that 
\begin{equs}
 \varphi^{0/1}_{\overline{R}}[V_{R_j}^c \mid V_{R_{j+1}}^c]
 \geq
 c^{3\alpha}.
\end{equs}
We iterate the conditioning in \eqref{eq: pushdual iteration}. In order to close the iteration, we note that by monotonicity in boundary conditions and \eqref{eq: push lemma 2} we have that 
\begin{equs}
\varphi^{0/1}_{\overline{R}}[V^c_{R_{83}}]
\geq
c^{3\alpha}.
\end{equs}
Thus, iterating $83$ times, we obtain
\begin{equs}
\varphi^{0/1}_{\overline{R}}[V^c_{R_{1}}]
\geq
c^{249\alpha}.
\end{equs}
which establishes the desired result.

It remains to establish the dichotomy \eqref{eq: push lemma 1} and \eqref{eq: push lemma 2}. Define the rectangles
\begin{equs}
R'_i
&:=
[0,2n]\times[in/3,(i+1)n/3], \qquad i = 0,1,2
\end{equs}
and let $R'=[0,2n]\times[0,n]$.
Note that if there exists $i$ such that $\varphi^{0/1}_{S_n}[V_{R'_i}] \geq 1/6$, then by inclusion of events and \eqref{eq: RSW} we obtain that $\varphi^{0/1}_{S_n}[H_{R''_i}] \geq t$ for some constant $t>0$, where $R''_i=[0,4n]\times [in/3,i(n+1)/3]$. Note that when the rectangles $(kn,0)+R''_i$ are crossed horizontally, and the rectangles $(kn,0)+R'_i$ are crossed vertically for $k=0,1,\ldots,\alpha$, then $H_R$ occurs. Thus \eqref{eq: push lemma 1} is satisfied in this case by the FKG and translation invariance of the measure. Similarly, by \eqref{eq: RSWstar}, if there exists $i$ such that $\varphi^{0/1}_{S_n}[H_{R_i'}^c] \geq 1/6$, then \eqref{eq: push lemma 2} holds. Therefore we assume that for all $i$,
\begin{equs} \label{eq: push assumption}
\varphi^{0/1}_{S_n}[V_{R_i'}^c], \varphi^{0/1}[H_{R_i'}]
\geq
5/6. 
\end{equs}
In order to prove the desired result, we do a gluing of bridges construction analogous to the proof of Proposition \ref{prop: RSW}. We define precisely one of these bridge events and then show that they occur with positive probability, the rest of the argument then follows from FKG. Let us only focus only on \eqref{eq: push lemma 2}. For $i=0,\dots,5$ define the segment 
\begin{equs}
I_i
&=
[in/3,(i+1)n/3]\times\{n\}.
\end{equs}
As hinted above, our goal is to show that that there exists $c>0$ such that
\begin{equs}
\varphi^{0/1}_{S_n}[I_1 \overset{*, R'}{\longleftrightarrow} I_4]
\geq 
c.
\end{equs}

By the assumption
\eqref{eq: push assumption} we have that
\begin{equs}
\varphi^{0/1}_{S_n}[I_1 \overset{*,R'}{\longleftrightarrow} I_4]
&\geq
\varphi^{0/1}_{S_n}[I_1 \overset{*,R'}{\longleftrightarrow} I_4 \mid V_{R'_1}^c] \, \varphi^{0/1}_{S_n}[V_{R_1'}^c]
\geq
\frac 56 \varphi^{0/1}_{S_n}[I_1 \overset{*,R'}{\longleftrightarrow} I_4 \mid V_{R'_1}^c].
\end{equs}
In order to bound the probability on the righthand side, we introduce the rectangles 
\begin{equs}
K_i
&=
[in/3, (i+1)n/3]\times[0,2n], \qquad i = 1,\dots,5.
\end{equs}
Note that each of these rectangles are $\pi/2$ rotations of $R_i$. 
Let $K=\bigcup_{i=0}^5 K_i$.
Note that for each $i$ the segment $I_i$ bisects the rectangle $K_i$. Note that a helpful figure for what follows can be found in \cite[Figure 9]{DCT}. We partition the event $V_{R_1'}^c$ according to the lowest dual crossing $\Gamma^*$ in $R_1'$. Letting $\gamma^*$ be such a dual crossing of $R_1'$, we denote $K(\gamma^*)$ to be the uppermost connected component of $K\setminus \gamma^*$ and let also $R'(\gamma*)$ denote the uppermost connected component of $R' \setminus \gamma^*$. Then by the spatial Markov property, monotonicity in boundary conditions, and the FKG inequality, we have
\begin{equs}
\varphi^{0/1}_{S_n}[I_1 \overset{*,R'}{\longleftrightarrow} I_4 \mid \Gamma^* = \gamma^*]
&\geq
\varphi^{{\rm mix}}_{K(\gamma^*)}[I_1 \overset{*,R'(\gamma^*)}{\longleftrightarrow}I_4]
\\
&\geq
\varphi^{{\rm mix}}_{K(\gamma^*)}[I_1 \overset{*,K_1}{\longleftrightarrow} \gamma^*] \varphi^{{\rm mix}}_{K(\gamma^*)}[I_4 \overset{*,K_4}{\longleftrightarrow} \gamma^*].
\end{equs}
Above, ${\rm mix}$ denotes the boundary conditions that are $0$ on $\Top[K(\gamma^*)]$, $0$ on $\Bot[K(\gamma^*)]$ and wired on the other sides. Let $K_L = K_0\cup K_1\cup K_2$ and $K_R = K_3 \cup K_4 \cup K_4$, and let ${\rm mix}$ denote the analogous boundary conditions on these rectangles. Note that these are rotations of $R'$. Then by the monotonicity in boundary conditions and $\pi/2$-rotational symmetry, we may write 
\begin{equs}
\varphi^{0/1}_{S_n}[I_1 \overset{*,R'}{\longleftrightarrow} I_4 \mid \Gamma^* = \gamma^*]
\geq
\varphi^{\rm mix}_{R'}[V_{R_1}^c]^2
\end{equs}
where $\rm mix$ on $R'$ denotes wired on the top, wired on the bottom, and free on the left and righthand sides. 
In order to bound the righthand side from below, note that by a union bound (on the complementary events), our assumption \eqref{eq: push assumption}, and spatial Markov property we have
\begin{equs}
1/2
\leq
\varphi^{0/1}_{S_n}[H_{R_0}\cap V_{R_1}^c \cap H_{R_2}]
\leq
\varphi^{\rm mix}_{R'}[V_{R_1}^c].
\end{equs} 
Thus, putting the above bounds together, we find that 
\begin{equs}
\varphi^{0/1}_{S_n}[I_1 \overset{*,R'}{\longleftrightarrow} I_4]
\geq
\frac{5}{24}. 
\end{equs}

\end{proof}

We now prove Lemma \ref{lem: renormalisation inequality for strip densities}. We stress again that the proof is an adaptation of the proof of \cite[Lemma 15]{DCT}, where the main differences are those discussed in the proof of Proposition \ref{prop: RSW}.
\begin{proof}[Proof of Lemma \ref{lem: renormalisation inequality for strip densities}] 
 Without loss of generality, we prove the second inequality (renormalisation of the $q_n$'s). Moreover, without loss of generality we may assume that \eqref{eq: pushdual} occurs.

Assume that $n \in 9\NN$. Define the rectangles
\begin{align*}
R&=[0,\alpha n] \times [0, 6\lambda n + 3n] 
\\
R_i&= [0,\alpha n] \times [6in+3n, 6in + 6n], &\qquad 0 \leq i \leq \lambda - 1
\\
R_i'&=[0,\alpha n] \times [6in, 6in + 3n], &\qquad 0 \leq i \leq \lambda.
\end{align*}
Note that these rectangles are not the same as those of Lemma \ref{lem: push bc}. Furthermore, note that the $R_i$ and $R_i'$ partition $R$. We further subdivide each $R_i$ horizontally into three thin rectangles. For $0 \leq i \leq \lambda -1$, define
\begin{equs}
\tilde R_i^-
&:=
[0,\alpha n]\times [6 i n + 3n + \frac{12}{9}n, 6 i n + 3n + \frac{13}{9}n ]
\\
\tilde R_i
&:=
[0,\alpha n]\times [6 i n + 3n + \frac{13}{9}n, 6 i n + 3n + \frac{14}{9}n ]
\\
\tilde R_i^+
&:=
[0,\alpha n]\times [6 i n + 3n + \frac{14}{9}n, 6 i n + 3n + \frac{15}{9}n ].
\end{equs} Let $\cE$ denote the event that each rectangle $R_i$ is crossed horizontally. Let $\cF$ denote the event that each rectangle $R_i'$ is not crossed vertically, i.e. that there is a dual horizontal crossing. Let $\tilde \cE$ be the event that all of the $\tilde R_i$ are crossed horizontally. Let $\tilde \cF$ be the event that none of the $\tilde R_i^{\pm}$ are crossed vertically.

Let $1/0$ denote the boundary condition that is wired on $\Top[R] \cup \Bot[R]$ and free on $\Left[R] \cup \Right[R]$. Assume the following estimates hold:
\begin{align} \label{eq: pf of ren 1}
\varphi^{1/0}_R[\tilde \cE \cap \cF]
&\geq
\frac{(2r)^{-12\lambda n + 6n}}{\lambda^{C\lambda \alpha }}
\varphi^1_{[0,\alpha n] \times [-n, 2n]}[V_{[0,\alpha n]\times [0,n]}^c]^{\lambda + 1}
\\ \label{eq: pf of ren 2}
\varphi^{1/0}_R[\tilde\cF \mid \tilde \cE \cap \cF]
&\geq
c^{2\lambda \alpha}
\\ \label{eq: pf of ren 3}
\varphi^{1/0}_R[\tilde \cE \cap \tilde \cF]
&\leq
\varphi^0_{[0,\alpha n] \times [-n/9, 2n/9]}[H_{[0,\alpha n]\times [0, n/9]}]^\lambda.
\end{align}
Then, by \eqref{eq: pf of ren 1} and \eqref{eq: pf of ren 2},
\begin{equs}
\varphi^{1/0}_R[\tilde \cE \cap \tilde \cF]
&\geq
\varphi^{1/0}_R[\tilde \cF \mid \tilde \cE \cap \cF] \varphi^{1/0}_R[\tilde \cE \cap \cF]
\\
&\geq
\frac{c^{2\lambda \alpha}(2r)^{-12\lambda n + 6n}}{\lambda^{C\lambda \alpha }}
\varphi^1_{[0,\alpha n] \times [-n, 2n]}[V_{[0,\alpha n]\times [0,n]}^c]^{\lambda + 1}.
\end{equs}
Hence, by \eqref{eq: pf of ren 3}, we get that
\begin{equs}
\varphi^1_{[0,\alpha n] \times [-n, 2n]}[V_{[0,\alpha n]\times [0,n]}^c]^{(\lambda + 1)/\alpha} 
\leq
\frac{\lambda^{C\lambda }}{c^{2\lambda }(2r)^{(-12\lambda n + 6n)/\alpha}}
\varphi^0_{[0,\alpha n] \times [-n/9, 2n/9]}[H_{[0,\alpha n]\times [0, n/9]}]^{\lambda/\alpha}
\end{equs}
The renormalisation inequality then follows: i) raising to the power of $1/\lambda$; ii) taking $\alpha \rightarrow \infty$ to express these events in terms of $p_m$ and $q_m$, where $m=m(n)$; and, iii) applying the duality relation between $p_m$ and $q_m$ of Lemma \ref{lem: dual relation of strip densities}. It remains to prove \eqref{eq: pf of ren 1}-\eqref{eq: pf of ren 3}. 

To prove \eqref{eq: pf of ren 1}, first note that
\begin{equs} \label{eq: ren1 pf1}
\phi^{1/0}_{R}[\tilde \cE \cap \cF] 
\geq
(2r)^{-12\lambda n + 6n} \varphi^1_R[\tilde \cE \cap \cF]
\geq 
(2r)^{-12\lambda + 6n} \varphi^1_R[\tilde \cE] \varphi^1_R[\cF \mid \tilde \cE].
\end{equs}
Note that wired boundary conditions are favourable for the event $\tilde\cE$. Thus, by monotonicity in boundary conditions, the FKG inequality, and arguing as in \cite[Corollary 11]{DCT} to find the dependencies on $\lambda$ and $\alpha$, we find
\begin{equs} \label{eq: ren1 pf2}
\varphi^1_R[\tilde\cE]
\geq
\varphi^1_{S_{6\lambda n + 3n}}[\tilde\cE]
\geq
\prod_{i=1}^{\lambda - 1} 
\varphi^1_{S_{6\lambda n + 3n}}[\tilde R_i]
\geq
\lambda^{C \lambda \alpha}
\end{equs}
where above we abuse notation and write $S_{6\lambda n + 3n} = \RR \times[0, 6\lambda n + 3n]$.
On the other hand, note that we may partition the event $\tilde \cE$ as follows. For each $1 \leq i \leq \lambda - 1$, let $\Gamma_i$ be the top-most path in $\tilde R_i$, and set $\Gamma_{0} = \Bot[R]$ and $\Gamma_{\lambda}=\Top[R]$. Let $R_i'(\Gamma)$ be the random domain with boundary given by $\Gamma_{i+1}$, $\Gamma_{i}$, and the relevant left and right segments of the boundary of $R$. Then, by conditioning on $\Gamma$, the spatial Markov property, independence, monotonicity in boundary conditions, and translation invariance, we find
\begin{equs} \label{eq: ren1 pf3}
\varphi^1_R[\cF \mid \tilde \cE]
\geq
\prod_{i=0}^{\lambda}\varphi^1_{R_i'(\Gamma)}[V_{R_i'}^c]
\geq
\varphi^1_{[0,\alpha n] \times [-n, 2n]}[V_{[0,\alpha n]\times [0,n]}^c]^{\lambda + 1}.
\end{equs}
Inserting \eqref{eq: ren1 pf2} and \eqref{eq: ren1 pf3} into \eqref{eq: ren1 pf1} establishes \eqref{eq: pf of ren 1}.

A similar conditioning argument, this time on the dual paths occuring in the $\tilde R_i^{\pm}$, together with the stochastic domination in Lemma \ref{free-empty-vertex}, yields \eqref{eq: pf of ren 3}. Finally, in order to obtain \eqref{eq: pf of ren 2}, we note that by a similar conditioning (this time on the crossings in $\tilde \cE$ and dual crossings of $\cF$) and pushing boundary conditions argument, together with translation invariance, we have the following estimate:
\begin{equs}
\varphi^{1/0}_R[\tilde \cF \mid \tilde \cE \cap \cF]
\geq
\varphi^{\rm LTR}_{[0,\alpha n] \times [0, 28n/9]}[V_{[0,\alpha n]\times [0,n/9]}^c]^{2\lambda}.
\end{equs}
Applying the \eqref{eq: pushdual} then completes the proof.

\end{proof}

\subsection{Applications of the quadrichotomy}

We now state some important consequences of Proposition \ref{prop:quad} that allow us to get a better quantitative understanding of the phase diagram in $d=2$. We suppress notational dependence on $p$ and $a$ when clear from context. 

The following propositions characterise the off-critical behaviour. The proofs of the first two are standard and can be found in \cite[Section 6]{DCT}.
\begin{prop}\label{cor:1}
Assume that {\bf (SubCrit)} occurs. There exists $c>0$ such that for every $n\geq 1$,
\begin{equs}\label{eq:kj}\varphi_{\Lambda_n}^1[0\longleftrightarrow \partial\Lambda_n]\leq e^{-cn}.\end{equs}
In particular, $\varphi^1[0\longleftrightarrow\infty]=0$ and $\varphi^0=\varphi^1$.
\end{prop}

\begin{prop}\label{cor:2}
Assume that {\bf (SupCrit)} occurs. There exists $c>0$ such that for every $n\geq 1$,
\begin{equs}
\varphi^0[\Lambda_n\centernot\longleftrightarrow\infty]\leq e^{-cn}.
\end{equs}
In particular, $\varphi^0[0\longleftrightarrow\infty]>0$ and $\varphi^0=\varphi^1$.
\end{prop}

Let $\mathcal{C}_n$ be the event that there is a circuit in $\omega\cap(\Lambda_{2n}\setminus \Lambda_n)$ surrounding the origin.

\begin{prop}\label{sup-circuit}  
Assume that {\bf (SupCrit)} occurs. Then there exists $t>0$ such that for every $n\geq 1$, we have $\varphi^0[\mathcal{C}_n] \geq 1- e^{-tn}$.
\end{prop}
\begin{proof}
Note that if the event $\mathcal{C}_n$ does not happen, then either $H^*_{[-2n,-n]\times[-2n,2n]}$ or one of its rotations by $\pi/2$, $\pi$ or $3\pi/2$ happens. By the $\pi/2$ symmetry of $\varphi^0$, it suffices to show that $\varphi^0[H^*_{[-2n,-n]\times[-2n,2n]}]$ decays exponentially. Since we are assuming that {\bf (SupCrit)} occurs, using \eqref{eq:MON} we can conclude that $\varphi^0[H_n]\geq 1-e^{-cn}$. The desired exponential decay follows now from \eqref{eq: RSW} and translation invariance.
\end{proof}

The following proposition characterises the discontinuous critical behaviour.
\begin{prop}\label{cor:3}
Assume that {\bf (DiscontCrit)} occurs. There exists $c>0$ such that for every $n\geq 1$,
\begin{equs}
\varphi^0[0\longleftrightarrow\partial \Lambda_n]\leq e^{-cn} \quad \text{and} \quad \varphi^1[\Lambda_n\centernot\longleftrightarrow\infty]\leq e^{-cn}.
\end{equs}
In particular,  $\varphi^0[0\longleftrightarrow\infty]=0$ and  $\varphi^1[0\longleftrightarrow\infty]>0$.
\end{prop}
In order to prove Proposition \ref{cor:3}, we require an intermediate lemma. Similarly as above, let $\mathcal{C}^*_n$, be the event that there is a circuit in $\omega^*\cap(\Lambda_{2n}\setminus \Lambda_n)$ surrounding the origin.

\begin{lem}\label{dual-circuit} 
Assume that {\bf (DiscontCrit)} occurs. Then there exists $t>0$ such that for every $n\geq 1$, we have $\varphi^0[\mathcal{C}^*_n] \geq 1- e^{-tn}$ and $\varphi^1[\mathcal{C}_n] \geq 1- e^{-tn}$.
\end{lem}
\begin{proof}
We will show the assertion in the case of $\varphi^0$, with the case of $\varphi^1$ being similar.
We start by showing that $\varphi^0_{\Lambda_{4n}}[\mathcal{C}^*_n] \geq 1- e^{-tn}$, and then proceed to show that the same holds at infinite volume.

Note that if the rectangles $[-2n,-n]\times[-2n,2n]$ and $[n,2n]\times[-2n,2n]$ have dual vertical crossings, while $[-2n,2n]\times[-2n,-n]$ and $[-2n,2n]\times[n,2n]$ have dual horizontal crossings, then $\mathcal{C}^*_n$ occurs. By the $\pi/2$ rotational symmetry of the measure $\varphi^0_{\Lambda_{4n}}$, the FKG inequality, and duality, it suffices to show that $\varphi^0_{\Lambda_{4n}}[H_{[-2n,-n]\times[-2n,2n]}]$ decays exponentially fast to $0$.

To this end, recall that $p_n$ decays exponentially by our assumption that {\bf (DiscontCrit)} occurs. It follows from the union bound that there are vertices $x$ and $y$ on the left and right side of $[-2n,-n]\times[-2n,2n]$, respectively, such that 
\begin{equs}
\varphi^0_{\Lambda_{4n}}[x\longleftrightarrow y \text{ in } [-2n,-n]\times[-2n,2n]]\geq \frac{\varphi^0_{\Lambda_{4n}}[H_{[-2n,-n]\times[-2n,2n]}]}{(4n+1)^2}.
\end{equs}
Moreover, for every $k\geq 1$, we can create a horizontal crossing of the rectangle $[0,4kn]\times [0,4n]$ by combining $4k$ translations and reflections of the event $\{x\longleftrightarrow y \text{ in } [-2n,-n]\times[-2n,2n]\}$, as appropriate. Using the FKG inequality, \eqref{eq:MON}, and the reflection symmetry of the measure $\varphi^0_{S_{4n}}$, we obtain
\begin{equs}
\varphi^0_{S_{4n}}[H_{[0,4kn]\times [0,4n]}]\geq\left(\frac{\varphi^0_{\Lambda_{2n}}[H_{[-2n,-n]\times[-2n,2n]}]}{(4n+1)^2}\right)^{4k}.
\end{equs}
By the finite energy property, there exists a constant $c>0$ such that  
\begin{equs}
\phi^0_{[0,4kn]\times[-4n,8n]}[H_{[0,4kn]\times [0,4n]}]\geq e^{-cn} \phi^0_{S_{4n}}[H_{[0,4kn]\times [0,4n]}],
\end{equs}
hence taking $k$th roots and sending $k$ to infinity, we obtain 
\begin{equs}
\varphi^0_{\Lambda_{4n}}[H_{[-2n,-n]\times[-2n,2n]}]\leq (4n+1)^2 p_{4n}^{1/4}.
\end{equs}
This implies that $\varphi^0_{\Lambda_{2n}}[H_{[-2n,-n]\times[-2n,2n]}]$ decays exponentially, as desired.

To prove the assertion at infinite volume, consider some $k\geq 1$, and note that 
\begin{equs}
\varphi^0_{\Lambda_{2^k n}}[\mathcal{C}^*_{n}]\geq \varphi^0_{\Lambda_{2^k n}}\left[\bigcap_{i=1}^k \mathcal{C}^*_{2^{k-i} n}\right]=\varphi^0_{\Lambda_{2^k n}}\left[\mathcal{C}^*_{2^{k-1} n}\right]\prod_{i=2}^k\varphi^0_{\Lambda_{2^k n}}\Big[\mathcal{C}^*_{2^{k-i}n} \mid \bigcap_{j=1}^{i-1} \mathcal{C}^*_{2^{k-j} n} \Big].
\end{equs}
By Lemma~\ref{free-closed} and \eqref{eq:DMPRC}, \eqref{eq:MON}, we have 
\begin{equs}
\varphi^0_{\Lambda_{2^k n}}\Big[\mathcal{C}^*_{2^{k-i}n} \mid \bigcap_{j=1}^{i-1} \mathcal{C}^*_{2^{k-j} n} \Big]\geq \varphi^0_{\Lambda_{2^{k-i+2} n}}[\mathcal{C}^*_{2^{k-i}n}],
\end{equs}
and $\varphi^0_{\Lambda_{2^k n}}[\mathcal{C}^*_{2^{k-1} n}]\geq \varphi^0_{\Lambda_{2^{k+1}n}}[\mathcal{C}^*_{2^{k-1} n}]$, which implies that $\varphi^0_{\Lambda_{2^k n}}[\mathcal{C}^*_{n}]\geq 1-e^{-t'n}$ for some constant $t'>0$. Sending $k$ to infinity we obtain the desired assertion.
\end{proof}

We are now ready to prove Proposition~\ref{cor:3}.

\begin{proof}[Proof of Proposition~\ref{cor:3}]
For the first inequality we use that $\varphi^0[0\longleftrightarrow\partial \Lambda_n]\leq 1-\varphi^0[\mathcal{C}^*_{\lfloor n/2\rfloor}]$ and Lemma~\ref{dual-circuit}. For the second inequality, we observe that when $\Lambda_n$ is not connected to infinity, there is an open circuit in $\omega^*$ surrounding $\Lambda_n$. This implies that for some $k\geq n$, the dual edge $\{(k+1/2,1/2),(k+1/2,-1/2)\}$ is connected to distance $k$ in $\omega^*$, which in turn implies that the event $(k,0)+\mathcal{C}_{\lfloor k/2 \rfloor}$ does not happen. Hence the second inequality follows from Lemma~\ref{dual-circuit}.
\end{proof}

The following proposition, which gives polynomial bounds on one-arm events when there is continuous critical behaviour, is a standard consequence of Proposition \ref{prop:quad}. See \cite[Section 5]{HDC-IP}.
\begin{prop}\label{cor:a}
Assume that {\bf (ContCrit)} occurs. There exists $c>0$ such that for every $n\geq 1$,
\begin{equs}\label{prop: one arm}
\frac{c}{n}\leq \varphi^1[0\longleftrightarrow\partial\Lambda_n]\leq \frac1{n^c}.
\end{equs}
In particular, $\varphi^1[0\longleftrightarrow\infty]=0$ and $\varphi^1=\varphi^0$.
\end{prop}

The following proposition makes rigorous the intuitively clear statement that at $p_c$, only {\bf (ContCrit)} or {\bf (DiscontCrit)} can occur. We remark that the reverse is also true, namely {\bf (ContCrit)} and {\bf (DiscontCrit)} can only occur at $p_c$ and not any other value, but this requires to first establish sharpness of the phase transition, Theorem~\ref{thm: exp dec}, which is proved in Section~\ref{sec: subcrit}.

\begin{prop}\label{prop: pc behaviour}
Let $a \in (0,1)$. If $p=p_c(a)$, then either {\bf (ContCrit)} or {\bf (DiscontCrit)} occurs. Furthermore, the set of $a\in (0,1)$ for which {\bf (DiscontCrit)} occurs at $p_c(a)$ is open.
\end{prop}

\begin{proof}
The fact that at $p_c(a)$  {\bf (ContCrit)} or {\bf (DiscontCrit)} must occur follows from a finite size criterion, which implies that the set of points at which {\bf (SubCrit)} and {\bf (SupCrit)} hold are open. Indeed, let us show that {\bf (SubCrit)} does not hold at $p_c(a)$. Arguing as in \cite[Lemma 10]{DCT}, one can see that there exists $\delta>0$ such that if for some $p\in (0,1)$, $a\in (0,1)$ and some $k\geq 1$, $\varphi^1_{\Lambda_{4k},p,a}[\Lambda_k\longleftrightarrow \partial \Lambda_{2k}]<\delta$, then {\bf (SubCrit)} holds. If {\bf (SubCrit)} holds at $p_c(a)$, then $\varphi^1_{\Lambda_{4k},p_c(a),a}[\Lambda_k\longleftrightarrow \partial \Lambda_{2k}]<\delta$ for some $k$ large enough, and by continuity there exists $\varepsilon>0$ such that $\varphi^1_{\Lambda_{4k},p_c(a)+\varepsilon,a}[\Lambda_k\longleftrightarrow \partial \Lambda_{2k}]<\delta$. This would imply that {\bf (SubCrit)} holds at $(p_c(a)+\varepsilon,a)$, which is a contradiction. A similar consideration applied to dual crossings under the free measure shows that {\bf (SupCrit)} cannot happen at $p_c(a)$. 

In addition, a standard consequence of the renormalisation inequality in Lemma \ref{lem: renormalisation inequality for strip densities} is that there exists a $\delta'>0$ such that if $p_n<\delta'$ for some $n\geq 1$, then $p_n$ decays exponentially. When the latter happens we have that {\bf (DiscontCrit)} holds. Thus the set of points at which {\bf (DiscontCrit)} holds is open. 
\end{proof}

Finally, we establish that at $p_c(a)$, {\bf (ContCrit)} and {\bf (DiscontCrit)} correspond exactly to continuous and discontinuous critical points for the Blume-Capel model, respectively.
\begin{prop}\label{correspondence}
Let $\Delta\in \RR$ and $a=\frac{e^{\Delta}}{1+e^{\Delta}}$. If {\bf (ContCrit)} occurs at $p=p_c(a)$, then $\langle \sigma_0 \rangle^+_{\beta_c(\Delta),\Delta}=0$. If {\bf (DiscontCrit)} occurs at $p=p_c(a)$, then $\langle \sigma_0 \rangle^+_{\beta_c(\Delta),\Delta}>0$.
\end{prop}

\begin{proof}
Recall that $\langle \sigma_0 \rangle^+=\varphi^1[0\longleftrightarrow\infty]$ from the Edwards-Sokal coupling. If {\bf (ContCrit)} occurs, then $\varphi^1[0\longleftrightarrow\infty]=0$ by Proposition~\ref{cor:a}. If {\bf (DiscontCrit)} occurs, then $\varphi^1[0\longleftrightarrow\infty]>0$ by Proposition~\ref{cor:3}. This proves the desired result.
\end{proof}

\section{Tricritical point in $d=2$ via crossing probabilities}\label{sec: tric 2}

\subsection{A direct proof from the quadrichotomy}

One of the importance consequences of Proposition \ref{prop: pc behaviour} is that it allows to deduce that the two point function decays to $0$ at $p_c(a)$ for any $a \in (0,1)$. This, together with the mapping described in Section \ref{sec: mapping bc-ising}, implies Theorem \ref{thm: existence} for $\Delta^+ = -\log 2$. 

\begin{proof}[Proof of Theorem \ref{thm: existence} in $d=2$ for $\Delta^+ = -\log 2$]

Let $\Delta \geq - \log 2$ (so that the corresponding Ising model in the mapping of Corollary \ref{corollary: correlations} is ferromagnetic) and set $a=\frac{e^\Delta}{1+e^\Delta}$. By Proposition \ref{prop: pc behaviour}, we know that either {\bf (ContCrit)} or {\bf (DiscontCrit)} occurs at $p_c(a)$. Using Propositions \ref{cor:3} and \ref{cor:a}, we obtain that in any case, $\lim_{|x|\rightarrow \infty} \phi^0_{p_c(a),a}[0 \conn x] = 0$. As a consequence of Corollary \ref{corollary: correlations} we have that
\begin{equs}
\lim_{|x| \rightarrow \infty}\langle \tau^1_0 \tau^1_x \rangle^{{\rm Ising},0}_{J(\beta_c(\Delta), \Delta)}
=
0.
\end{equs}

Recall from the discussion in Section \ref{sec: pf in 3d} that the result of \cite{AR} implies \eqref{eq: raoufi}, namely that
\begin{align} \label{eq: raoufi2}
\langle \tau^1_x \tau^1_y \rangle^{{\rm Ising}, 0}_{J(\beta_c(\Delta),\Delta)}
=
\langle \tau^1_x \tau^1_y \rangle^{{\rm Ising}, +}_{J(\beta_c(\Delta),\Delta)}.
\end{align}
Hence, by Corollary \ref{corollary: correlations}, translation invariance and Griffiths' inequality for $\langle \cdot \rangle^{{\rm Ising},+}_{J(\beta_c(\Delta),\Delta)}$, and \eqref{eq: raoufi2}, for any $x \in \mathbb{Z}^d$
\begin{equs}
\Big(\langle \sigma_0 \rangle^+_{\beta_c(\Delta), \Delta}\Big)^2
=
\Big(\langle \tau^1_0 \rangle^{{\rm Ising},+}_{J(\beta_c(\Delta),\Delta)} \Big)^2
\leq
\langle \tau^1_0 \tau^1_x \rangle^{{\rm Ising},+}_{J(\beta_c(\Delta), \Delta)}
=
\langle \tau^1_0 \tau^1_x \rangle^{{\rm Ising},0}_{J(\beta_c(\Delta), \Delta)}.
\end{equs}
Thus, taking limits as $|x|\rightarrow \infty$, we obtain $\langle \sigma_0 \rangle^+_{\beta_c(\Delta), \Delta} = 0$, as desired. 
\end{proof}

\subsection{A sufficient criterion for continuity } \label{subsec: sufficient criterion continuity}

We give a sufficient criterion to prove continuity in dimension $2$ that avoids the use of the coupling with Ising explained in Section \ref{sec: mapping bc-ising}, thus having the potential to extend beyond $\Delta^+ = - \log 2$. In the next subsection, we show that this is indeed the case.

Let $G=(V,E)$ be a finite subgraph of $\ZZ^2$, and let $\varepsilon>0$.
We denote by $\mu_{G,\beta,\Delta}^{+,\varepsilon}$ the probability measure defined on spin configurations $\sigma \in \{-1,0,1\}^{V}$ by
\begin{equs}
\mu_{G,\beta,\Delta}^{+,\varepsilon}(\sigma)
=
\frac{1}{Z_{G,\beta,\Delta}^{+,\varepsilon}} e^{-H_{G,\beta,\Delta}^{+,\varepsilon}(\sigma)}	
\end{equs}
where
\begin{equs}
H_{G,\beta,\Delta}^{+,\varepsilon}(\sigma)
=
-\beta\sum_{xy \in E} \sigma_x\sigma_y -\Delta\sum_{x \in V} \sigma_x^2 - \varepsilon \sum_{\substack{xy \in \mathbb{E}^d \\x \in V, y \in \ZZ^d\setminus V }} \sigma_x.	
\end{equs}
and $Z_{\Lambda,\beta,\Delta}^{+,\varepsilon}$ is the partition function. This corresponds to the Blume-Capel model with $\varepsilon$-boundary conditions.
 
When the phase transition is continuous, $\mu_{G}^{+,\varepsilon}[\sigma_0]$ converges to $\mu^+[\sigma_0]=0$ for every $\varepsilon>0$. On the other hand, when the phase transition is discontinuous, it is unclear whether this property holds. We introduce the following intermediary property, denoted by ${\bf (H)}$, where we assume this holds conditional on $\sigma_0 \neq 0$. This conditioning appears technical and is used throughout the following subsection on the Lee-Yang theory. 

\begin{defn}
For $\beta > 0$ and $\Delta \in \R$, we say ${\bf (H)}$ is satisfied if and only if
\begin{equs}
\lim_{n\to\infty} \mu^{+,\varepsilon}_{\Lambda_n,\beta,\Delta} \left[\sigma_0 \mid \sigma_0\neq 0\right] 
=
\mu^+_{\beta,\Delta} \left[\sigma_0 \mid \sigma_0\neq 0\right],
\qquad \forall \epsilon > 0.
\end{equs}
\end{defn}

Given $\delta\in (0,1)$, we now define the dilute random cluster measure $\varphi^{1,\delta}_{\Lambda,p,a}$ corresponding to $\mu^{+,\varepsilon}_{\Lambda,\beta,\Delta}$ for $\varepsilon=-\frac{1}{2}\log(1-\delta)$ via the Edwards-Sokal coupling. The measure $\varphi^{1,\delta}_{\Lambda,p,a}$ is defined by 
\begin{equs}
\varphi^{1,\delta}_{\Lambda,p,a} [\theta] =
\frac{\1[\theta \in \Theta_\Lambda^1]}{Z^{1,\delta}_{\Lambda,p,a}}  2^{k(\theta,\Lambda)}
\prod_{x\in V} \left(\frac{a}{1-a}\right)^{\psi_x}
\prod_{e\in E_{\psi,\Lambda}} r_e\left(\frac{p_e}{1-p_e}\right)^{\omega_e},
\end{equs}
where $Z^{1,\delta}_{\Lambda,p,a}$ is the normalisation constant. Here $r_e=\sqrt{1-p}$ if both endpoints of $e$ belong to $\Lambda$, and $r_e=\sqrt{1-\delta}$ otherwise; whilst $p_e=p$ if both endpoints of $e$ belong to $\Lambda$, and $p_e=\delta$ otherwise. Note that samples from $\varphi^{1,\delta}_{\Lambda, p, a}$ are supported on $\Theta^1_{\Lambda}$, although we stress that not all vertices in $\Lambda_n$ are required to be open. 

In the next lemma, we obtain a comparison between $\varphi^{1,\delta}_{\Lambda_n,p,a}$ and $\varphi^0_{\Lambda_n,p,a}$ which states that, up to an exponential rewiring cost that is proportional to $n$ and $\delta$, we may bound probabilities under $\varphi^0_{\Lambda_n,p,a}$ from below by probabilities under $\varphi^{1,\delta}_{\Lambda_n,p,a}$. We recall that $b(\Lambda_n)$ is the set of edges induced by $\Lambda_n$. We write $\partial_E \Lambda = \{ xy \in \EE^d : x \in \Lambda, y \in \ZZ^d \setminus \Lambda \}$ to denote the edge boundary of $\Lambda$.

\begin{lem}\label{comparison}
For every $\delta\in (0,1)$, $n \in \NN$, and event $A$ depending on $(\Lambda_n, b(\Lambda_n))$, we have
\begin{equs}
\varphi^{1,\delta}_{\Lambda_n,p,a}[A]\leq \left(\frac{1}{1-\delta}\right)^{12n+6}\varphi^0_{\Lambda_n,p,a}[A].
\end{equs}
\end{lem}
\begin{proof}
Let $\theta\in \{0,1\}^{\Lambda_n}\times \{0,1\}^{b(\Lambda_n)}$ be a configuration on $\Lambda_n$. Given $\xi \in \Theta$ and $\theta' \in \Theta^\xi_{\Lambda_n}$, we write $\theta' \sim_{\Lambda_n} \theta$ if $\theta'|_{\Lambda_n \times b(\Lambda_n)} = \theta$, i.e.\ if the two configurations coincide exactly in $(\Lambda_n, b(\Lambda_n))$. Note that if $\theta',\theta'' \in \Theta^\xi_{\Lambda_n}$ are such that $\theta', \theta'' \sim_{\Lambda_n}\theta$, then they may only disagree on $\partial_E \Lambda_n$.

Let $Z^{1,\delta}_{\Lambda_n, p,a}[\theta]$ and $Z^{0}_{\Lambda_n, p, a}[\theta]$ be defined such that $\varphi^{1,\delta}_{\Lambda_n, p,a}[\theta] = \frac{Z^{1,\delta}_{\Lambda_n, p,a}[\theta]}{Z^{1,\delta}_{\Lambda_n, p,a}}$ and $\varphi^{0}_{\Lambda_n, p,a}[\theta] = \frac{Z^{0}_{\Lambda_n, p,a}[\theta]}{Z^{0}_{\Lambda_n, p,a}}$, respectively. Note that 
\begin{equs}
\varphi^{1,\delta}_{\Lambda_n, p, a}[\theta]
=
\frac{1}{Z^{1,\delta}_{\Lambda_n,p ,a}} \, \sum_{\substack{\theta^1 \in \Theta^1_{\Lambda_n} \\ \theta^1 \sim_{\Lambda_n} \theta }} Z^{1,\delta}_{\Lambda_n, p, a}[\theta^1] \quad \text{ and } \quad 
\varphi^{0}_{\Lambda_n, p, a}[\theta]
=
\frac{1}{Z^{0}_{\Lambda_n,p ,a}} \, \sum_{\substack{\theta^0 \in \Theta^0_{\Lambda_n} \\ \theta^0 \sim_{\Lambda_n} \theta } } Z^{0}_{\Lambda_n, p, a}[\theta^0],
\end{equs}
where we point out that in the second expression there is only one $\theta^0 \in \Theta^0_{\Lambda_n}$ such that $\theta^0 \sim_{\Lambda_n} \theta$ since the zero boundary conditions enforces closed edges. Thus, the proof is immediate if we can prove
\begin{equs} \label{eq: expwire claims}
\sum_{\substack{\theta^1 \in \Theta^1_{\Lambda_n} \\ \theta^1 \sim_{\Lambda_n} \theta }}Z^{1,\delta}_{\Lambda_n,p,a}[\theta^1]
\leq
2\left(\frac{1}{1-\delta}\right)^{8n+4} Z^{0}_{\Lambda_n,p,a}[\theta^0]
\quad \text{ and } \quad 
Z^{1,\delta}_{\Lambda_n,p,a}
\geq 
2(1-\delta)^{4n+2} Z^{0}_{\Lambda_n,p,a}.
\end{equs}

The first inequality of \eqref{eq: expwire claims} follows from three observations. First, note that for every $\theta^1\in \Theta^1_{\Lambda_n}$ and $\theta^0\in \Theta^0_{\Lambda_n}$ that coincide with $\theta$ on $\Lambda_n$, we have $k(\theta^1,\Lambda_n)\leq k(\theta^0,\Lambda_n)+1$, where the $+1$ term comes from the contribution of $\partial \Lambda_n$ to $k(\theta^1,\Lambda_n)$ in the extreme case when $\theta^1$ has all closed edges on $\partial_E \Lambda_n$. Second, note that $r_e\leq 1$ (so we can ignore any contribution coming from products of $r_e$). Finally, note that
\begin{equs}
\sum_{\omega \in \{0,1\}^{\partial_E \Lambda_n}}\prod_{e\in \partial_E \Lambda_n} \left(\frac{\delta}{1-\delta}\right)^{\omega_e}
=
\left(\frac{1}{1-\delta}\right)^{|\partial_E \Lambda_n|}
=
\left(\frac{1}{1-\delta}\right)^{8n+4}.
\end{equs}

For the second inequality of \eqref{eq: expwire claims}, we truncate the sum in $Z^{1,\delta}_{\Lambda_n,p,a}$ over all configurations $\theta^1$ whose edge projection is equal to $0$ for every edge in $\partial_E \Lambda_n$. Thus, for any such $\theta^1$, we have $k(\theta^1,\Lambda_n)= k(\theta^0,\Lambda_n)+1$. The estimate follows from this consideration together with the bound
\begin{equs}
\prod_{\substack{x\in \Lambda_n, y\not\in \Lambda_n\\ xy\in E_{\psi^1,\Lambda_n}}} r_{xy} 
\geq
(1-\delta)^{\frac{8n+4}2}
=
(1-\delta)^{4n+2}
\end{equs}
where $\psi^1$ is the projection of $\theta^1$ onto its first coordinate.
\end{proof}

We now show that {\bf (H)} is a sufficient condition for continuity.
\begin{prop} \label{prop: h implies continuity}
Let $\Delta \in \RR$. Assume that {\bf (H)} is satisfied at $(\beta_c(\Delta), \Delta)$. Then
\begin{equs}
\langle \sigma_0 \rangle^+_{\beta_c(\Delta), \Delta}
=
0.
\end{equs}
\end{prop}

\begin{proof}
Let $(\beta_c(\Delta), \Delta)$ be such that {\bf (H)} is satisfied and assume that  we have $\langle \sigma_0 \rangle^+_{\beta_c(\Delta),\Delta}>0$. By Proposition~\ref{correspondence}, {\bf (DiscontCrit)} occurs at $p=p_c(a)$ for $a=\frac{e^{\Delta}}{1+e^{\Delta}}$.
Let $t$ be the constant of Lemma~\ref{dual-circuit}. Then, by Lemma~\ref{comparison} and finite-energy, there exist $\delta,C>0$ such that
\begin{equs}
\varphi^{1,\delta}_{\Lambda_{2n}}[(\mathcal{C}^*_n)^c \mid \psi_0=1]
\leq 
C\varphi^{1,\delta}_{\Lambda_{2n}}[(\mathcal{C}^*_n)^c]
\leq 
Ce^{tn/2} \varphi^0_{\Lambda_{2n}}[(\mathcal{C}^*_n)^c] 
\leq 
Ce^{-tn/2}
\end{equs}
for every $n\geq 1$. 
By the Edwards-Sokal coupling,
\begin{equs}
\mu^{+,\varepsilon}_{\Lambda_{2n}}[\sigma_0 \mid \sigma_0\neq 0]=\varphi^{1,\delta}_{\Lambda_{2n}}[0\longleftrightarrow \Lambda_{2n}^c \mid \psi_0=1] \leq \varphi^{1,\delta}_{\Lambda_{2n}}[(\mathcal{C}^*_n)^c \mid \psi_0=1]\leq Ce^{-tn/2},
\end{equs}
where $\varepsilon=-\frac{1}{2}\log(1-\delta)$. 
Letting $n$ go to infinity and using Lemma~\ref{weak+limit} we obtain 
$\mu^{+}[\sigma_0 \mid \sigma_0\neq 0]=0$,
hence $\mu^+[\sigma_0]=0$, from which we obtain a contradiction.
\end{proof}

\subsection{ $\Delta^+ = -\log 4$ via Lee-Yang theory}

In order to prove Theorem \ref{thm: existence} in two dimensions with $\Delta^+=-\log{4}$, we need the following general statement of Lee-Yang type on the complex zeros of partition functions, which works for any finite graph. 
\newcommand{\Ising}{\operatorname{Ising}}
\newcommand{\BC}{\operatorname{BC}}

Let us first define the Blume-Capel model with arbitrary magnetic field. Given a finite graph $G=(V, E)$ and a \emph{real} magnetic field $h:V\rightarrow\mathbb{R}$, we denote by $\mu_{G,\beta,\Delta}^{h}$ the probability measure defined on spin configurations $\sigma \in \{-1,0,1\}^{V}$ by
\begin{equs}
\mu_{G,\beta,\Delta}^{h}(\sigma)
=
\frac{1}{Z_{G,\beta,\Delta}^{h}} e^{-H_{G,\beta,\Delta}^{h}(\sigma)}	\end{equs}
where
\begin{equs}
H_{G,\beta,\Delta}^{h}(\sigma)
=
-\beta\sum_{xy \in E} \sigma_x\sigma_y -\Delta\sum_{x \in V} \sigma_x^2 -  \sum_{x \in V} h_x\sigma_x.	
\end{equs}
and $Z_{\Lambda,\beta,\Delta}^{h}$ is the partition function. 

One can extend the definition of $\mu_{G,\beta,\Delta}^{h}$ to a \emph{complex} magnetic field $h:V\rightarrow\mathbb{C}$, which is in general a complex measure, as long as the partition function $Z_{G,\beta, \Delta}^h$ does not vanish. Our aim is to study the complex zeros of $Z_{G,\beta, \Delta}^h$ and, more generally, the complex zeros of partition functions of events. Given an event $\cE$, we define $Z^{h}_{G,\beta,\Delta}[\cE]$ as
\[
Z_{G,\beta,\Delta}^h[\cE]:=\sum_{\sigma\in \cE} \exp{\left\{\sum_{xy\in E} \beta\sigma_x\sigma_y +\Delta\sum_{x\in V} \sigma^2+\sum_{x\in V} h_x\sigma_x\right\}}.
\]
In the following lemma we study the zeros of the partition function $Z_G^h[\sigma_A\neq 0]$ with complex magnetic field $h$. Here, $A \subset V$ and $\sigma_A = \prod_{x\in A}\sigma_x$. The proof is based on an adaptation of \cite{dunlop1977zeros}.

\begin{lem}\label{lem:Lee-Yang}
Let $\beta>0$ and $\Delta\geq -\log4$. Consider a finite graph $G=(V, E)$ and let $A\subset V$. Let $h:V\rightarrow\mathbb{C}$ be a magnetic field such that for each $x\in V$, we have $\Re{(h_x)}\geq |\Im{(h_x)}|$. Then $Z^{h}_{G,\beta,\Delta}[\sigma_A\neq 0]\neq 0$.
\end{lem}
\begin{proof}
Let $\cE=\{\sigma_A\neq 0\}$. We first express $\left|Z_{G,\beta,\Delta}^h[\cE]\right|^2$ as a sum over pairs of configurations:

\begin{equs}
\left|Z_{G,\beta,\Delta}^h[\cE]\right|^2=\sum_{\sigma, \sigma'\in \cE}\exp\left\{\beta \sum_{xy\in E}(\sigma_x\sigma_y + \sigma'_x\sigma'_y) + \Delta \sum_{x\in V} (\sigma_x^2+{\sigma'_x}^2) +\sum_{x\in V} (h_x\sigma_x+\overline{h_x}\sigma'_x)\right\}.
\end{equs}
We now express our duplicated system as follows.
For each vertex $x$ we define $\nu_x:=e^{-i\pi/4}\cdot(\sigma_x+i\sigma'_x)\in \mathbb{C}$. We then have 
\[
\sigma_x\sigma_y+\sigma'_x\sigma'_y = \frac{ (\sigma_x+i\sigma'_x)(\sigma_y-i\sigma'_y)+(\sigma_x-i\sigma'_x)(\sigma_y+i\sigma'_y)}{2}= \frac{\nu_x\overline{\nu_y}+\overline{\nu_x}\nu_y}{2},
\]
\[\sigma_x^2+{\sigma'_x}^2 = \left| \nu_x \right|^2,\]
and 
\begin{align*}
h_x\sigma_x+\overline{h_x}\sigma'_x &= \Re(h_x)(\sigma_x+\sigma'_x)+i\Im(h_x)(\sigma_x-\sigma'_x)= 
\Re(h_x)\frac{\nu_x+\overline{\nu_x}}{\sqrt{2}}+\Im(h_x)\frac{\overline{\nu_x}-\nu_x}{\sqrt{2}}
\\& =
\nu_x\frac{\Re(h_x)-\Im(h_x)}{\sqrt{2}}+\overline{\nu_x}\frac{\Re(h_x)+\Im(h_x)}{\sqrt{2}}.
\end{align*}

Putting this together we obtain
\begin{align*}
    \left|Z_{G,\beta,\Delta}^h[\cE]\right|^2 &= \sum_{\nu\in S^{V\setminus A}_{\textup{BC}}\times S_{\ising}^A}   \left(\prod_{xy\in E} \exp\left\{\beta \nu_x\overline{\nu_y}/2\right\}\cdot\exp\left\{\beta \overline{\nu_x}\nu_y/2\right\}
    \right)
    \cdot
    \left(\prod_{x\in V} e^{\Delta |\nu_x|^2}\right)
    \\&
    \cdot\left(\prod_{x\in V} \exp\left\{\nu_x\frac{\Re(h_x)-\Im(h_x)}{\sqrt{2}}\right\}\cdot\exp\left\{\overline{\nu_x}\frac{\Re(h_x)+\Im(h_x)}{\sqrt{2}}\right\} \right),
\end{align*}
where the sets $S_{\Ising}$ and $S_{\BC}$ are given by 
\[
S_{\Ising}:=\sqrt{2}\cdot \{1, i, -1, -i\},
\qquad
S_{\BC}:=\tfrac{1}{\sqrt{2}}\cdot \{0, 2, 2i, -2, -2i, 1+i, -1+i, -1-i, 1-i\}.
\]
Now, since both $\Re{(h_x)}\pm \Im{(h_x)}$ are non-negative, when we expand all exponentials, except for $e^{\Delta|\nu_x|^2}$, as $e^s=\sum s^k/k!$, and exchange summations, we get 
\begin{equation}\label{eq:partition-squared-expanded}
    \left|Z_{G,\beta,\Delta}^h[\cE]\right|^2 = \sum_{n, n':V\rightarrow \mathbb{Z}_{\geq 0}} f(n, n')\cdot \prod_{x\in V} \left(\sum_{\nu_x} \nu_x^{n_x}\overline{\nu_x}^{n'_x} e^{\Delta |\nu_x|^2}\right),
\end{equation}
where $\nu_x$ ranges over $S_{\Ising}$ for $x\in A$ and over $S_{\BC}$ otherwise, and all factors $f(n, n')$ are non-negative with $f(\mathbf{0}, \mathbf{0})=1$. We have that for $x\in A$, the sum is equal to 
\begin{equs}
\sum_{\nu_x} \nu_x^{n_x}\overline{\nu_x}^{n'_x} e^{\Delta |\nu_x|^2} =  4\cdot \sqrt{2}^{n_x+n'_x}\cdot e^{2\Delta}\1_{\{4|n_x-n'_x\}},
\end{equs}
whereas for $x\in V\setminus A$ we have 
\begin{equs}
\sum_{\nu_x} \nu_x^{n_x}\overline{\nu_x}^{n'_x} e^{\Delta |\nu_x|^2} 
&= 
(4e^{2\Delta} + 4e^{\Delta}+1)\1_{\{n_x=0,n'_x=0\}}
\\
& \qquad \qquad +  \1_{\{4|n_x-n'_x, (n_x,n_x')\neq (0,0)\}}\cdot \left[ 4\cdot \sqrt{2}^{n_x+n'_x}\cdot e^{2\Delta}+4e^\Delta\cdot (-1)^{\frac{n_x-n'_x}{4}}\right].    
\end{equs}

Since $n_x-n'_x\equiv 4\mod 8$ implies $n_x+n'_x\geq 4$, all these sums are strictly positive if $\Delta \geq -\log{4}$, except when $n_x+n'_x=4$ and $\Delta=-\log 4$, in which case the sum is equal to $0$. This means that all summands on the right hand side of \eqref{eq:partition-squared-expanded} are non-negative. Since $f(\mathbf{0},\mathbf{0})=1$, the contribution from $(n,n'):=(\mathbf{0},\mathbf{0})$ is strictly positive, hence $Z_{G,\beta,\Delta}^h[\cE]\not=0$.
\end{proof}

In what follows, we condition on $\sigma_0\neq 0$ and show that property $({\bf H})$ holds at $(\beta_c(\Delta),\Delta)$ for every $\Delta\geq -\log{4}$, using the above lemma. 
 
\begin{prop}\label{weak+limit}
Let $\beta>0$ and $\Delta\geq -\log{4}$. Then {\bf (H)} is satisfied, i.e.
\begin{equs}
\lim_{n\to\infty} \mu^{+,\varepsilon}_{\Lambda_n,\beta,\Delta} \left[\sigma_0 \mid \sigma_0\neq 0\right] 
=
\mu^+_{\beta,\Delta} \left[\sigma_0 \mid \sigma_0\neq 0\right], \qquad \forall \epsilon > 0.
\end{equs}
\end{prop}
\begin{proof}
Note that for every $\varepsilon\geq \beta$,  $\mu^{+,\varepsilon}_{\Lambda_n}\left[\cdot \mid \sigma_0\neq 0\right]$ stochastically dominates $\mu^{+}\left[\cdot \mid \sigma_0\neq 0\right]_{\Lambda_n}$. Here, we use the monotonicity of the conditional measure in $\varepsilon$ and $G$, which follows the GKS and FKG inequalities (derived for the conditional measures themselves), respectively. On the other hand, the restriction of $\mu^{+,\varepsilon}_{\Lambda_n}\left[\cdot \mid \sigma_0\neq 0\right]$ on $\Lambda_{n-1}$ is stochastically dominated by $\mu^+_{\Lambda_{n-1}}\left[\cdot \mid \sigma_0\neq 0\right]$.
This implies that
\begin{equs}
\lim_{n\to \infty} \mu^{+,\varepsilon}_{\Lambda_n,\beta,\Delta} \left[\sigma_0 \mid \sigma_0\neq 0\right]= \mu^+_{\beta,\Delta} \left[\sigma_0 \mid \sigma_0\neq 0\right]
\end{equs}
for every $\varepsilon\geq \beta$. Hence it suffices to show that $\mu^{+,\varepsilon}_{\Lambda_n,\beta,\Delta} \left[\sigma_0 \mid \sigma_0\neq 0\right] $ converges to an analytic function of $\varepsilon>0$.

Note that 
\begin{equs}\label{w-expansion}
\mu^{+,\varepsilon}_{\Lambda_n,\beta,\Delta} \left[\sigma_0 \mid \sigma_0\neq 0\right]
= 
\frac{Z^{+,\varepsilon}_{\Lambda_n,\beta,\Delta}[\sigma_0=1]-Z^{+,\varepsilon}_{\Lambda_n,\beta,\Delta}[\sigma=-1]}{Z^{+,\varepsilon}_{\Lambda_n,\beta,\Delta}[\sigma_0=1]+Z^{+,\varepsilon}_{\Lambda_n,\beta,\Delta}[\sigma=-1]}
=
\frac{1-W^{\varepsilon}_{\Lambda_n,\beta,\Delta}}{1+W^{\varepsilon}_{\Lambda_n,\beta,\Delta}},
\end{equs}
where 
\begin{equs}
W^{\varepsilon}_{\Lambda_n,\beta,\Delta}=\frac{Z^{+,\varepsilon}_{\Lambda_n,\beta,\Delta}[\sigma_0=-1]}{Z^{+,\varepsilon}_{\Lambda_n,\beta,\Delta}[\sigma_0=1]}.
\end{equs}
For $\varepsilon\geq \beta$, $W^{\varepsilon}_{\Lambda_n,\beta,\Delta}$ converges to a constant, i.e.
\begin{equs}\label{w-limit}
\lim_{n\to\infty} W^{\varepsilon}_{\Lambda_n,\beta,\Delta}= \frac{1-\mu^+_{\beta,\Delta} \left[\sigma_0 \mid \sigma_0\neq 0\right]}{1+\mu^+_{\beta,\Delta} \left[\sigma_0 \mid \sigma_0\neq 0\right]}.
\end{equs}

We wish to show that $W^{\varepsilon}_{\Lambda',\beta,\Delta}$ remains bounded for $\varepsilon$ in a complex neighbourhood of $(0,\infty)$. For $h,\eta$ with $\Re(h)> |\Im(h)|$ and $\Re(\eta)> |\Im(\eta)|$, let $\mathbf{h}:\Lambda_n\to \mathbb{C}$ be the function which is defined by
\begin{equs}
\mathbf{h}_x=\eta \1_{x=0}+\sum_{\substack{y\in \ZZ^2\setminus \Lambda_n \\ y\sim x}}h\1_{x\neq 0}.
\end{equs}
Then we can write
\begin{equs}
Z^{\mathbf{h}}_{\Lambda_n,\beta,\Delta}[\sigma_0\neq 0]=e^{\eta} Z^{+,h}_{\Lambda_n,\beta,\Delta}[\sigma_0=1]+e^{-\eta} Z^{+,h}_{\Lambda_n,\beta,\Delta}[\sigma_0=-1]
\end{equs}
where we recall the definitions of the partition functions of the righthand side from Subsection \ref{subsec: sufficient criterion continuity}.
Note that $Z^{\mathbf{h}}_{\Lambda_n,\beta,\Delta}[\sigma_0\neq 0]\neq 0$ and $Z^{\mathbf{h}}_{\Lambda_n,\beta,\Delta}[\sigma_0=1]\neq 0$ by Lemma~\ref{lem:Lee-Yang}, where the latter holds because $Z^{\mathbf{h}}_{\Lambda_n,\beta,\Delta}[\sigma_0=1]$ coincides with a partition function on $\Lambda_n\setminus\{0\}$ with complex magnetic field still satisfying the assumptions of the lemma. It follows that $W^{h}_{\Lambda_n,\beta,\Delta}\neq -e^{2\eta}$.
Since for any $z\in \mathbb{C}$ with $|z|> e^{2\pi}$ there is some $\eta$ with $\Re(\eta)> |\Im(\eta)|$ such that $z=-e^{2\eta}$, we can conclude that
$\left\vert W^{h}_{\Lambda_n,\beta,\Delta}\right\vert \leq e^{2\pi}$.
It follows from \eqref{w-limit} and Vitali's theorem that 
\begin{equs}
\lim_{n\to\infty} W^{h}_{\Lambda_n,\beta,\Delta}= \frac{1-\mu^+_{\beta,\Delta} \left[\sigma_0 \mid \sigma_0\neq 0\right]}{1+\mu^+_{\beta,\Delta} \left[\sigma_0 \mid \sigma_0\neq 0\right]}
\end{equs}
for every $h\in \CC$ with $\Re(h)> |\Im(h)|$. The desired assertion follows readily from \eqref{w-expansion} by taking the limit as $n$ tends to infinity.
\end{proof}

\begin{proof}[Proof of Theorem \ref{thm: existence} for $\Delta^+ = -\log 4$]
This is a direct corollary of Propositions \ref{prop: h implies continuity} and \ref{weak+limit}.
\end{proof}

\begin{rem}
We expect the overall strategy above to extend to higher dimensions. Given that the tricritical point of the mean-field Blume-Capel model is $\Delta=-\log{4}$, see \cite{EllisOttoTouchette}, we expect that for any $\Delta<-\log{4}$ there is a sequence of graphs whose Lee-Yang zeroes have an accumulation point on the positive real line.
\end{rem}

\section{Subcritical sharpness for dilute random cluster on $\ZZ^d$}\label{sec: subcrit}

The aim of this section is to obtain an OSSS inequality for the dilute random cluster measure, which we then use to prove Theorem~\ref{sharpness}.

\subsection{Weakly monotonic measures and the OSSS inequality}\label{sec: OSSS}

In general, the dilute random cluster measure is not monotonic, as was already observed in \cite[p.12]{GG}. In this section, we show that the dilute random cluster measure satisfies a weaker notion of monotonicity, once we restrict to certain types of boundary conditions. This allows us to prove an OSSS inequality for the  dilute random cluster measure.

Let $\Lambda\subset \ZZ^d$ be a finite set of vertices, and let $n=|\Lambda|+|E_{\Lambda}|$. We define $D=\Lambda\sqcup E_{\Lambda}$, and we call it a \textit{domain}. We also define $S_{\Lambda}$ to be the set of sequences $(d_1,d_2,\ldots,d_k)$, $k\leq n$, such that each element of $D$ appears at most once in $(d_1,d_2,\ldots,d_k)$, and for every edge $e$ appearing in $(d_1,d_2,\ldots,d_k)$, all the endpoints of $e$ lying in $\Lambda$ appear in $(d_1,d_2,\ldots,d_k)$, and they all precede $e$ (if $e\in \partial_E \Lambda$, then only one endpoint lies in $\Lambda$). We write $U_{\Lambda}$ for the corresponding set of unordered objects, i.e.\ $\{d_1,\ldots,d_k\}$ belongs to $U_{\Lambda}$ whenever each element of $D$ appears at most once, and for every edge $e\in \{d_1,\ldots,d_k\}$, all the endpoints of $e$ lying in $\Lambda$ belong to $\{d_1,\ldots,d_k\}$.

A \textit{decision tree} is a pair $T = (d_1, (\phi_t)_{2\leq t\leq n})$, where $d_1\in D$, and for each $t>1$, the function $\phi_t$
takes a pair
$((d_1,\ldots, d_{t-1}), \eta_{(d_1,\ldots,d_{t-1})})$ as an input, where $d_i\in D$ and $\eta\in \{0,1\}^D$, and returns an element $d_t \in D \setminus \{d_1,\ldots, d_{t-1}\}$. For a comprehensive introduction to decision trees, we refer the reader to \cite{OD}. We call a decision tree \textit{admissible} if $d_1\in \Lambda$ and for every $t>1$, if $(d_1,\ldots,d_{t-1})\in S_{\Lambda}$, then $(d_1,\ldots,d_t)\in S_{\Lambda}$. In other words, an admissible decision tree starts always from a vertex, and queries the state of an edge $e$ only if its endpoints lying in $\Lambda$ have been queried at previous steps. 

Our result applies beyond the dilute random cluster measure to measures satisfying the following property.
We call a measure $\mu$ on $\{0,1\}^D$ \textit{weakly monotonic} if for every $\{d_1,\ldots,d_k\}\in U_{\Lambda}$ and every $\eta^1,\eta^2\in \{0,1\}^{\{d_1,\ldots,d_k\}}$ such that $\eta^1\leq \eta^2$ pointwise,
\begin{equs}
    \mu[\eta_{d_0}=1 \mid \eta^1] \leq \mu[\eta_{d_0}=1 \mid \eta^2]
\end{equs}
for every $d_0\in D$.

For an $n$-tuple $d_{[n]}=(d_1,\dots,d_n)$ and $t\leq n$, write $d_{[t]}=(d_1,\dots,d_t)$ and $\eta_{d_{[t]}}=(\eta_{d_1},\dots,\eta_{d_t})$. 
Let $T$ be an admissible decision tree and let $f:\{0,1\}^D\mapsto \RR$. Given a pair $(d,\eta)$ produced by $T$, we define
\begin{equs}
\tau_f(\eta)=\tau_{f,T}(\eta):=\min\big\{ t\ge1:\forall \eta'\in\{0,1\}^D,\quad \eta'_{d_{[t]}}=\eta_{d_{[t]}}\Longrightarrow f(\eta)=f(\eta') \big\}.
\end{equs}
When clear from context we simply write $\tau$ for this stopping time.

We can now state the main technical result of this section. 
\begin{thm}[OSSS inequality for weakly monotonic measures]\label{OSSS}
Let $\Lambda\subset \ZZ^d$ be a finite set of vertices, and let $f:\{0,1\}^D\mapsto [0,1]$ be an increasing function. For any weakly monotonic measure $\mu$ on $\Lambda$ and any admissible decision tree $T$,
\begin{equs}
    \textup{Var}_\mu(f)\leq \sum_{d_0 \in D} \delta_{d_0}(f,T) \textup{Cov}_\mu(f, \eta_{d_0}),
\end{equs}
where $\delta_{d_0}(f,T):=\mu[\exists t\leq \tau_f(\eta): d_t=d_0]$ is the revealment (of $f$) for the decision tree $T$. 
\end{thm}
\begin{proof}
The proof is similar to that of Theorem 1.1 in \cite{DCRT}. We include the details for readers convenience, and we highlight the places where we use that $\mu$ is weakly-monotonic and that $T$ is admissible.

Let $(\U_i)_{i\geq 1}$ and $(\V_i)_{i\geq 1}$ be two independent sequences of iid uniform $[0,1]$ random variables. Write $\mathbb P$ for the coupling between these variables (and $\mathbb E$ for its expectation). 
We define inductively the triple $(\mathbf{d},\bf X,\tau)$ as follows: for $t\geq 1$,  
\begin{align*}
 \mathbf{d}_{t}&=\begin{cases}
   d_1 & \text{if $t=1$}\\
   \phi_t(\mathbf{d}_{[t-1]},\bf X_{\mathbf{d}_{[t-1]}})&\text{if $t>1$}\end{cases}\quad\text{and}\quad \bf X_{\mathbf{d}_t}=\begin{cases} 1&\text{ if }\U_{t}\geq\mu[\eta_{\mathbf{d}_{t}}=0\,|\,\eta_{\mathbf{d}_{[t-1]}}=\bf X_{\mathbf{d}_{[t-1]}}],\\
   0&\text{ otherwise,}\end{cases}
  \end{align*}
and $\tau:=\min\big\{t\ge1:\forall x\in\{0,1\}^D,x_{\mathbf{d}_{[t]}}=\bf X_{\mathbf{d}_{[t]}}\Rightarrow f(x)=f(\bf X)\big\}$. Moreover, for $u\in [0,1]^n$ we let $F_{\bf d}(u)=x$, where $x$ is defined inductively as follows: for $1\le t\le n$
\begin{equation}\label{eq:ccc}
x_{d_t}:=\begin{cases} 1&\text{ if }u_{t}\ge\mu[\eta_{d_{t}}=0\,|\,\eta_{e_{[t-1]}}=x_{d_{[t-1]}}],\\
   0&\text{ otherwise}.
   \end{cases}
\end{equation}
Finally, for $0\le t\le n$, define $\Y^t:=F_\mathbf{d}(\mathbf W^t),$ where
$$\mathbf W^t:= \mathbf W^t (\U,\V) =(\V_1,\dots,\V_t,\U_{t+1},\dots,\U_\tau,\V_{\tau+1},\dots,\V_n).$$

Now Lemma 2.1 in \cite{DCRT} (which applies to general measures without any assumption of monotonicity) implies that $\bf X$ has law $\mu$ and is $\U$-measurable; whilst $\Y^{n}$ has law $\mu$ and is independent of $\U$. Therefore, 
\begin{equs}
{\rm Var}_\mu(f)\le \tfrac12\mu\big[|f-\mu[f]|\big]=\tfrac12\bbE\Big[\big|\,\bbE[f(\bf X)|\U]-\bbE[f(\Y^{n})|\U]\,\big|\Big]\le\tfrac12\mathbb E\big[|f(\bf X)-f(\Y^{n})|\big].
\end{equs}
Note that $f(\Y^{0})=f(\bf X)$, hence
\begin{equs}
  {\rm Var}_\mu(f)\le \tfrac12\mathbb E\big[|f(\Y^{0})-f(\Y^{n})|\big].
\end{equs}
By telescoping, the right-hand side of the previous inequality is at most
\begin{align*}\sum_{t=1}^{n}\mathbb E \big[|f(\Y^t)-f(\Y^{t-1})|\big]&= \sum_{t=1}^{n}\mathbb E\Big[\,|f(\Y^t)-f(\Y^{t-1})|\,\cdot\1_{t\leq \tau}\Big]\\
&=\sum_{d_0\in D} \sum_{t=1}^{n}\mathbb E\Big[\mathbb E\big[|f(\Y^t)-f(\Y^{t-1})| ~\big|~\U_{[t-1]}\big]\,\1_{t\leq \tau,\mathbf{d}_t=d_0}\Big],
\end{align*}
where we used that $\Y^t=\Y^{t-1}$ for any $t>\tau$.

Our aim now is to bound $\mathbb E\big[\,|f(\Y^t)-f(\Y^{t-1})| ~\big|~\U_{[t-1]}\big]$. To this end, observe that on the event $\{t\le \tau,\mathbf{d}_t=d_0\}$, if $\Y^t _{d_0}= \Y^{t-1}_{d_0}$ then $\Y^t = \Y^{t-1}$. Combining this observation with the fact that $f$ is increasing we obtain
\begin{align}
|f(\Y^t)-f(\Y^{t-1})| &=(f(\Y^t)-f(\Y^{t-1}) ) \,(\Y^t_{d_0} - \Y^{t-1}_{d_0} )\nonumber\\
&=f(\Y^{t-1})\Y^{t-1}_{d_0}+  f(\Y^t)\Y^t_{d_0}- f(\Y^{t-1})\Y^t_{d_0}- f(\Y^t)\Y^{t-1}_{d_0}.\label{eq:tt}
\end{align}

To handle the latter expression, we proceed in steps. First, note that 
for any measurable function $g$ and $t\le n$,
\begin{equs}\label{eq:bbb} 
\mathbb E[g(\Y^t)|\U_{[t]}]=\mu[g(\eta)],
\end{equs}
which follows from \cite[Lemma 2.1]{DCRT}. Applying \eqref{eq:bbb} to $g(\eta)=f(\eta)\eta_{d_0}$ gives that 
\begin{equs} \label{eq:a1}
\bbE[f( \Y^{t-1}) \Y^{t-1}_{d_0} \,|\, \U_{[t-1]}]= \mu[f(\eta)\omega_{d_0}]=\bbE[f( \Y^t) \Y^t_{d_0} \,|\, \U_{[t]}]=\bbE[f( \Y^t) \Y^t_{d_0} \,|\, \U_{[t-1]}],
\end{equs}
where for the last equality we average on $\U_t$.

Since $\mu$ is weakly-monotonic and $T$ is an admissible decision tree, for fixed $\U_{[n]}$ and $i$, $\Y^i$ is an increasing function of $(\V_i)$. Moreover, since $f$ and $\mathbf W_{d_0}$ are increasing functions of $(\V_i)$, we deduce that $f(\Y^{t-1})$ and $\Y^t_{d_0}$ are increasing functions of $(\V_i)$. The FKG inequality applied to the iid random variables $(\V_i)$ gives
\begin{equs}
\bbE[f(\Y^{t-1})\Y^t_{d_0}|\U_{[n]}]&\ge \bbE[f(\Y^{t-1})|\U_{[n]}]\bbE[\Y^t_{d_0}|\U_{[n]}].
\end{equs}
Averaging over $\U_{[t-1]}$ we obtain
\begin{equs}\label{eq:a2}
\bbE[f(\Y^{t-1})\Y^t_{d_0}|\U_{[t-1]}]&\ge \bbE\big[\bbE[f(\Y^{t-1})|\U_{[n]} \big]\bbE \big[\Y^t_{d_0}|\U_{[n]}]\,\big|\,\U_{[t-1]}\big]\nonumber\\
&= \bbE \big[f(\Y^{t-1})|\U_{[t-1]}\big]\bbE\big[\Y^t_{d_0}|\U_{[t-1]}\big]=\mu[f(\eta)]\mu[\eta_{d_0}],
\end{equs}
where we used that $\bbE[\Y^t_{d_0}|\U_{[n]}]$ is $\U_{[t-1]}$-measurable and \eqref{eq:bbb}. 

Similarly, $f(\Y^t)$ and $\Y^{t-1}_{d_0}$ are increasing functions of $(\V_i)$ and arguing as above we have
\begin{align*}\bbE[f(\Y^t)\Y^{t-1}_{d_0}|\U_{[t]}]&\ge \bbE\big[\bbE[f(\Y^t)|\U_{[n]}]\bbE[\Y^{t-1}_{d_0}|\U_{[n]}]\,\big|\,\U_{[t]}\big]\\
&= \bbE[f(\Y^t)|\U_{[t]}]\bbE[\Y^{t-1}_{d_0}|\U_{[t]}]=\mu[f(\eta)]\bbE[\Y^{t-1}_{d_0}|\U_{[t]}].\end{align*}
Taking expectation with respect to $\bbE[\,\cdot\,|\U_{[t-1]}]$ we obtain
$$\bbE[f(\Y^t)\Y^{t-1}_{d_0}|\U_{[t-1]}] \ge \mu[f(\eta)]\bbE[\Y^{t-1}_{d_0}|\U_{[t-1]}]=\mu[f(\eta)]\mu[\eta_{d_0}],$$
where we used \eqref{eq:bbb} for the last equality.

This inequality together with \eqref{eq:a2}, \eqref{eq:a1} and \eqref{eq:tt} give that 
\begin{equation}\mathbb E\big[\,|f(\Y^t)-f(\Y^{t-1})| ~\big|~\U_{[t-1]}\big]~\le~ 2\mathrm{Cov}_\mu (f, \eta_{d_0} ).\label{eq:g}\end{equation}
Since $\sum_{t=1}^{n}\mathbb P[t\le \tau,\mathbf{d}_t=d_0]=\delta_{d_0}(f,T),$
the desired result follows.
\noindent

\end{proof}

The following lemma, which is analogous to \cite[Lemma 3.2]{DCRT}, can be viewed as a sharp threshold result for the event $\{0\longleftarrow \partial \Lambda_n\}$.

\begin{lem}\label{cov-bound}
Let $\Lambda\subset \ZZ^d$ be a finite set containing $0$. For every weakly monotonic measure $\mu$ on $\Lambda$ and every $n\geq 1$, we have 
\begin{equs}
     \sum_{d\in D} \textup{Cov}_\mu(\1_{0\longleftrightarrow \partial \Lambda_n},\eta_d)\geq \frac{n}{4dQ_n}\mu[0\longleftrightarrow\partial \Lambda_n]\left(1-\mu[0\longleftrightarrow\partial \Lambda_n]\right),
\end{equs}
where $Q_n:=\max_{x\in \Lambda_n}\sum_{k=0}^{n-1}\mu[x\longleftrightarrow\partial \Lambda_k(x)]$.
\end{lem}

\begin{proof}
For any $k\in \{1,2,\ldots,n\}$, we wish to construct an admissible decision tree $T=T(k)$ determining $\1_{0\longleftrightarrow\partial\Lambda_n}$ such that 
for every $u \in \Lambda_n$,
\begin{equs}\label{eq:deltau}
\delta_u(T)\leq \1_{u\in \partial \Lambda_k}+\1_{u\not\in \partial \Lambda_k} \sum_{\substack{x\in \Lambda_n \\ x\sim u}} \mu[x\longleftrightarrow\partial \Lambda_k]
\end{equs}
and for each edge $e=uv$,
\begin{equs} \label{eq:deltae}
\delta_e(T)\le \mu[u\longleftrightarrow \partial\Lambda_k]+\mu[v\longleftrightarrow \partial\Lambda_k].
\end{equs}    
Note that for $u\in \Lambda_n$,
\begin{equs}
\sum_{k=1}^n\mu[u\longleftrightarrow \partial\Lambda_k]
\leq \sum_{k=1}^n  \mu[u\longleftrightarrow \partial\Lambda_{|k-\|u\|_1 |}(u)] \leq 2Q_n
\end{equs}
from which the desired assertion follows by applying Theorem~\ref{OSSS} for each $T(k)$ with $f=\1_{0\longleftrightarrow\partial\Lambda_n}$ and then summing over $k\in\{1,2,\ldots,n\}$.

The decision tree $T$ is defined at follows. We first fix an ordering of the vertices and the edges. Then we explore the state of each vertex of $\partial \Lambda_k$, one at a time, according to the ordering, and once we have explored the state of all vertices in $\partial \Lambda_k$, we explore the state of the edges with both endpoints being open and both in $\partial \Lambda_k$. Let $F_0$ be the set of vertices explored thus far, namely $\partial \Lambda_k$, and $V_0$ the set of open vertices in $F_0$. Next, we explore the state of all the vertices $u\in \Lambda_n\setminus F_0$ adjacent to $F_0$, one at a time, according to the ordering, and if $u$ is open, then we explore the state of the edges connecting $u$ to $V_0$, according to the ordering. We let $F_1$ be the union of $F_0$ with the set of vertices $u\in \Lambda_n\setminus F_0$ adjacent to $V_0$, and $E_1$ be the set of open edges explored thus far. 

We then proceed inductively. Assuming we have defined $F_t$ and $E_t$ for some $t\geq 1$, we then consider the first vertex $u\in \Lambda_n\setminus F_t$ incident to $E_t$, and if $u$ is open, then we explore the state of the edges connecting $u$ to $E_t$, according to the ordering. We let $F_{t+1}=F_t \cup \{u\}$, and $E_{t+1}$ be the union of $E_t$ with the open edges between $u$ and $E_t$. As long as there is an open vertex incident to $E_{t+1}$, we keep exploring the cluster of $\partial \Lambda_k$. Otherwise, the exploration process stops, in which case the event $\{0\longleftrightarrow \partial \Lambda_n\}$ has been determined, and we can define our decision tree after that point arbitrarily.

It is clear from the construction that $T$ is an admissible decision tree and that \eqref{eq:deltae} and \eqref{eq:deltau} are satisfied. This concludes the proof.
\end{proof}

\subsection{A generalisation of the dilute random cluster}

In this section, we introduce the generalized dilute random cluster measure, which depends on three parameters $p,a,r$ and is defined as the standard dilute random cluster measure, except that we are not requiring $r=\sqrt{1-p}$. We show that for certain values of $r$, the generalized dilute random cluster measure is weakly monotonic, which together with Theorem~\ref{OSSS} will be used in the proof of subcritical sharpness for the two parameter dilute random cluster measure.

Given $a\in (0,1)$, $p\in (0,1)$, and $r>0$, we let $\varphi^{\xi}_{\Lambda,p,a,r}$ denote the measure on
$\Lambda$ with boundary condition $\xi$, that is,
\begin{equs}
\varphi^{\xi}_{\Lambda,p,a,r} [\theta] =
\frac{\1[\theta \in \Theta_\Lambda^\xi]}{Z^{\xi}_{\Lambda,p,a,r}} r^{|E_{\psi,\Lambda}|} 2^{k(\theta,\Lambda)}
\prod_{x\in V} \left(\frac{a}{1-a}\right)^{\psi_x}
\prod_{e\in E_{\psi,\Lambda}} \left(\frac{p}{1-p}\right)^{\omega_e},
\end{equs}
where $Z^{\xi}_{\Lambda,p,a,r}$ is the normalisation constant. 

\begin{prop}\label{weakly}
Let $a,p\in (0,1)$ and $r\geq \frac{2(1-p)}{2-p}$. Then the measure $\varphi^{\xi}_{\Lambda,p,a,r}$ is weakly monotonic.
\end{prop}
\begin{proof}
Let $\{d_1,\ldots,d_k\}\in U_{\Lambda}$ and let $\eta^1,\eta^2\in \{0,1\}^{\{d_1,\ldots,d_k\}}$ be such that $\eta^1\leq \eta^2$. We wish to show that the vertex marginal $\Psi_2:=\Psi^{\xi}_{\Lambda,p,a,r}[\cdot \mid \eta^2]$ stochastically dominates the vertex marginal $\Psi_1:=\Psi^{\xi}_{\Lambda,p,a,r}[\cdot \mid \eta^1]$. Then the desired assertion follows as in the proof of Proposition~\ref{free-closed}.

It suffices to verify \eqref{eq:stoch1} and \eqref{eq:stoch2}. To this end, let $\Lambda'=\Lambda\setminus \{d_1,d_2,\ldots,d_k\}$, and let $x\in \Lambda'$. Then, for every configuration $\psi\in \{0,1\}^{\Lambda'}$, 
\begin{equs}
\frac{\Psi_2[\psi^{\{x\}}]}{\Psi_2[\psi_{\{x\}}]}
=r^{N^2_x} \frac{a}{1-a}  \frac{Z^{\textup{RC},\xi\cup\eta^2}_{\Lambda,\psi^{\{x\}}}}
{Z^{\textup{RC},\xi\cup\eta^2}_{\Lambda,\psi_{\{x\}}}}
\quad \text{and} \quad 
\frac{\Psi_1[\psi^{\{x\}}]}{\Psi_1[\psi_{\{x\}}]}
=r^{N^1_x} \frac{a}{1-a} \frac{Z^{\textup{RC},\xi\cup\eta^1}_{\Lambda,\psi^{\{x\}}}}{Z^{\textup{RC},\xi\cup\eta^1}_{\Lambda,\psi_{\{x\}}}},
\end{equs}
where $N_x^i$ is the number of neighbours of $x$ in $V_{\psi}\cup(V_{\xi\cup\eta^i}\setminus \Lambda')$. Define $A(x)$ to be the event that all edges incident to $x$ are closed, and note that 
\begin{equs}
\frac{Z^{\textup{RC},\xi\cup\eta^2}_{\Lambda,\psi^{\{x\}}}}
{Z^{\textup{RC},\xi\cup\eta^2}_{\Lambda,\psi_{\{x\}}}}= \frac{1}{\phi^{\textup{RC},\xi\cup\eta^2}_{\Lambda,\psi^{\{x\}}}[A(x)]}
\quad \text{and} \quad 
\frac{Z^{\textup{RC},\xi\cup\eta^1}_{\Lambda,\psi^{\{x\}}}}{Z^{\textup{RC},\xi\cup\eta^1}_{\Lambda,\psi_{\{x\}}}}=\frac{1}{\phi^{\textup{RC},\xi\cup\eta^1}_{\Lambda,\psi^{\{x\}}}[A(x)]}.
\end{equs}

It remains to show that 
\begin{equs}\label{stoch-ineq}
\phi^{\textup{RC},\xi\cup\eta^2}_{\Lambda,\psi^{\{x\}}}[A(x)]
\leq
r^{N_x^2-N_x^1} \phi^{\textup{RC},\xi\cup\eta^1}_{\Lambda,\psi^{\{x\}}}[A(x)].
\end{equs}
For this, recall that by the finite energy property for the random cluster model \cite[Theorem 3.1]{GGBook}, 
\begin{equs}
\phi^{\textup{RC},\xi\cup\eta^2}_{\Lambda,\psi^{\{x\}}}[A(x)]
\leq
c_p^{N_x^2-N_x^1} \phi^{\textup{RC},\xi\cup\eta^1}_{\Lambda,\psi^{\{x\}}}[A(x)],
\end{equs}
where $c_p=\max\left\{1-p,1-\frac{p}{p+2(1-p)}\right\}=\frac{2(1-p)}{2-p}$. 
Since $r\geq \frac{2(1-p)}{2-p}$, it follows that \eqref{stoch-ineq} holds. Thus we have verified \eqref{eq:stoch1}.

We wish to verify \eqref{eq:stoch2} for $\Psi^1$. It is not hard to see that 
\begin{equs}
\frac{\Psi^1\left[\psi^{\{x,y\}}\right]\Psi^1\left[\psi_{\{x,y\}}\right]}{\Psi^1\left[\psi^{\{x\}}_{\{y\}}\right]\Psi^1\left[\psi^{\{y\}}_{\{x\}}\right]}
=
r^{\1\{xy\in \E^d\}}\frac{Z^{\textup{RC},\xi\cup\eta^1}_{\Lambda,\psi^{\{x,y\}}}}{Z^{\textup{RC},\xi\cup\eta^1}_{\Lambda,\psi^{\{x\}}_{\{y\}}}} 
\frac{Z^{\textup{RC},\xi\cup\eta^1}_{\Lambda,\psi_{\{x,y\}}}}{Z^{\textup{RC},\xi\cup\eta^1}_{\Lambda,\psi^{\{y\}}_{\{x\}}}}. 
\end{equs}
As above, we have
\begin{equs}
\frac{Z^{\textup{RC},\xi\cup\eta^1}_{\Lambda,\psi^{\{x,y\}}}}{Z^{\textup{RC},\xi\cup\eta^1}_{\Lambda,\psi^{\{x\}}_{\{y\}}}} 
\frac{Z^{\textup{RC},\xi\cup\eta^1}_{\Lambda,\psi_{\{x,y\}}}}{Z^{\textup{RC},\xi\cup\eta^1}_{\Lambda,\psi^{\{y\}}_{\{x\}}}}=\frac{\phi^{\textup{RC},\xi\cup\eta^1}_{\Lambda,\psi^{\{y\}}_{\{x\}}}[A(y)]}{\phi^{\textup{RC},\xi\cup\eta^1}_{\Lambda,\psi^{\{x,y\}}}[A(y)]}.
\end{equs}
Let $B(x,y)$ be the event that all edges $xu\in \E^d\setminus\{xy\}$ are closed. By the FKG inequality 
\begin{equs}
\phi^{\textup{RC},\xi\cup\eta^1}_{\Lambda,\psi^{\{x,y\}}}[A(y)]\leq \phi^{\textup{RC},\xi\cup\eta^1}_{\Lambda,\psi^{\{x,y\}}}[A(y) \mid B(x,y)].
\end{equs}
Using the finite energy property we see that by closing $xy$ we get
\begin{equs}
\phi^{\textup{RC},\xi\cup\eta^1}_{\Lambda,\psi^{\{x,y\}}}[A(y) \mid B(x,y)]\leq \left(\frac{2(1-p)}{2-p}\right)^{\1\{xy\in \E^d\}} \phi^{\textup{RC},\xi\cup\eta^1}_{\Lambda,\psi^{\{y\}}_{\{x\}}}[A(y)],
\end{equs}
which completes the proof.
\end{proof}

The following result is a standard calculation similar to the case of the classical random cluster model \cite[Theorem 3.12]{GGBook}.

\begin{prop}
Let $a,p\in (0,1)$ and $r>0$. Let also $\Lambda\subset \ZZ^d$ finite. For every event $A$ that is $\Lambda$-measurable, we have
\begin{equs}
\frac{d\varphi^{\xi}_{\Lambda,p,a,r}[A]}{dp}=\frac{1}{p(1-p)}\sum_{e\in E_{\Lambda}} \textup{Cov}(A,\eta_e) 
\end{equs}
and
\begin{equs}
\frac{d\varphi^{\xi}_{\Lambda,p,a,r}[A]}{da}=\frac{1}{a(1-a)}\sum_{x\in \Lambda} \textup{Cov}(A,\eta_x).
\end{equs}
\end{prop}

\subsection{Proof of subcritical sharpness}

The main result of this section is the following. 
\begin{thm}\label{sharpness}
Let $a\in (0,1)$. For every $p<p_c(a)$ there exists $c=c(p,a)>0$ such that for every $n\geq 0$,
\begin{equs}
\varphi^1_{\Lambda_n,p,a}[0\longleftrightarrow \partial \Lambda_n]\leq e^{-cn}.
\end{equs}
\end{thm}

Before proving this result, we use it to prove Theorem~\ref{thm: exp dec}. 

\begin{proof}[Proof of Theorem~\ref{thm: exp dec} assuming Theorem \ref{sharpness}]
For $x\in \ZZ^d$ with $\lVert x \rVert_{\infty}=n$, we have 
\begin{equs}
\langle \sigma_0 \sigma_x \rangle^+_{\beta,\Delta}\leq \varphi^1_{p,a}[0\longleftrightarrow \partial \Lambda_n] \leq  \varphi^1_{\Lambda_n,p,a}[0\longleftrightarrow \partial \Lambda_n],
\end{equs}
by the Edwards-Sokal coupling and \eqref{eq:MON}. The assertion follows from Theorem~\ref{sharpness}.
\end{proof}

We are now ready to prove Theorem~\ref{sharpness}.
\begin{proof}[Proof of Theorem~\ref{sharpness}]
Let $a_0\in (0,1)$ and consider $p_0<p_c(a_0)$. We prove the assertion for the pair $(p_0,a_0)$.

Define $c_0=\frac{p_0(1-a_0)}{(1-p_0)a_0}$. For $n\geq 1$, $p\in (0,1)$, $a(p)=\frac{p}{p+c_0(1-p)}$, and $r_0=\sqrt{1-p_0}$, define 
\begin{equs}
\mu_n:=\varphi_{\Lambda_{2n},p,a(p),r_0}^1,\quad \theta_k(p):=\mu_k[0\longleftrightarrow \partial\Lambda_k],\quad \text{and} \quad S_n:=\sum_{k=0}^{n-1}\theta_k.
\end{equs}
Note that there exists $\varepsilon>0$ such that $r_0\geq \frac{2(1-p)}{2-p}$ for every $p\geq p_0-\varepsilon$, hence $\mu_n$ is weakly monotonic for every $p\geq p_0-\varepsilon$ by Proposition~\ref{weakly}. 
Note also that $a(p_0)=a_0$ and 
\begin{equs}
a'(p)=\frac{a(p)(1-a(p))}{p(1-p)},  
\end{equs}
hence
\begin{equs}
 \theta_k'(p)=&\frac{1}{p(1-p)}\sum_{e\in E_{\Lambda}}\textup{Cov}(\1_{0\longleftrightarrow \partial \Lambda_k},\omega_e)+\frac{a'(p)}{a(p)(1-a(p))}\sum_{x\in \Lambda}\textup{Cov}(\1_{0\longleftrightarrow \partial \Lambda_k},\psi_x)\\ =& \frac{1}{p(1-p)}\sum_{d\in D}\textup{Cov}(\1_{0\longleftrightarrow \partial \Lambda_k},\eta_d)\geq  \sum_{d\in D}\textup{Cov}(\1_{0\longleftrightarrow \partial \Lambda_k},\eta_d).   
\end{equs}
Now, the comparison between boundary conditions together with the fact that, for $k \leq n/2$, $\Lambda_{2k}(x)\subset\Lambda_{2n}$, imply that for $x\in\Lambda_n$,
\begin{equs}
\sum_{k=1}^{n-1}\mu_n[x\longleftrightarrow\partial\Lambda_k(x)]\leq 2\sum_{k\leq n/2}\mu_n[x\longleftrightarrow\partial\Lambda_k(x)] \leq 2\sum_{k\leq n/2}\mu_k[0\longleftrightarrow\partial\Lambda_k]\leq 2S_n.
\end{equs}
By Lemma~\ref{cov-bound}, there exists a constant $t>0$ such that
\begin{equs}
 \theta_n'(p)\geq t\frac{n}{S_n}\theta_n(p)
\end{equs}
for $p\leq 1-\varepsilon$, where we used that $1-\theta_n(p)\geq 1-\theta_1(p)$, and the latter is bounded away from $0$. Using Lemma 3.1 in \cite{DCRT} for $f_n=\theta_n/t$, we obtain the existence of a critical point $p_1>0$ such that the following holds: for every $p\in (p_0-\varepsilon,p_1)$ there exists $c=c(p)>0$ such that
\begin{equs}
\theta_n(p)\leq e^{-cn} \text{ for every } n\geq 0,
\end{equs}
and for every $p\in (p_1,1-\varepsilon)$,
\begin{equs} 
 \varphi^1_{p,a(p),r_0}[0\longleftrightarrow\infty]\geq t(p-p_1).
\end{equs}

Note that $\mu_n[0\longleftrightarrow\partial \Lambda_{2n}]\leq \theta_n(p)$, and $\varphi_{\Lambda_{n},p,a(p),r_0}^1[0\longleftrightarrow\partial \Lambda_{n}]$ is decreasing in $n$, hence $\varphi_{\Lambda_{n},p,a(p),r_0}^1[0\longleftrightarrow\partial \Lambda_{n}]$ decays exponentially for every $p<p_1$. To conclude it suffices to show that $p_1>p_0$.
To this end, consider some $p_0<p'<p_c(a_0)$. We wish to show that $\varphi_{p',a_0}^1$ stochastically dominates $\varphi_{p'',a(p''),r_0}^1$ for some $p''>p_0$. This implies that either $p''$ is subcritical or critical along the $(p, a(p))$ line, i.e.\ under the trajectory of measures $\{ \varphi^1_{p,a(p),r_0} : p \in [0,1]\}$. Hence, by the strict inequality, we have that $p_0$ is subcritical along the $(p,a(p))$ line. This, by the above, means that $p_0 < p_1$, as desired. In order to obtain the existence of such a $p''$, we may argue as in the proof of Proposition~\ref{prop:stochastic-domination-2}. Let $\Lambda'$ denote the graph of a single edge and let $\xi$ be any boundary condition. For every non-empty increasing event $A$,  we have that
\begin{equs}
    \varphi^{\xi}_{\Lambda',p',a_0}[A]> \varphi^{\xi}_{\Lambda',p'',a(p''),r_0}[A]
\end{equs}
for some $p''>p_0$. 
since the strict monotonicity of the function $p\mapsto \varphi^{\xi}_{\Lambda',p,a_0}[A]$ implies that
\begin{equs}
    \varphi^{\xi}_{\Lambda',p',a_0}[A]> \varphi^{\xi}_{\Lambda',p_0,a_0}[A],
\end{equs}
and $\varphi^{\xi}_{\Lambda',p,a,r_0}[A]$ is continuous as a function of $a$ and $p$. The desired result is obtained by proceeding as in the proof of Proposition \ref{prop:stochastic-domination-2}.
\end{proof}

\section{Further analysis of critical behaviour in dimension $2$}\label{sec: further}

In this section, we prove Theorem~\ref{thm: truncated}. The statements {\bf (OffCrit)}, {\bf (Discont)}, {\bf (Cont)}, and {\bf (TriCrit)} are almost direct consequences of the quadrichotomy in Proposition \ref{prop:quad}. We therefore first develop the tools to establish {\bf (Perc)} before writing the proof. Let us introduce the following notions. We say that the $\{0,-\}$ spins percolate if with positive probability there is an infinite connected component $\mathscr{C}\subset \Z^2$ such that $\sigma_x\in \{0,-\}$ for every $x\in \mathscr{C}$. Define
\begin{equs}
\cL_c
=
\{ (\Delta,\beta_c(\Delta)) : \{0,-\} \text{ spins do not percolate under } \mu^0\}. 	
\end{equs}

\subsection{A geometric argument for continuity}

We start by showing the first direction of {\bf (Perc)}. Let us first recall some standard facts. Consider a finite subset $A$ of $\ZZ^2$, which induces a connected graph. Let $C$ be the set of vertices $x\in A$ for which there exists an adjacent vertex $y\in \ZZ^d\setminus A$ such that $y$ can be connected to infinity in $\ZZ^d\setminus A$. In general, $C$ is not connected in $\ZZ^2$, but it is a {\it *-connected circuit}. The latter means that $C$ can be represented as a walk $w=(w(0),w(1),\ldots, w(n))$ such that $\lVert w(i)-w(i+1)\rVert_{\infty}=1$, $w(0)=w(n)$, and otherwise $w(i)\neq w(j)$ for $i\neq j$. Given a *-connected circuit $C$ surrounding $0$, we write $C^{\rm int}$ for the union of $C$ with the connected component of $\ZZ^2\setminus C$ that contains $0$.

We now prove the following fact, which is a modification of a geometric argument for the planar Ising model, see \cite{W09}.
\begin{prop}\label{prop:0-}
For every $(\Delta, \beta_c(\Delta)) \in \cL_c$, $\langle \sigma_0 \rangle^+_{\beta_c(\Delta),\Delta}=0$.
\end{prop}

\begin{proof}
We fix some $(\Delta,\beta_c(\Delta))\in \cL_c$ and drop it from the notation. Let $\mathcal{A}_n$ be the event that there is no path $\Pi$ connecting $0$ to $\partial \Lambda_n$ such that $\sigma_x\in \{0,-\}$ for every vertex $x\in \Pi$. On this event, there is a *-connected circuit of $\{+\}$ spins surrounding $0$ which is contained in $\Lambda_n$. Let $\cC_n$ denote the outermost such circuit in $\Lambda_n$. Then, for every $n\geq 1$
\begin{equs}
\langle \sigma_0 \rangle^0
&=
\langle \sigma_0 \1_{\sigma \notin \mathcal{A}_n} \rangle^0 + \langle \sigma_0 \1_{\sigma \in \mathcal{A}_n} \rangle^0
=
\langle \sigma_0 \1_{\sigma \notin \mathcal{A}_n} \rangle^0  + \sum_C \langle \sigma_0 \1_{\cC_n=C} \rangle^0
\\
&=
\langle \sigma_0 \1_{\sigma \notin \mathcal{A}_n} \rangle^0  + \sum_C \langle \sigma_0 \rangle_{C^{\rm int}}^+ \mu^0[\cC_n=C]
\geq
\langle \sigma_0 \1_{\sigma \notin \mathcal{A}_n} \rangle^0  + \langle \sigma_0 \rangle^+ \mu^0[\mathcal{A}_n].
\end{equs}
In the second line we have partitioned the event $\mathcal{A}_n$ over all possibilities for $\cC_n$. In the third line we have used the spatial Markov property and the fact that the event $\{\cC_n=C\}$ is measurable with respect to the configuration on $C\cup (\ZZ^d\setminus C^{\rm int})$. In the final line we have used the monotonicity of the $+$ boundary conditions in the domain.

Taking limits as $n \rightarrow \infty$ and using that $\mu^0[\mathcal{A}_n]$ converges to $1$, we obtain that
$0 =\langle \sigma_0 \rangle^0	\geq\langle \sigma_0 \rangle^+$.
Since $\langle \sigma_0 \rangle^+\geq 0$, the desired assertion follows.
\end{proof}

\subsection{Infinite cluster and crossing probabilities}

Our aim now is to relate the existence of an infinite $\{0,-\}$ cluster with crossing probabilities, from which {\bf (Perc)} will follow. In order to do so, we first need to show that $\mu^0$ is mixing at the critical line. We remark that $\mu^0$ is not mixing when $\varphi^0[0 \longleftrightarrow \infty]>0$, since there is probability $1/2$ that there is an infinite cluster of $\{+\}$ spins, and probability $1/2$ that there is an infinite cluster of $\{-\}$ spins. For this reason, our proof utilises that $\varphi^0[0 \longleftrightarrow \infty]=0$ at the critical line (this holds thanks to Propositions \ref{cor:3} and \ref{cor:a}).

Let us introduce the following definition. For every $x\in \ZZ^2$, we define $\mathscr{C}_x=\{x\}\cup\{y\in \ZZ^2 :  x \text{ is connected to }  y \text{ by an open path in } \omega \}$ to be the cluster of $x$ in a configuration $(\psi,\omega)$. Note that $x\in \mathscr{C}_x$ even if $\psi_x=0$.

\begin{lem}\label{lem: varphi0 no inf cluster}
Let $\beta>0, \Delta\in \RR$ and let $p=1-e^{-2\beta}$, $a=\frac{e^{\Delta}}{1+e^{\Delta}}$. Assume that 
$\varphi^0_{p,a}[0 \longleftrightarrow \infty]=0$.
Then $\mu^0_{\beta,\Delta}$ is mixing, i.e.\ for all events $A,B$ 
\begin{equs}
   \lim_{\|x\|_\infty\rightarrow \infty} \mu^0_{\beta,\Delta}[A \cap \tau_x B]
   =
   \mu^0_{\beta,\Delta}[A] \mu^0_{\beta,\Delta}[B].
\end{equs}
\end{lem}

\begin{proof}
We fix $\beta,\Delta,p$ and $a$ and drop them from the notation. It suffices to show that $$\lim_{\|x\|_\infty \to\infty}\mu^0[A\cap\tau_x B]=\mu^0[A] \: \mu^0[B]$$
for all events $A, B$ depending on the spins $\sigma_x$, $x\in \Lambda_k$ for some $k\geq 1$. 

To this end, let us define $\mathscr{C}_k=\bigcup_{x\in \Lambda_k}\mathscr{C}_x$. Since $\varphi^0[0 \longleftrightarrow \infty]=0$, $\mathscr{C}_k$ is finite almost surely. Note that $\mathscr{C}_k$ spans a connected subgraph of $\ZZ^2$ (since $x \in \mathscr{C}_x$), hence almost surely there is dual circuit surrounding $\Lambda_k$.

Fix $\varepsilon>0$ and let $n>k$ be such that $\varphi^0[\mathcal{C}_{k,n}]\geq 1-\varepsilon$, where $\mathcal{C}_{k,n}$ denotes the event that there is a dual circuit surrounding $\Lambda_k$ which is contained in $\Lambda_n$. For every $x\in \Z^2$, we have 
\begin{equs}
\P^0[A \cap \cC_{k,n} \cap \tau_x B \cap \tau_x\cC_{k,n}]\leq \mu^0[A\cap \tau_x B]
\leq \P^0[A\cap \cC_{k,n} \cap \tau_x B\cap \tau_x\cC_{k,n}]+2\varepsilon.
\end{equs}
Above, $\P^0$ is the infinite volume coupling measure arising in the Edwards-Sokal coupling\footnote{We defined these measures in finite volume. However, one can define the infinite volume coupling (there is no ambiguity since there is no infinite cluster).}.
Consider some $\Vert x \Vert_{\infty}\geq 2n+1$, and a pair of configurations $\theta'=(\psi',\omega'),\theta''=(\psi'',\omega'') \in \Theta_{\Lambda_n}$, i.e.\ living on $\Lambda_n$, such that $\theta',\theta''\in\cC_{k,n}$. For ease of notation, we write $\theta'$ and $\tau_x \theta''$ to denote the events $\{ \theta|_{\Lambda_n} = \theta'\}$ and $\{ \theta|_{\tau_x \Lambda_n} = \tau_x \theta''\}$, respectively. Then
\begin{equs}
\P^0\left[A\cap \tau_x B \cap \theta'\cap\tau_x\theta''\right]
=
\varphi^0\left[\theta'\cap\tau_x\theta''\right]\P^0[A\mid\theta'] \, \P^0[B \mid \theta''].
\end{equs}
because conditioned on $\theta'$ and $\tau_x \theta''$, the spins on the vertices of $\Lambda_k$ and $\tau_x \Lambda_k$ are assigned independently, and the assignment is measurable with respect to $\theta'$ in the first case and with respect to $\theta''$ in the second case, due to the existence of the dual circuits.
By the ergodicity of $\varphi^0$, $\varphi^0[\theta'\cap\tau_x \theta'']$ converges to $\varphi^0[\theta'] \,\varphi^0[\theta'']$ as $\|x\|_\infty$ tends to infinity, hence summing over all possible $\theta$ and $\theta'$ we obtain 
\begin{equs}
\lim_{\|x\|_\infty \to\infty} \P^0[A\cap \cC_{k+1,n} \cap \tau_x B \cap \tau_x \cC_{k+1,n}]= \P^0[A\cap \cC_{k+1,n}] \, \P^0[B\cap \cC_{k+1,n}].
\end{equs}
Sending $n$ to infinity and $\varepsilon$ to $0$, we obtain that $\P^0[A\cap \cC_{k+1,n}]$ converges to $\mu^0[A]$, and $\P^0[B\cap \cC_{k+1,n}]$ converges to $\mu^0[B]$. The desired assertion follows readily.
\end{proof}

As a corollary of the above lemma and Propositions~\ref{cor:3} and \ref{cor:a} we obtain the following.

\begin{cor}\label{free mixing}
Let $\Delta\in \RR$. Then $\mu^0_{\beta_c(\Delta),\Delta}$ is mixing.
\end{cor}

Let $\mathcal{N}_{0,-}$ be the number of infinite clusters of $\{0,-\}$ spins. Since $\mu^0$ satisfies the finite energy property and the FKG inequality, a Burton-Keane argument implies the following.

\begin{lem}\label{Burton-Keane}
Let $\Delta\in \RR$ and $\beta>0$. Assume that $\mu^0_{\beta,\Delta}$ is ergodic. Then either $\mu^0_{\beta,\Delta}[\mathcal{N}_{0,-}=0]=1$ or $\mu^0_{\beta,\Delta}[\mathcal{N}_{0,-}=1]=1$.
\end{lem}

Let $H^{0,-}_n$ be the event that there is a horizontal crossing $\gamma$ in $\Lambda_n$ such that $\sigma_x\in \{0,-\}$ for every vertex $x\in \gamma$. Define $V^{0,-}_n$, $H^+_n$ and $V^+_n$ analogously.
The following result tells us that when $\mu^0_{\beta,\Delta}$ is ergodic and there is a unique infinite cluster of $\{0,-\}$ spins, crossing probabilities go to $1$. It is an adaptation of Zhang's argument and its proof is given in \cite[Proposition 2.1]{HDC-IP}.

\begin{lem}\label{cross_1}
Let $\Delta\in \RR$ and $\beta>0$. Assume that  $\mu^0_{\beta,\Delta}[\mathcal{N}_{0,-}= 1]=1$. Then 
\begin{equs}
\lim_{n\to\infty}\mu_{\beta,\Delta}^0[H^{0,-}_n]=1.
\end{equs}
\end{lem}

\subsection{Proof of Theorem~\ref{thm: truncated}}

\begin{proof}[Proof of Theorem~\ref{thm: truncated}]
Since $\beta,\Delta$ are fixed, we drop them from the notation. Note that x is a standard consequence of the renormalisation inequalities of Lemma \ref{lem: renormalisation inequality for strip densities}. We first establish {\bf (OffCrit)} and {\bf (Discont)}. For $\beta<\beta_c(\Delta)$, we have $\langle \sigma_0 ; \sigma_x\rangle^+=\langle \sigma_0  \sigma_x\rangle^+$. In this case, the assertion follows from Theorem~\ref{thm: exp dec}. 

Let $\beta\geq \beta_c(\Delta)$ and $\Delta\in \RR$ be such that $\langle \sigma_0 \rangle^+>0$. Recall that $\mathcal{C}_n$ is the event that there is an open circuit in $\omega\cap(\Lambda_{2n}\setminus \Lambda_n)$. Also recall that Propositions~\ref{cor:2}, \ref{cor:3} and Propositions~\ref{dual-circuit}, \ref{sup-circuit} state that there exists $c>0$ such that for every $n\geq 1$, 
 \begin{equs}\label{eq:ineq}
 \varphi^1[\Lambda_n\centernot\longleftrightarrow\infty]\leq e^{-cn} \quad \text{and} \quad \varphi^1[\mathcal{C}_n]\geq 1- e^{-cn}.
 \end{equs}

On the one hand, for every $\lVert x \rVert_{\infty}>4n$ we have
 \begin{equs}
 \langle \sigma_0 \sigma_x \rangle^+ = \varphi^1[0\longleftrightarrow x]\leq \left(\varphi^1_{\Lambda_{2n}}[0\longleftrightarrow \partial \Lambda_{2n}]\right)^2
 \end{equs}
 by the Edwards-Sokal coupling and \eqref{eq:MON}.
 On the other hand, when each of $\mathcal{C}_n$, $\{0\longleftrightarrow \mathcal{S}\}$,  and $\{\Lambda_n\longleftrightarrow\infty\}$ happens, where $\mathcal{S}$ denotes the exterior most circuit in $\omega\cap (\Lambda_{2n}\setminus \Lambda_{n})$, $0$ is connected to infinity. Hence,
 \begin{equs}
 \langle \sigma_0 \rangle^+=& \varphi^1[0\longleftrightarrow \infty]\geq \varphi^1[\mathcal{C}_n, 0\longleftrightarrow \mathcal{S},  \Lambda_n\longleftrightarrow\infty] \\
 \geq& (1-e^{-cn})\varphi^1[0\longleftrightarrow \mathcal{S} \mid \mathcal{C}_n]-e^{-cn} \\ \geq& (1-e^{-cn})\varphi^1_{\Lambda_{2n}}[0\longleftrightarrow \partial \Lambda_{2n}]-e^{-cn}\\
 \geq& \varphi^1_{\Lambda_{2n}}[0\longleftrightarrow \partial \Lambda_{2n}] -2e^{-cn}
 \end{equs}
 by \eqref{eq:ineq}, \eqref{eq:DMPRC} and \eqref{eq:MON}, and the latter is positive for every $n$ large enough. 
 Thus 
 \begin{equs}
 \langle \sigma_0 ; \sigma_x\rangle^+ &\leq \left(\varphi^1_{\Lambda_{2n}}[0\longleftrightarrow \partial \Lambda_{2n}]\right)^2-\left(\varphi^1_{\Lambda_{2n}}[0\longleftrightarrow \partial \Lambda_{2n}]-2e^{-cn}\right)^2
 \\ &=
 4e^{-cn}\varphi^1_{\Lambda_{2n}}[0\longleftrightarrow \partial \Lambda_{2n}]-4e^{-2cn} \leq 
 4e^{-cn}.
 \end{equs}
 This proves the desired assertion for $\beta\geq \beta_c(\Delta)$ such that $\langle \sigma_0 \rangle^+>0$.
 
We now establish {\bf (Cont)}. Let $\beta=\beta_c(\Delta)$ and assume that $\langle \sigma_0 \rangle^+=0$. Recall that $\langle  \sigma_0 \sigma_x \rangle^+=\varphi^1[0\longleftrightarrow x]$, hence it suffices to bound the latter. For the upper bound, we use Proposition~\ref{prop: one arm}. For the lower bound, if $2n-1\leq \lVert x \rVert_{\infty}\leq 2n$, then one way for $0$ to be connected to $x$ is if both $\mathcal{C}_n$ and $x+\mathcal{C}_n$ happen, $0$ is connected to $\partial \Lambda_{2n}$, and $x$ is connected to $x+\partial \Lambda_{2n}$. Then the assertion follows from Propositions~\ref{prop:quad} and \ref{prop: one arm} and the FKG inequality.

Finally, we prove {\bf (Perc)}. The reverse implication follows from Proposition~\ref{prop:0-}. For the forward implication, let us assume that $\langle \sigma_0 \rangle^+=0$. Then {\bf (ContCrit)} for the dilute random cluster model (i.e.\ in the sense of Proposition \ref{prop:quad}) occurs at $p=p_c(a)$ by Proposition~\ref{correspondence}.
It follows from the monotonicity in the domain and the $\pi/2$ rotational symmetry that for every $n\geq 1$, $\varphi^0[V_n]\geq c$.
Since we have probability $1/2$ to colour a vertical crossing with $\{+\}$ spins, it follows that for every $n\geq 1$,
$\mu^0[V^+_n]\geq c/2$.
Note that when $V^+_n$ happens, the event $H^{0,-}_n$ does not happen, hence
\begin{equs}
\mu^0[H^{0,-}_n]\leq 1-\mu^0[V^+_n]\leq 1- c/2.
\end{equs}
It follows from Corollary~\ref{free mixing} and Lemmas~\ref{Burton-Keane} and \ref{cross_1} that the $\{0,-\}$ spins do not percolate under $\mu^0$, as desired.
\end{proof}

\bibliographystyle{alpha}
\bibliography{surveybis.bib}

\end{document}